\documentclass[]{amsart}       
\usepackage{amsfonts,amssymb,amsmath, amsthm}
\usepackage{color}
\usepackage[margin=1.5in]{geometry}

\theoremstyle{plain}
\newtheorem{theorem}{Theorem}[section]
\newtheorem{corollary}[theorem]{Corollary}
\newtheorem{lemma}[theorem]{Lemma}
\newtheorem{proposition}[theorem]{Proposition}
\newtheorem{preremark}[theorem]{Remark}
\newenvironment{remark}{\begin{preremark}\normalfont}{\end{preremark}}

\theoremstyle{definition}
\newtheorem{assumption}{Assumption}
\newtheorem{definition}[theorem]{Definition}

\allowdisplaybreaks

\definecolor{myblue}{cmyk}{1,0.7,0,0}
\definecolor{mylightblue}{cmyk}{0.5,0.3,0,0}
\definecolor{mypaleblue}{cmyk}{0.08,0.05,0,0}
\definecolor{mypurple}{cmyk}{0.6,1,0,0}
\definecolor{myorange}{cmyk}{0,0.83,1,0.5}
\definecolor{mygreen}{cmyk}{0.9,0,1,0.2}
\definecolor{myred}{cmyk}{0,1,1,0.1}
\definecolor{mycolor}{cmyk}{0,0.5,1,0.5}
\definecolor{mymint}{cmyk}{0.93,0,0.75,0}

\renewcommand{\l}{\mathcal L}
\renewcommand{\r}{\mathcal R}
\newcommand{\e}{\mathcal E}
\newcommand{\F}{\mathcal F}

\newcommand{\R}{\mathbb R}

\title[Parabolic Harnack inequality on fractal-type Dirichlet spaces]{Parabolic Harnack inequality on fractal-type metric measure Dirichlet spaces}
\author{Janna Lierl}

\begin{document}

\begin{abstract}
This paper proves the strong parabolic Harnack inequality for local weak solutions to the heat equation associated with time-dependent (nonsymmetric) bilinear forms. The underlying metric measure Dirichlet space is assumed to satisfy the volume doubling condition, the strong Poincar\'e inequality, 
and a cutoff Sobolev inequality. The metric is not required to be geodesic.
Further results include a weighted Poincar\'e inequality, as well as upper and lower bounds for non-symmetric heat kernels.
\end{abstract}

\maketitle



\noindent
\emph{Author affiliation and address:} Department of Mathematics, University of Connecticut, 341 Mansfield Road, Storrs, CT 06250.

\tableofcontents

\section{Introduction}
Parabolic Harnack inequalities are relevant in studying regularity of solutions to the heat equation, and to obtain heat kernel estimates.
On some metric measure spaces, sharp two-sided bounds of (sub-)Gaussian type for the transition density of a diffusion process can be characterized by the parabolic Harnack inequality.
Moreover, parabolic Harnack inequalities can be characterized by geometric conditions, namely the volume doubling property and the Poincar\'e inequality. 
This equivalence was first proved on complete 
Riemannian manifolds by Saloff-Coste \cite{SC92, SC95} and Grigor'yan \cite{Gri91}. Sturm \cite{SturmIII} extended this result to metric measure Dirichlet spaces. Biroli and Mosco \cite{BM95} proved the elliptic Harnack inequality on Dirichlet spaces.

It is desirable to obtain similar results under minimal assumptions on the metric of the underlying Dirichlet space.
Interesting and comprehensive results in this direction have been obtained in recent years. See, e.g., \cite{HSC01, BBK06, GT12, BGK12, GHL15} and references therein for results in the context of fractal-type Dirichlet spaces.
The main focus of these works is on bounds for symmetric heat kernels. Harnack inequalities are used to obtain or characterize these estimates. For this purpose, one may replace the parabolic Harnack inequality by the 
elliptic Harnack inequality together with some additional conditions, e.g., resistance estimate, or exit time estimate. 

In this paper, we present three main results. The first is the strong parabolic Harnack inequality on any metric measure Dirichlet space that satisfies volume doubling, strong Poincar\'e inequality, and the cutoff Sobolev inequality on annuli. 
We emphasize that we do not require the metric to be geodesic, though if the metric is geodesic then we also have the converse implication, namely that the parabolic Harnack inequality implies the strong Poincar\'e inequality. See Proposition \ref{prop:VD,PI,CSA equiv}. 

More specifically, we show that the strong parabolic Harnack inequality
\[ \sup_{Q^-} u \leq C \inf_{Q^+} u \]
holds for any non-negative local weak solution $u(t,x)$ of the heat equation on a time-space cylinder
$Q(x,a,r) := (a,a + \Psi(r)) \times B(x,r)$, where $Q^- := (a + \tau_1 \Psi(r), a + \tau_2 \Psi(r)) \times B(x,\delta r)$ and 
$Q^+ := (a + \tau_3 \Psi(r), a + \tau_4 \Psi(r)) \times B(x,\delta r)$ are two smaller time-space cylinder of radius $\delta r < r$ that are separated by a 
time gap $(a + \tau_3 \Psi(r)) - (a + \tau_2 \Psi(r))$. Here $a$ is any real number, $x \in X$ is any point in the underlying metric measure space, and $C$ is a positive constant depending on 
the arbitrary choice of parameters $0 < \tau_1 < \tau_2 < \tau_3 < \tau_4 \leq 1$. The function $\Psi$ describes the appropriate time-space scaling that is implicit in the assumed 
Poincar\'e inequality PI($\Psi$) and the cutoff Sobolev inequality CSA($\Psi$) whose definitions we recall in the main text. Our only condition on $\Psi$ is that it satisfies a polynomial growth condition 
\eqref{eq:beta} given in Section \ref{ssec:cutoff sobolev}.

In the absence of a geodesic metric, we must distinguish between the {\em strong} parabolic Harnack inequality as stated above, and the {\em weak} parabolic Harnack inequality (see \cite{BGK12}) in which the Harnack
constant exists for {\em some} parameters $0 < \tau_1 < \tau_2 < \tau_3 < \tau_4 \leq 1$ but not necessarily for any arbitrary choice of parameters. See \cite{BGK12, GHL15} for equivalence results
for the weak parabolic Harnack inequality on symmetric Dirichlet spaces.

The second main result concerns weak solutions of the heat equation associated with time-dependent and/or non-symmetric bilinear forms $(\e_t,\F)$, $t \in \R$. These bilinear forms generalize Dirichlet forms: they may lack the Markovian property, non-negative definiteness, or symmetry. We think of these forms as perturbations of a symmetric strongly local regular reference Dirichlet form $(\e^*,\F)$. Our hypothesis is that the bilinear forms $\e_t$ satisfy certain structural conditions (see Assumption \ref{as:0}) and quantitative conditions (Assumptions \ref{as:skew1}, \ref{as:skew2}). 
We establish the local boundedness of local weak solutions (Corollary \ref{cor:u loc bounded}) and the strong parabolic Harnack inequality for $\e_t$ (Theorem \ref{thm:PHI}) under natural geometric conditions on the reference Dirichlet space. 
The local boundedness and the H\"older continuity (Corollary \ref{cor:Hoelder}) of local weak solutions are well-known consequences of the parabolic Harnack inequality.
A priori, however, the local boundedness of weak solutions is not obvious. We derive it from mean value estimates which we prove using a Steklov average technique similar to that in \cite{LierlSC2}.

Third, we present upper and lower bounds for the nonsymmetric heat kernels or, in the time-dependent case, heat propagators associated with $\e_t$, $t \in \R$.
As in \cite{LierlSC2}, our assumptions on the non-symmetric perturbations cover plenty of examples on Euclidean space, Riemannian manifolds, or polytopal complexes. For instance, our results apply to uniformly elliptic second order differential operators with (time-dependend) bounded measurable coefficients.
Examples of non-symmetric bilinear forms on an abstract Dirichlet space are not immediate.
In Section \ref{sec:examples}, we construct a non-symmetric perturbation $\e$ of a symmetric strongly local regular Dirichlet form $(\e^*,\F)$ so that $\e$ satisfies the strong parabolic Harnack inequality and heat kernel estimates.

Our setting includes fractal spaces like the Sierpinski carpet, though in this case the strong parabolic Harnack inequality is equivalent to the weak parabolic Harnack inequality because the metric is geodesic. Nevertheless, this case is interesting because we give a proof that does not rely on heat kernel estimates.

This work is in part motivated by applications to estimates for nonsymmetric Dirichlet heat kernels on inner uniform domains in fractal spaces \cite{LierlHKEf}. A common hypothesis in the works \cite{BB04, BBK06, AB15} which treat fractal-type spaces, is the conservativeness of the Dirichlet form. Since the estimates in \cite{LierlHKEf} are proved using Doob's transform and it is not clear a priori that the transformed Dirichlet space would be conservative, it was important to not assume conservativeness in the present work. We remark that the assumption of conservativeness was already dropped in, e.g., \cite{GHL15} in a similar context.

We prove our main results using the {\em parabolic} Moser iteration scheme \cite{Moser64, Moser67, Moser71}.
It was proved by Barlow and Bass in \cite{BB04,BBK06} that the {\em elliptic} Moser iteration scheme can be applied to obtain the elliptic Harnack inequality on a fractal-type
metric measure Dirichlet space which is symmetric strongly local regular and which satisfies the volume doubling property, the strong Poincar\'e inequality, and a cutoff Sobolev inequality. The parabolic Harnack inequality was then derived 
through an estimate for the resistance of balls in concentric larger balls.
The approach in \cite{BB04,BBK06} is to follow Moser's line of arguments with $d\mu$ replaced by a measure $d\gamma_{x,R} = \Psi(R) d\Gamma(\phi,\phi) + d\mu$, 
where $d\Gamma(\cdot,\cdot)$ is the energy measure of the Dirichlet form, and $\phi$ is a cutoff function for the ball $B(x,R/2)$ with compact support in the larger ball $B(x,R)$. 
This approach does not seem to generalize to the parabolic case: the estimates for sub- and supersolutions (cf.~Lemma \ref{lem:estimate subsol p>2} and \ref{lem:estimate subsol p<2}),
which are an important step in obtaining mean value estimates, are not available with $\gamma_{x,R}$ in place of $\mu$. 
Therefore, the parabolic case requires that the energy measure $d\Gamma(\psi,\psi)$ of a suitable cutoff function $\psi$ must be estimated through a cutoff Sobolev inequality very early in the line of arguments, that is, when 
proving sub- and supersolution estimates. This is possible thanks to the cutoff Sobolev inequality on annuli CSA($\Psi$) which was introduced in \cite{AB15}. 
The relevant property of this condition is that for {\em every} $\epsilon \in (0,1)$ there exists a cutoff function $\psi$ for $B(x,R)$ in $B(x,R+r)$ that satisfies the inequality
\begin{align} \label{csa}
 \int f^2 d\Gamma(\psi,\psi) \leq \epsilon \int \psi^2 d\Gamma(f,f) + C \frac{ \epsilon^{1-\beta_2/2}}{\Psi(r)} \int_{B(x,R+r)} f^2 d\mu,
\end{align}
for all $f \in \F$, where $C$ is a positive constant independent of $\psi,f,x,R,r,\epsilon$.

A slightly weaker condition is the {\em generalized capacity condition} introduced in \cite{GHL15}: It is inequality \eqref{csa} for bounded functions $f \in \F \cap L^{\infty}(X)$ 
and the cutoff functions $\psi$ are
allowed to depend on $f$. The generalized capacity condition appears to be too weak to run the parabolic Moser iteration. Indeed, since the local boundedness of weak
solutions is not known a priori, several approximation arguments are used in our proof. 
Because of this we need the cutoff functions to be independent of the functions that approximate the weak solution.

Once the mean value estimates for sub- and supersolutions are proved, we apply a weighted Poincar\'e inequality to complete the proof of 
the parabolic Harnack inequality. More specifically, we need the weight to be a cutoff function that satisfies CSA($\Psi$). The weighted Poincar\'e inequality is obtained in Theorem \ref{thm:weighted PI}.

It is worth pointing out that our arguments are local. Therefore, our hypotheses on the space (volume doubling and Poincar\'e inequality) are local. That is, they are stated for balls $B(x,R)$ that lie in
some subset $Y$ of the underlying space $X$, with radii $R$ up to a fixed scale $R \leq R_0 \in (0,\infty]$.

Regarding the notion of (local) weak solutions to the heat equation, we adopt the definition that is natural from the viewpoint of existence and uniqueness theory (see, e.g., \cite{LM68, Wlo87en, EldredgeSC14}).
In order to clarify the relation of recent literature to our results, we verify that the space of local weak solutions to the heat equation associated with a symmetric strongly local regular Dirichlet
 form constitutes a space  of caloric functions in the sense of \cite{BGK12}. 
Along the way, we obtain a proof of the parabolic maximum principle (Proposition \ref{prop:para max princ}) using the Steklov average technique.  We remark that the axiomatic properties of caloric functions implicitly presume the strong locality of the Dirichlet form.

In part of this paper, we will work with the so-called very weak solutions introduced in \cite{LierlSC2}. Very weak solutions may lack continuity in the time-variable and are thus too general to satisfy the parabolic Harnack inequality unless we additionally assume continuity in the time-variable, which then leaves us with weak solutions.

\ \\

\textbf{Structure of the paper.}
In Section \ref{sec:cutoff sobolev} we recall basic properties of the underlying metric measure Dirichlet space and introduce non-symmetric perturbations of the reference Dirichlet form $(\e^*,\F)$.
Since the assumptions we impose on the perturbations involve cutoff functions, we provide some background on cutoff Sobolev inequalities in the same section, and introduce a localized cutoff Sobolev condition.

In Section \ref{sec:Sobolev and Poincare} we consider Sobolev and Poincar\'e inequalities for the reference form. 
The main result of this section is the weighted Poincar\'e inequality of Theorem \ref{thm:weighted PI}.

In Section \ref{sec:Moser} we return to the setting of time-dependent non-symmetric local bilinear forms. We recall the definition of very weak solutions introduced in \cite[Definition 3.1]{LierlSC2} in Section \ref{ssec:very weak solutions}
 and then follow Moser's reasoning: We first prove estimates for non-negative local weak sub- and supersolutions 
(Section \ref{ssec:estim for sub and supsol}) and then run the parabolic Moser iteration scheme to obtain mean value estimates (Section \ref{ssec:Mean value estimates}). 
A main result of the paper, the local boundedness of weak solutions, hides in Corollary \ref{cor:u loc bounded}.

Section \ref{sec:PHI} is devoted to parabolic Harnack inequalities. 
Section \ref{ssec:PHI main results} contains main results, namely parabolic Harnack inequalities in the context of non-symmetric local bilinear forms.  
In Section \ref{ssec:PHI symmetric case} we take a closer look at the case of a symmetric strongly local regular Dirichlet form, relating the present paper to recent literature. 
This subsection relies on a parabolic maximum principle and a super-mean value property for local weak solutions. We prove these in Section \ref{sec:parabolic max princ}.

In Section \ref{sec:heat propagator} we present applications: estimates for symmetric and non-symmetric heat kernels and, in the time-dependent case, heat propagators. 
Some of these estimates are proved under the additional assumption that the metric is geodesic, and the bilinear forms satisfy a further quantitative condition (Assumption \ref{as:skew davies}).

We conclude the paper by constructing an example of a non-symmetric local bilinear form on a fractal-type metric measure space, see Section \ref{sec:examples}.

\textbf{Acknowledgement.}
The author thanks Laurent Saloff-Coste for discussions.

\section{Cutoff Sobolev type conditions and non-symmetric forms}
\label{sec:cutoff sobolev}
\subsection{The symmetric reference form}
\label{ssec:reference form}
Let $(X,d,\mu)$ be a locally compact separable metric measure space, where $\mu$ is a Radon measure on $X$ with full support. 
Throughout this paper we fix a symmetric strongly local regular Dirichlet form $(\e^*,\F)$ on $L^2(X,\mu)$. The Dirichlet form $(\e^*,\F)$ induces the norm 
\[ ||f||_{\F}^2 := \e^*(f,f) + \int f^2 d\mu \] 
on its domain $\F$. 
The energy measure $\Gamma$ of $\e^*$ (in \cite{FOT94} denoted as $\frac{1}{2}\mu^c_{<\cdot,\cdot>}$) satisfies a Cauchy-Schwarz inequality, cf.~\cite[Lemma 5.6.1]{FOT94},
\begin{align} \label{eq:CS}
\left| \int fg \, d\Gamma(u,v) \right|
\leq & \left(\int f^2 d\Gamma(u,u) \right)^{\frac{1}{2}} 
      \left(\int g^2 d\Gamma(v,v) \right)^{\frac{1}{2}},
\end{align}
for any $u,v \in \F$ and any bounded Borel measurable functions $f,g$ on $X$.
We have the following chain rule for $\Gamma$: For any $v, u_1, u_2, \ldots, u_m \in \F \cap L^{\infty}(X,\mu)$, $u = (u_1, \ldots, u_m)$, and $\Phi \in \mathcal C^1(\R^m)$ with $\Phi(0)=0$, we have $\Phi(u) \in \F \cap L^{\infty}(X,\mu)$ and
\begin{align} \label{eq:chain rule for Gamma}
d\Gamma(\Phi(u),v) = \sum_{i=1}^{m} \Phi_{x_i}(\tilde u) d\Gamma(u_i, v),
\end{align}
where $\Phi_{x_i}:=\partial \Phi / \partial x_i$ and $\tilde u$ is a quasi-continuous version of $u$, see \cite[(3.2.27) and Theorem 3.2.2]{FOT94}. 
When $\Phi_{x_i}$ is bounded for every $i \in \{1,\ldots,m\}$ in addition, then $\Phi(u) \in \F$ and \eqref{eq:chain rule for Gamma} hold for any $u_1, \ldots, u_m \in \F$ and any $v \in \F \cap L^{\infty}(X,\mu)$; see \cite[(3.2.28)]{FOT94}.

Inequality \eqref{eq:CS} together with a Leibniz rule \cite[Lemma 3.2.5]{FOT94} implies that
\begin{align} \label{eq:Gamma(fg)}
 \int d\Gamma(fg,fg)
& \leq  2\int f^2 d\Gamma(g,g)  
      + 2\int g^2 d\Gamma(f,f),
\end{align}
for any $f,g \in \F \cap L^{\infty}(X)$. Here, on the right hand side, quasi-continuous versions of $f$ and $g$ must be used. 

By definition, the (essential) support of $f\in L^2(X,\mu)$ is the support of the measure $|f| d\mu$.
For an open set $U \subset X$, we set
\begin{align*}
& \F_{\mbox{\tiny{c}}}(U) := \{ f \in \F : \textrm{ The support of } f \textrm{ is compact in } U \}, \\
& \F^0(U) := \mbox{ closure of } \F_{\mbox{\tiny{c}}}(U) \mbox{ in } (\F,\| \cdot \|_{\F}), \\
& \F_{\mbox{\tiny{{loc}}}}(U)  :=  \{ f \in L^2_{\mbox{\tiny{loc}}}(U) : \forall \textrm{ compact } K \subset U, \ \exists f^{\sharp} \in \F, f\big|_K = f^{\sharp}\big|_K \mbox{ $\mu$-a.e.} \}.
\end{align*} 
For functions in $\F_{\mbox{\tiny{loc}}}(U)$ we always take their quasi-continuous versions.
Note that $\Gamma(f,g)$ can be defined locally on $U$ for $f,g \in \F_{\mbox{\tiny{loc}}}(U)$ by virtue of \cite[Corollary 3.2.1]{FOT94}. For any $v, u_1, \ldots, u_m \in \F_{\mbox{\tiny{loc}}}(U) \cap L^{\infty}_{\mbox{\tiny{loc}}}(U,\mu)$ and $\Phi \in \mathcal C^1(\R^m)$, we have $\Phi(u) \in \F_{\mbox{\tiny{loc}}}(U) \cap L^{\infty}_{\mbox{\tiny{loc}}}(U,\mu)$ and the chain rule \eqref{eq:chain rule for Gamma} holds. 
For convenience, we set $\F_{\mbox{\tiny{b}}} := 
\F \cap L^{\infty}(X,\mu)$, $\F_{\mbox{\tiny{c}}} := 
\F_{\mbox{\tiny{c}}}(X)$ and $\F_{\mbox{\tiny{loc}}} := 
\F_{\mbox{\tiny{loc}}}(X)$.

Throughout the paper we will use the notation $f \vee a := \max\{f,a\}$, $f \wedge a := \min \{f,a\}$, $f^+ := f \vee 0$ and $f^- := (-f)^+$, for a function $f$ and a real number $a$.

\subsection{Cutoff Sobolev inequalities} \label{ssec:cutoff sobolev}
For the ease of readability, we suppose in this subsection that any metric ball $B(x,R+r) \subset X$ under consideration is relatively compact. Later, we will localize this assumption; see condition (A2-$Y$) in Subsection \ref{ssec:localized CSA}.

Let $\Psi:[0,\infty) \to [0,\infty)$ be a continuous strictly increasing  bijection.
Assume there exist $\beta_1,\beta_2 \in [2,\infty)$ and $C_{\Psi} \in [1,\infty)$ such that, for all $0 < s < R$,
\begin{align} \label{eq:beta}
C_{\Psi}^{-1} \left( \frac{R}{s} \right)^{\beta_1}
\leq
\frac{\Psi(R)}{\Psi(s)} 
\leq 
C_{\Psi} \left( \frac{R}{s} \right)^{\beta_2}.
\end{align}

\begin{definition}
A function $\psi \in \F$ is called a \emph{cutoff function} for $B(x,R)$ in $B(x,R+r)$, where $x \in X$, $R>0$, $r>0$, if
\begin{enumerate}
\item
$\psi$ is continuous,
\item
$0 \leq \psi \leq 1$ $\mu$-a.e.,
\item
$\psi = 1$ on $B(x,R)$ $\mu$-a.e.,
\item
The compact support of $\psi$ is contained in $B(x,R+r)$.
\end{enumerate}
\end{definition}

\begin{definition} 
$(X,d,\mu,\e^*,\F)$ satisfies the cutoff Sobolev condition on annuli, CSA($\Psi$), if there exists a constant $C_0 \in (0,\infty)$ such that for any 
$\epsilon \in (0,1)$, $x \in X$, $R>0$, $r>0$, there exists a cutoff function $\psi$ for $B(x,R)$ in $B(x,R+r)$ such that
 \begin{align} \label{eq:CSA}
\forall f \in \F, \quad \int_A f^2 d\Gamma(\psi,\psi) \leq \epsilon \int_A \psi^2 d\Gamma(f,f) + \frac{C_0 \epsilon^{1-\beta_2/2}}{ \Psi( r )} \int_A \psi f^2 d\mu,
 \end{align}
 where $A = B(x,R+r) \setminus B(x,R)$.
\end{definition}

Abusing notation, we denote by CSA($\Psi$) not only the cutoff Sobolev condition on annuli, but also the collection of all cutoff functions that satisfy \eqref{eq:CSA} for some $x,R,r$. 
We will sometimes write $\psi \in \mbox{CSA($\Psi,\epsilon$)}$ or $\psi \in \mbox{CSA($\Psi,\epsilon,C_0$)}$ when $\psi$ satisfies \eqref{eq:CSA} for the specified
 $\epsilon$ and $C_0$. To keep notation simple, we will write $C_0(\epsilon)$ for $C_0 \epsilon^{1-\beta_2/2}$.

The cutoff Sobolev condition on annuli was introduced in \cite{AB15} for fixed $\epsilon = \frac{1}{8}$. From the proof of \cite[Lemma 5.1]{AB15} it is clear that CSA($\Psi$) holds with some fixed $\epsilon$ and for all $r>0, R>0$ if and only if it holds for all $\epsilon \in (0,1)$ and for all $r>0, R>0$ (with a different cutoff function for each $\epsilon$). Thus, the two definitions are equivalent. More precisely, we have the following lemma which quantifies the scaling of the zero order term on the right hand side of \eqref{eq:CSA} as $\epsilon$ varies. 

\begin{lemma} \label{lem:CSA epsilon}
Let $B(x,R+r) \subset X$ be relatively compact. For every $\epsilon \in (0,1)$ there exists $\lambda \in (0,\infty)$ such that the following holds. For each non-negative integer $n$ let 
$b_n = e^{-n\lambda}$, $s_n = c_{\lambda} r e^{-n \lambda/\beta_2}$, where $c_{\lambda}$ is chosen so that $\sum_{n=1}^{\infty} s_n =: r' < r$. Let $r_0=0$,
\begin{align} \label{eq:r_n}
 r_n  = \sum_{k=1}^n s_k
 \end{align}
 and $B_n=B(x,R+r_n)$. Let $\psi_n$ be a cutoff function for $B_{n-1}$ in $B_n$ which satisfies
  \begin{align*}
\forall f\in \F, \qquad \int_{B_n \setminus B_{n-1}} f^2 d\Gamma(\psi_n,\psi_n) \leq c_1 \int_{B_n \setminus B_{n-1}} d\Gamma(f,f) + \frac{c_2}{\Psi( s_n )} \int_{B_n \setminus B_{n-1}} f^2 d\mu,
 \end{align*}
 for some fixed constants $c_1, c_2$ that do not depend on $n, x, r, R$.
  Let
\[ \psi := \sum_{n=1}^{\infty} (b_{n-1} - b_n) \psi_n. \]
Then $\psi$ is a cutoff function for $B(x,R)$ in $B(x,R+r)$ and $\psi$ satisfies \eqref{eq:CSA} for the given $\epsilon$ with some constant $C_0 \in (0,\infty)$ 
that depends only on $\beta_2, C_{\Psi}, c_1, c_2$.
\end{lemma}

The cutoff function $\psi$ constructed in Lemma \ref{lem:CSA epsilon} will serve as a weight function in the weighted Poincar\'e inequality of Theorem \ref{thm:weighted PI}. 
We include the full proof of this lemma for the convenience of the reader, though it is essentially the same as \cite[Proof of Lemma 5.1]{AB15}.

\begin{proof}
Let $f \in \F$. Note that $\psi=1$ on $B_0 = B(x,R)$, and $\psi - (b_{n-1}-b_n)\psi_n$ is constant on $B_n \setminus B_{n-1}$. 
Because of the strong locality and \cite[Theorem 4.3.8]{CF12}, we obtain
\begin{align*} 
& \int f^2 d\Gamma(\psi,\psi) \nonumber \\
& = \int_{B_0} f^2 d\Gamma(\psi,\psi) 
 + \sum_{n=1}^{\infty} (b_{n-1} - b_n)^2 \int_{B_n \setminus B_{n-1}} f^2 d\Gamma(\psi_n,\psi_n)  \nonumber \\
& \quad + 2\sum_{n=1}^{\infty} (b_{n-1} - b_n) \int_{B_n \setminus B_{n-1}} f^2 d\Gamma(\psi_n,\psi- (b_{n-1} - b_n) \psi_n)  \\
& \quad + \sum_{n=1}^{\infty}  \int_{B_n \setminus B_{n-1}} f^2 d\Gamma(\psi-(b_{n-1} - b_n)\psi_n,\psi-(b_{n-1} - b_n)\psi_n)  \\
& = \sum_{n=1}^{\infty} (b_{n-1} - b_n)^2 \int_{B_n \setminus B_{n-1}} f^2 d\Gamma(\psi_n,\psi_n) \nonumber \\
& \le
\sum_{n=1}^{\infty} (b_{n-1} - b_n)^2  \left( c_1 \int_{B_n \setminus B_{n-1}} d\Gamma(f,f) + \frac{c_2}{ \Psi( s_n )} \int_{B_n \setminus B_{n-1}} f^2 d\mu  \right)  \nonumber\\
& \leq
(e^{\lambda} - 1)^2 \left( \sum_{n=1}^{\infty} e^{-2n\lambda}   c_1 \int_{B_n \setminus B_{n-1}} d\Gamma(f,f)\right) \nonumber\\
& \quad  +  \sum_{n=1}^{\infty} (b_{n-1} - b_n)^2 \frac{c_2}{ \Psi( s_n )} \int_{B_n \setminus B_{n-1}} f^2 d\mu \nonumber\\
& \leq 
(e^{\lambda} - 1)^2 c_1 \int \psi^2 d\Gamma(f,f) 
+  \sum_{n=1}^{\infty} (b_{n-1} - b_n)^2 \frac{c_2}{ \Psi( s_n )} \int_{B_n \setminus B_{n-1}} f^2 d\mu. \nonumber
\end{align*} 
The last inequality is where we needed the annuli (rather than balls) because we want the sum to be a telescoping sum. We also used the fact that $\psi_n \ge b_n = e^{-n\lambda}$ on $B_{n-1} \setminus B_n$.
By \eqref{eq:beta}, we have
\begin{align} \label{eq:Psi(r)/Psi(s_n)}
\frac{\Psi( r )}{ \Psi( s_n )} 
 \leq C_{\Psi}
\left( \frac{r}{c_{\lambda} r e^{-n \lambda/\beta_2}} \right)^{\beta_2} 
& \leq C_{\Psi} \frac{ e^{\lambda} - 1}{c_{\lambda}^{\beta_2}(b_{n-1}-b_n)}.
\end{align}
Thus,
\begin{align*}
& \sum_{n=1}^{\infty} (b_{n-1} - b_n)^2 \frac{c_2}{ \Psi( s_n )} \int_{B_n \setminus B_{n-1}} f^2 d\mu 
 \leq
\frac{ (e^{\lambda} - 1)^2 }{c_{\lambda}^{\beta_2}}  \frac{c_2 \cdot C_{\Psi}}{\Psi( r )} \int \psi f^2 d\mu.
\end{align*}
Finally,
\begin{align*}
& \int f^2 d\Gamma(\psi,\psi) 
 \leq 
(e^{\lambda} - 1)^2 c_1 \int \psi^2 d\Gamma(f,f) 
+ \frac{ (e^{\lambda} - 1)^2 }{c_{\lambda}^{\beta_2}} 
 \frac{c_2 \cdot C_{\Psi}}{\Psi( r )} \int \psi f^2 d\mu.
\end{align*} 
Choose $\lambda := \log( 1 + (\epsilon/c_1)^{1/2})$. Then $(e^{\lambda}-1)^2 c_1 = \epsilon$. 
By the choice of $c_{\lambda}$,
\[ c_{\lambda} 
= e^{\lambda/\beta_2} (1-e^{-\lambda/\beta_2}) \frac{r'}{r}
= (e^{\lambda/\beta_2} - 1) \frac{r'}{r}.
\]
Hence,
\begin{align*}
\frac{ (e^{\lambda} - 1)^2 }{c_{\lambda}^{\beta_2}}
& = \frac{ (e^{\lambda} - 1)^2}{(e^{\lambda/\beta_2}-1)^{\beta_2}} \left( \frac{r'}{r} \right)^{-\beta_2}
= \frac{\epsilon}{c_1} (e^{\lambda/\beta_2}-1)^{-\beta_2} \left( \frac{r'}{r} \right)^{-\beta_2} \\
& \le \mbox{const}(\beta_2,c_1,r'/r) \cdot \left(\frac{\epsilon}{c_1} \right)^{1-\beta_2/2},
\end{align*}
where we applied the trivial inequality $(e^x-1)^{-1} \leq x^{-1}$ with $x = \log(1+(\epsilon/c_1)^{1/2})/\beta_2$. This completes the proof.
\end{proof}

\subsection{Local cutoff Sobolev condition} \label{ssec:localized CSA}

Let $(X,d,\mu,\e^*,\F)$ be as in Section \ref{ssec:reference form}. Let $Y \subset X$ be open and $R_0 > 0$. 
\begin{definition}
The {\em cutoff Sobolev inequality on annuli, \em CSA($\Psi$)}, is satisfied on $Y$ up to scale $R_0$ if there exists a constant $C_0 \in (0,\infty)$ such that, for any $\epsilon \in (0,1)$, $0 < r < R \le R_0$, $B(x,2R) \subset Y$,
there exists a cutoff function $\psi$ for $B(x,R)$ in $B(x,R+r)$ such that
 \begin{align} \label{eq:localized CSA epsilon} 
\forall f \in \F, \quad \int_A f^2 d\Gamma(\psi,\psi) \leq \epsilon \int_A \psi^2 d\Gamma(f,f) + \frac{C_0 \epsilon^{1-\beta_2/2}}{ \Psi(r)} \int_A \psi f^2 d\mu,
 \end{align}
where $A = B(x,R+r) \setminus B(x,R)$.
\end{definition}

\subsection{Structural assumptions on the form}
Let $(X,d,\mu,\e^*,\F)$ be as in Section \ref{ssec:reference form}.
We will refer to $(\e^*,\F)$ as the {\em reference form} for the bilinear forms defined below.
Let $(\e_t,\F)$, $t \in \R$, be a family of (possibly non-symmetric) local bilinear forms that all have the same domain $\F$ as the reference form $(\e^*,\F)$. 
We always assume that, for every $f,g \in \F$, the map $t \mapsto \e_t(f,g)$ is measurable. 

For $f,g \in \F$, let $\e_t^{\mbox{\tiny{sym}}}(f,g) := \frac{1}{2} \big[ \e_t(f,g) + 
\e_t(g,f) \big]$ be the symmetric part of $\e_t(f,g)$ and let $\e_t^{\mbox{\tiny{skew}}}(f,g) := 
\frac{1}{2} \big[ \e_t(f,g) - \e_t(g,f) \big]$ be the skew-symmetric part.
Notice that $1 \in \F_{\mbox{\tiny{loc}}}$, thus $\e_t(1,f)$ and $\e_t(f,1)$ are well-defined for any $f \in \F_{\mbox{\tiny{c}}}$.
We will use the decomposition
\[ \e_t(f,g) = \e_t^{\mbox{\tiny{s}}}(f,g) + \e_t^{\mbox{\tiny{sym}}}(fg,1) + \l_t(f,g) + \r_t(f,g), \quad \mbox{ for any } f,g \in \F \mbox{ with } fg \in \F_{\mbox{\tiny{c}}}, \]
that we introduced in \cite{LierlSC2}.
Here, the so-called \emph{symmetric strongly local part} $\e_t^{\mbox{\tiny{s}}}$ is defined by
\[ \e_t^{\mbox{\tiny{s}}}(f,g) := \e_t^{\mbox{\tiny{sym}}}(f,g)- \e_t^{\mbox{\tiny{sym}}}(fg,1),
\quad f,g \in \F \mbox{ with } fg \in \F_{\mbox{\tiny{c}}}, \]
and the bilinear forms  $\l_t$ and $\r_t$ are defined by
\begin{align*}
\l_t (f,g) & := \frac{1}{4} \big[ \e_t(fg,1) - \e_t(1,fg) + 
\e_t(f,g) - \e_t(g,f) \big],\\
\r_t (f,g) & := \frac{1}{4} \big[ \e_t(1,fg) - \e_t(fg,1) + \e_t(f,g) - \e_t(g,f) \big]
= - \l_t(g,f),
\end{align*}
for any $f,g \in \F$ with $fg \in \F_{\mbox{\tiny{c}}}$. Due to the locality of $\e_t$, the bilinear forms 
$\l_t(f,g)$ and $\r_t(f,g)$ are well-defined whenever $f \in \F_{\mbox{\tiny{loc}}} \cap L^{\infty}_{\mbox{\tiny{loc}}}(X,\mu)$ and $g \in \F_{\mbox{\tiny{c}}} \cap L^{\infty}_{\mbox{\tiny{loc}}}(X,\mu)$, or vice versa.

Let $\mathcal D$ be a linear subspace of $\F \cap \mathcal C_{\mbox{\tiny{c}}}(X)$ such that
\begin{enumerate}
\item
$\mathcal D$ is dense in $(\F,\| \cdot \|_{\F})$.
\item
If $f \in \mathcal D$ then $(f \vee 0) \in \mathcal D$ and $(f \wedge 1) \in \mathcal D$.
\item
If $f \in \mathcal D$ then $\Phi(f) \in \mathcal D$ for any function $\Phi \in \mathcal{C}^1(\R^m)$ with $\Phi(0)=0$, where $m$ is a positive integer.
\end{enumerate} By the regularity of the reference form $(\e^*,\F)$, such a space $\mathcal D$ exists.
We make the following assumption on the structure of the forms $\e_t$, $t \in \R$.
\begin{assumption} \label{as:0}
For each $t \in \R$, $\e_t$ is a local bilinear form with domain $D(\e_t) = \F$. For every $f,g \in \F$, the map $t \mapsto \e_t(f,g)$ is measurable. Moreover,
\begin{enumerate}
\item
there exists a constant $C_* \in(0,\infty)$ such that
\[ |\e_t(f,g)| \leq C_* || f||_{\F} ||g||_{\F}, \quad \forall f,g \in \F, \]
\item
for all $f,g \in \F_{\mbox{\tiny{b}}}$ with $fg \in \F_{\mbox{\tiny{c}}}$,
\[ |\e_t^{\mbox{\tiny{sym}}}(fg,1)| \leq C_* || f||_{\F} ||g||_{\F}, \]
\item there is a constant $C \in [1,\infty)$ such that
\[ \frac{1}{C} \e^*(f,f) \leq \e_t^{\mbox{\tiny{s}}}(f,f) \leq C \e^*(f,f), \quad \forall f \in \F \cap \mathcal{C}_{\mbox{\tiny{c}}}(X). \]
\item (Product rule for $\l_t$)
For any $u,v,f \in \mathcal D$, 
 \[ \l_t(uf,v) = \l_t(u,fv) + \l_t(f,uv). \]
\item (Chain rule for $\l_t$)
For any $v, u_1, u_2, \ldots, u_m \in \mathcal D$ and $u = (u_1, \ldots, u_m)$, and for any $\Phi \in \mathcal C^2(\R^m)$,
\begin{align*}
\l_t(\Phi(u),v) = \sum_{i=1}^{m} \l_t(u_i, \Phi_{x_i}(u) v).
\end{align*}
\item
There exist constants $0 < c \le \alpha < \infty$ such that, for all $f \in \F$,
\[ \e_t(f,f) + \alpha \int f^2 d\mu \geq c || f||_{\F}^2. \]
\end{enumerate}
\end{assumption}

Part (i) and (vi) of Assumption \ref{as:0} ensure the existence of weak solutions to the heat equation. See, e.g., \cite{LM68}.

Under Assumption \ref{as:0}, the bilinear forms $\e_t$, $\e_t^{\mbox{\tiny{sym}}}$, and $\e_t^{\mbox{\tiny{skew}}}$ are continuous on $\F \times \F$. For results on 
extending the bilinear forms $\l_t$ and $\r_t$ and the maps $(f,g) \mapsto \e_t(fg,1)$ and $(f,g) \mapsto \e_t(1,fg)$ to $\F \times \F$, 
see \cite[Section 7.2]{LierlSC2}. The elementary proof of the next lemma will be given elsewhere.

\begin{lemma}
Under {\em Assumption \ref{as:0}}(i)-(iii), the bilinear form $\e^{\mbox{\em\tiny{s}}}_t$, defined for $f,g \in \F_{\mbox{\em \tiny{b}}}$ with $fg \in \F_{\mbox{\em \tiny{c}}}(X)$, 
extends continuously to $\F \times \F$, and the extension $(\e^{\mbox{\em \tiny{s}}}_t,\F)$ is a strongly local regular symmetric Dirichlet form.
\end{lemma}

Under Assumption \ref{as:0}, the Dirichlet form $(\e_t^{\mbox{\tiny{s}}},\F)$ admits an energy measure $\Gamma_t$ which has all properties that are described in 
Section \ref{ssec:reference form} for the energy measure $\Gamma$ of $(\e^*,\F)$. In particular, $\Gamma_t$ satisfies the product rule, the chain rule, and a 
Cauchy-Schwarz type inequality.

Assumption \ref{as:0}(ii) implies that there exists a constant $C_{10} \in [1,\infty)$ such that
\begin{align} \label{eq:C_10}
 \frac{1}{C_{10}} \int f^2 d\Gamma(g,g) \leq \int f^2 d\Gamma_t(g,g) \leq C_{10} \int f^2 d\Gamma(g,g), \quad \forall f,g \in \F \cap \mathcal{C}_{\mbox{\tiny{c}}}(X).
 \end{align}
See \cite{Mosco94}. Of course, this inequality extends to all bounded Borel measurable functions 
$f:X \to (-\infty,+\infty)$ and $g \in \F$. The inequality also holds when $f \in \F$ and $g \in \mbox{CSA($\Psi$)}$.
If the reference form $(\e^*,\F)$ satisfies CSA($\Psi$,$C_0$) locally on $Y$ up to scale $R_0$, and if $(\e_t,\F)$ satisfies Assumption \ref{as:0}, then 
$(\e_t^{\mbox{\tiny{s}}},\F)$ satisfies 
CSA($\Psi$, $\hat C_0$) locally on $Y$ up to scale $R_0$ (with $\hat C_0$ depending on $C_0$ and $C_{10}$).

We refer to Section \ref{sec:examples} and to \cite{LierlSC2} for examples of forms $\e_t$ that satisfy Assumption \ref{as:0}.

\subsection{Quantitative assumptions on the perturbations} \label{ssec:quantitative assumptions}
Suppose Assumption \ref{as:0} is satisfied. 
In this section we introduce quantitative assumptions on the zero-order part and on the skew-symmetric part of each of the forms $(\e_t,\F)$, $t \in \R$. 
We will show in Section \ref{sec:Moser} below that our assumptions are sufficient to perform the Moser iteration technique to obtain $L^2$-mean value estimates. 
The statements of Assumption \ref{as:skew1} and Assumption \ref{as:skew2} are inspired by and weaker than \cite[Assumptions 1 and 2]{LierlSC2}. The new contribution here is that 
we state these quantitative conditions only for functions $\psi$ that are cutoff functions and in CSA($\Psi$). 

As before, we fix an open connected set $Y \subset X$ and $R_0>0$. Let $C_0 \in (0,\infty)$ be given.
Let 
\begin{align} \label{eq:C_1}
C_1(\epsilon) := \epsilon^{-1/2} C_0(\epsilon) = C_0 \cdot \epsilon^{\frac{1-\beta_2}{2}}, \quad \mbox{ for } \epsilon \in (0,1].
\end{align}

\begin{assumption} \label{as:skew1} 
There are constants $C_2, C_3, C_{11} \in [0,\infty)$ such that for all $t \in \R$, for any $\epsilon \in (0,1)$, any $0 < r < R \leq R_0$, any ball 
$B(x,2R) \subset Y$, any cutoff function $\psi \in \mbox{CSA}(\Psi,\epsilon,C_0)$ for $B(x,R)$ in $B(x,R+r)$, and any 
$0 \leq f \in \F_{\mbox{\tiny{loc}}}(Y) \cap L^{\infty}_{\mbox{\tiny{loc}}}(Y,\mu)$,
\begin{align*}
& \quad |\e_t^{\mbox{\tiny{sym}}}(f^2 \psi^2,1)| + |\e_t^{\mbox{\tiny{skew}}}(f^2 \psi^2,1)| +  \left| \e_t^{\mbox{\tiny{skew}}}(f,f\psi^2) \right| \\
& \leq C_{11} \epsilon^{1/2} \int \psi^2 d\Gamma(f,f) + (C_2 + C_3  \Psi( r )) \frac{C_1(\epsilon)}{ \Psi( r )} \int_B f^2 d\mu,
\end{align*}
where $B=B(x,R+r)$. 
\end{assumption}

\begin{assumption} \label{as:skew2} 
There are constants $C_4$, $C_5$, $C_{11} \in [0,\infty)$ such that for all $t \in \R$, for any $\epsilon \in (0,1)$, 
any $0 < r < R \leq R_0$, any ball $B(x,2R) \subset Y$, any cutoff function $\psi \in \mbox{CSA}(\Psi,\epsilon,C_0)$ for 
$B(x,R)$ in $B(x,R+r)$, and any $0 \leq f \in \F_{\mbox{\tiny{loc}}}(Y)$ with $f + f^{-1} \in L^{\infty}_{\mbox{\tiny{loc}}}(Y,\mu)$,
\begin{align*} 
\big|\e_t^{\mbox{\tiny{skew}}}(f,f^{-1} \psi^2)\big|
& \leq C_{11} \epsilon^{1/2} \int \psi^2 d\Gamma(\log f,\log f) + (C_4 + C_5  \Psi( r )) \frac{C_1(\epsilon)}{ \Psi( r )} \int_B d\mu,
\end{align*}
where $B=B(x,R+r)$.
\end{assumption}

\begin{remark}
For simplicity, we may and will assume that the constants $C_{11}$ in Assumption \ref{as:skew1} and in Assumption \ref{as:skew2} are the same.
\end{remark}

\subsection{Some preliminary computations} \label{ssec:algebraic computations}
In the next three lemmas, we consider bilinear forms $(\e_t,\F)$, $t \in \R$, which satisfy Assumption \ref{as:0} and Assumption \ref{as:skew1} with respect to the reference 
form $(\e^*,\F)$. Recall that $Y$ is an open subset of $X$. For a non-negative function $u$ and a positive integer $n$ let
\[ u_n := u \wedge n. \]

\begin{lemma} \label{lem:SUP s}
Suppose Assumption \ref{as:0} and Assumption \ref{as:skew1} are satisfied.
Let $p \in \R$, $\epsilon \in (0,1)$, $0 < r < R \leq R_0$, and $B(x,2R) \subset Y$. Let $\psi \in \mbox{\em {CSA($\Psi,\epsilon,C_0$)}}$ be a cutoff function for $B(x,R)$
 in $B(x,R+r)$, and $0 \leq u \in \F_{\mbox{\emph{\tiny{loc}}}}(Y) \cap L_{\mbox{\emph{\tiny{loc}}}}^{\infty}(Y,\mu)$. 
Assume either of the following hypotheses.
\begin{enumerate}
\item
$p\geq 2$,
\item
$u$ is locally uniformly positive.
\end{enumerate} 
Then $u u_n^q \in \F_{\mbox{\emph{\tiny{loc}}}}(Y)$, $u u_n^q \psi^2 \in \F_{\mbox{\emph{\tiny{c}}}}(Y)$, for any $q \geq 0$. Moreover, for any $k >0$ it holds
\begin{align} \label{eq:SUP s}
(1-p) \e_t^{\mbox{\emph{\tiny{s}}}}(u, u u_n^{p-2} \psi^2) 
& \le  \left( 8k \epsilon C_{10} + C \left( \frac{|1-p|^2}{k} + 1-p \right) \right) \int \psi^2  u_n^{p-2} d\Gamma(u,u) \nonumber \\       
      &  \quad + \left( 2 k \epsilon C_{10} (p-2)^2 - C' \left( (1-p)^2 + ( 1-p ) \right) \right) \int \psi^2  u_n^{p-2} d\Gamma(u_n,u_n) \nonumber \\
      & \quad + 4k C_{10} \frac{C_0(\epsilon)}{\Psi( r )} \int \psi u^2 u_n^{p-2} d\mu,
\end{align}
where $C=C_{10}$ if $\frac{|1-p|^2}{k} + 1-p>0$ and $C=\frac{1}{C_{10}}$ otherwise, and $C'=\frac{1}{C_{10}}$ if $ (1-p)^2 +  1-p > 0$ and $C' = C_{10}$ otherwise.
\end{lemma}

\begin{proof}
The first assertion follows from \cite[Lemma 1.3]{LierlSC2}. Moreover, by \eqref{eq:chain rule for Gamma} and \eqref{eq:CS} we have for any $k >0$ that
\begin{align*}
(1-p) \e_t^{\mbox{\tiny{s}}}(u, u u_n^{p-2} \psi^2) 
& \leq  4k \int u^2 u_n^{p-2} d\Gamma_t (\psi,\psi)  \nonumber \\
      & \quad + \left( \frac{|1-p|^2}{k} + (1-p) \right) \int \psi^2  u_n^{p-2}  d\Gamma_t(u,u) \nonumber \\       
      & \quad - ( (1-p)^2 + (1-p)) \int \psi^2  u_n^{p-2} d\Gamma_t(u_n,u_n).
\end{align*}
Hence \eqref{eq:SUP s} follows from applying \eqref{eq:CSA} and \eqref{eq:C_10}.
\end{proof}

\begin{lemma} \label{lem:SUP s for p<1}
Suppose Assumption \ref{as:0} and Assumption \ref{as:skew1} are satisfied.
Let $p \in (-\infty,1-\eta)$ for some small $\eta > 0$. Let $\epsilon \in (0,1)$, $0 < r < R \leq R_0$, and $B(x,2R) \subset Y$. Let $\psi \in \mbox{\em {CSA($\Psi,\epsilon,C_0$)}}$ be a cutoff function for $B(x,R)$ in $B(x,R+r)$, and $0 \leq u \in \F_{\mbox{\emph{\tiny{loc}}}}(Y) \cap L_{\mbox{\emph{\tiny{loc}}}}^{\infty}(Y,\mu)$. 
Assume us is locally uniformly positive and locally bounded. 
Then, for any $k >0$, it holds
\begin{align}
 \e_t^{\mbox{\emph{\tiny{s}}}}(u, u^{p-1} \psi^2) 
& \le   \left(  \frac{2C_{10} \epsilon}{\eta} p^2 + \frac{1}{C_{10}} \left( p - (1-\eta/2) \right) \right) \int \psi^2  u^{p-2} d\Gamma(u,u) \nonumber \\       
      & \quad + \frac{8C_{10}}{\eta}  \frac{C_0(\epsilon)}{\Psi( r )} \int \psi u^p d\mu.
\end{align}
\end{lemma}
For the proof, simply choose $k=\frac{2}{\eta} (1-p)$ in the proof of Lemma \ref{lem:SUP s}.

\begin{lemma} \label{lem:SUP skew} 
Suppose Assumption \ref{as:0} and Assumption \ref{as:skew1} are satisfied.
Let $p \in \R$, $\epsilon \in (0,1)$, $0 < r < R \leq R_0$, and $B(x,2R) \subset Y$. Let $\psi \in \mbox{\em {CSA($\Psi,\epsilon,C_0$)}}$ be a cutoff function for 
$B(x,R)$ in $B=B(x,R+r)$, and $0 \leq u \in \F_{\mbox{\emph{\tiny{loc}}}}(Y) \cap L^{\infty}_{\mbox{\emph{\tiny{loc}}}}(Y,\mu)$. 
Assume either of the following hypotheses.
\begin{enumerate}
\item
$p\geq 2$,
\item
$p \neq 0$ and $u$ is locally uniformly positive.
\end{enumerate}
Then,
\begin{align*}
|\e_t^{\mbox{\emph{\tiny{sym}}}}(u^2 u_n^{p-2} \psi^2,1)|
& \leq   
       2C_{11}\epsilon^{1/2} \int u_n^{p-2}  \psi^2 d\Gamma(u,u) + \frac{(p-2)^2}{2} C_{11} \epsilon^{1/2} \int u_n^{p-2} \psi^2 d\Gamma(u_n,u_n) \\
& \quad +  (C_2 + C_3  \Psi( r )) \frac{C_1(\epsilon)}{\Psi( r )} \int_B u^2 u_n^{p-2} d\mu,
\end{align*}
and,
\begin{align*}
|\e_t^{\mbox{\emph{\tiny{skew}}}}(u, u u_n^{p-2} \psi^2) |
& \leq 2 C_{11} \epsilon^{1/2} \int u_n^{p-2}  \psi^2  d\Gamma(u,u)  \\
& \quad   + C_{11} \epsilon^{1/2} \left(\frac{(p-2)^2}{2}  + \frac{|p(p-2)|}{4} \right) \int u_n^{p-2}  \psi^2  d\Gamma(u_n, u_n) \\
  & \quad + \big( C_2 + C_3  \Psi( r ) \big) \frac{C_1(\epsilon)}{ \Psi( r )} \int_B u^2 u_n^{p-2}  d\mu \\
& \quad  +  \left( C_2 + C_3  \Psi( r ) \right) \frac{|p-2|}{|p|} \frac{C_1(\epsilon)}{ \Psi( r )} \int_B u_n^p  d\mu.
\end{align*}
\end{lemma}

\begin{proof}
We will prove the assertion for $u \in \mathcal D$. Then the general case follows by approximation, using Assumption 0(i), the locality of $\e_t$, and the fact that 
$\mathcal D$ is dense in $(\F,\| \cdot \|_{\F})$. 
First consider the case when $u$ is uniformly positive on the support of $\psi$. 
By strong locality, \eqref{eq:Gamma(fg)} and \eqref{eq:chain rule for Gamma}, we have
\begin{align} \label{eq:uu_n^p-2/2}
\int \psi^2  d\Gamma(u u_n^{\frac{p-2}{2}},u u_n^{\frac{p-2}{2}})
& \leq 2 \int u_n^{p-2} \psi^2  d\Gamma(u,u) + \frac{ (p-2)^2}{2} \int u_n^{p-2} \psi^2  d\Gamma(u_n,u_n).
\end{align}
The first assertion follows easily from Assumption \ref{as:skew1} and \eqref{eq:uu_n^p-2/2}.
By \cite[Lemma 2.13]{LierlSC2}, we have
\begin{align} \label{eq:skew identity with p}
  \e_t^{\mbox{\tiny{skew}}}(u,u u_n^{p-2}  \psi^2) 
& =  \e_t^{\mbox{\tiny{skew}}}(u u_n^{\frac{p-2}{2}},u u_n^{\frac{p-2}{2}}  \psi^2) 
    + \frac{2-p}{p} \e_t^{\mbox{\tiny{skew}}}(u_n^{p/2},u_n^{p/2}  \psi^2) \nonumber \\
& \quad     + \frac{2-p}{p} \e_t^{\mbox{\tiny{skew}}}(u_n^p  \psi^2,1).
\end{align}    
Hence, by Assumption \ref{as:skew1}, \eqref{eq:Gamma(fg)} and \eqref{eq:uu_n^p-2/2}, we have
\begin{align*}
 |\e_t^{\mbox{\tiny{skew}}}(u,u u_n^{p-2}  \psi^2)| 
& \leq  
   2 C_{11}\epsilon^{1/2} \int u_n^{p-2}  \psi^2  d\Gamma(u,u)  \\
& \quad   
+ C_{11} \epsilon^{1/2} \left(\frac{(p-2)^2}{2}  +  \frac{|p(p-2)|}{4} \right) \int u_n^{p-2} \psi^2  d\Gamma(u_n, u_n) \\
  & \quad 
  + \big( C_2 + C_3  \Psi( r ) \big) \frac{C_1(\epsilon)}{ \Psi( r )} \int_B u^2 u_n^{p-2}  d\mu \\
& \quad  
+  \left( C_2 + C_3  \Psi( r ) \right) \frac{|p-2|}{|p|} \frac{C_1(\epsilon)}{ \Psi( r )} \int_B u_n^p  d\mu.
\end{align*}

In the case when $u$ is not uniformly positive on the support of $\psi$, repeat the proof with $u+\varepsilon$ in place of $u$. If $p \geq 2$, then we can let $\varepsilon$ tend to $0$ 
at the end of the proof. 
\end{proof}

For $\varepsilon >0$, let
\[ u_{\varepsilon} := u + \varepsilon. \]

\begin{lemma} \label{lem:e(1,)} 
Suppose Assumption \ref{as:0} and Assumption \ref{as:skew1} are satisfied.
Let $p \in \R$, $\epsilon \in (0,1)$, $0 < r < R \leq R_0$, and $B(x,2R) \subset Y$. Let $\psi \in \mbox{\em CSA}(\Psi,\epsilon,C_0)$ be a cutoff function 
for $B(x,R)$ in $B(x,R+r)$, and $0 \leq u \in \F_{\mbox{\emph{\tiny{loc}}}}(Y) \cap L^{\infty}_{\mbox{\emph{\tiny{loc}}}}(Y,\mu)$. 
Then, for any $k \geq 1$,
\begin{align*}
 |\e_t(\varepsilon,u_{\varepsilon}^{p-1} \psi^2)| 
& \leq C_{11} \epsilon^{1/2} \frac{(p-1)^2}{4} \int  \psi^2 u_{\varepsilon}^{p-2} d\Gamma(u_{\varepsilon}, u_{\varepsilon}) 
 +  (C_2 + C_3 \Psi(r)) \frac{C_1(\epsilon)}{\Psi(r)} \int_B u_{\varepsilon}^{p} d\mu,
\end{align*}
where $B=B(x,R+r)$.
\end{lemma}

\begin{proof} 
We apply Assumption \ref{as:skew1} and \eqref{eq:chain rule for Gamma}. Then,
\begin{align*}
& \quad 
|\e_t(\varepsilon,u_{\varepsilon}^{p-1} \psi^2)| \\
& \leq  \varepsilon |\e_t^{\mbox{\tiny{skew}}}(1,u_{\varepsilon}^{p-1} \psi^2)|
    + \varepsilon | \e_t^{\mbox{\tiny{sym}}}(1,u_{\varepsilon}^{p-1} \psi^2)| \\
& \le  \varepsilon  C_{11} \epsilon^{1/2}  \int \psi^2 d\Gamma( u_{\varepsilon}^{\frac{p-1}{2}},u_{\varepsilon}^{\frac{p-1}{2}}) 
 +  \varepsilon (C_2 + C_3 \Psi(r)) \frac{C_1(\epsilon)}{\Psi(r)} \int_B  u_{\varepsilon}^{p-1} d\mu \\
& \le   C_{11} \epsilon^{1/2}  \frac{(p-1)^2}{4} \int \varepsilon u_{\varepsilon}^{p-3} \psi^2 d\Gamma( u_{\varepsilon},u_{\varepsilon}) 
 +  (C_2 + C_3 \Psi(r)) \frac{C_1(\epsilon)}{\Psi(r)} \int_B \varepsilon u_{\varepsilon}^{p-1} d\mu.
\end{align*}
Applying $\varepsilon  \leq u_{\varepsilon}$ completes the proof.
\end{proof}

\section{Sobolev and Poincar\'e inequalities}
\label{sec:Sobolev and Poincare}

\subsection{Weak, strong, and weighted Poincar\'e inequalities}
In this section we consider Sobolev and Poincar\'e inequalities for the symmetric reference form $(\e^*,\F)$ defined in Section \ref{ssec:reference form}.
We fix an open connected set $Y \subset X$ and $R_0>0$.

For the rest of the paper we suppose that
\begin{align}
\mbox{ If } B(x,2R) \subset Y \mbox{ with } 0 < r < R \leq R_0, \mbox{ then } B(x,R+r) \mbox{ is relatively compact.} \tag{A2-$Y$}
\end{align}
Note that  any open set $Y$ such that $\overline{Y}$ is complete in $(X,d)$
satisfies (A2-$Y$), see, e.g., \cite[Lemma 1.1(i)]{SturmIII}.

\begin{definition}
The \emph{volume doubling property} is satisfied on $Y$ up to scale $R_0$ if there exists a constant $C_{\mbox{\tiny{VD}}} \in (1,\infty)$ such that for every ball $B(x,2R) \subset Y$, $0 < r < R \leq R_0$,
\begin{align} \label{eq:VD}
V(x,R+r) \leq C_{\mbox{\tiny{VD}}} \, V(x,R), \tag{VD}
\end{align}
where $V(x,R) = \mu( B(x,R))$ denotes the volume of $B(x,R)$.
\end{definition}

\begin{lemma} \label{lem:nu}
If {\em VD} is satisfied on $Y$ up to scale $R_0$, then for $\nu= \log_2(C_{\mbox{\em \tiny{VD}}})$,
\[ \frac{\mu(B(x,R))}{\mu(B(y,s))} \leq C_{\mbox{\em \tiny{VD}}}^2 \left( \frac{R}{s} \right)^{\nu}, \] for all $0 < s < R \leq R_0$ and $y \in B(x,R)$ with $B(y,2R) \subset Y$.
\end{lemma}

\begin{proof}
See \cite[Lemma 5.2.4]{SC02}.
\end{proof}

\begin{definition}
$(\e^*,\F)$ satisfies the \emph{(strong) Poincar\'e inequality} PI($\Psi$) on $Y$ up to scale $R_0$, if there exists a constant $C_{\mbox{\tiny{PI}}} \in (0,\infty)$ such 
that for any $0 <  r < R \leq R_0$ and $B(x,2R) \subset Y$,
\begin{align}
\forall f \in \F_{\mbox{\tiny{loc}}}(Y), \  \int_B |f - f_B|^2 d\mu  \leq  C_{\mbox{\tiny{PI}}} \, \Psi(R+r) \int_{B} d\Gamma(f,f), \tag{PI($\Psi$)} 
\end{align}
where $f_B = \frac{1}{V(x,R+r)} \int_{B(x,R+r)} f d\mu$ is the mean of $f$ over $B=B(x,R+r)$. 
\end{definition}

\begin{assumption} \label{as:VD+PI}
The reference form $(X,d,\mu, \e^*,\F)$ satisfies {A2-$Y$, VD, PI($\Psi$)} and {CSA($\Psi$)} on $Y$ up to scale $R_0$.
\end{assumption}

\begin{theorem} \label{thm:weighted PI} 
Suppose {\em Assumption \ref{as:VD+PI}} is satisfied.
Then $(\e^*,\F)$ satisfies a \emph{weighted Poincar\'e inequality} on $Y$ up to scale $R_0$. That is, there exists a constant $C_{\mbox{\em \tiny{wPI}}} \in (0,\infty)$ 
such that for any $0 < r < R \leq R_0$, any $B(x,2R) \subset Y$, and for every $\epsilon \in (0,1)$, there exists a cutoff function $\psi \in \mbox{\em CSA}(\Psi,\epsilon,C_0)$ for $B(x,R)$ in $B(x,R+r)$ such that
\begin{align} \label{eq:weighted PI}
\forall f \in \F_{\mbox{\em \tiny{loc}}}(Y), \  \int |f - f_{\psi}|^2 \psi^2 d\mu  \leq  C_{\mbox{\tiny{\emph{wPI}}}} \, \Psi(R+r) \int \psi^2 d\Gamma(f,f),
\end{align}
where 
\[ f_{\psi} = \frac{\int f \psi^2 d\mu }{ \int \psi^2 d\mu}. \]
The constant $C_{\mbox{\em \tiny{wPI}}}$ depends only on $C_0$, $C_{\mbox{\em \tiny{VD}}}$, $C_{\mbox{\em \tiny{PI}}}$.
\end{theorem}

\begin{proof} Let $\epsilon \in (0,1)$. Let 
\begin{align} \label{eq:construction psi}
 \psi = \sum_{n=1}^{\infty} (b_{n-1} - b_n) \psi_n
 \end{align}
be the cutoff function constructed in Lemma \ref{lem:CSA epsilon}. In particular, 
for each non-negative integer $n$, $b_n = e^{-n\lambda}$ for some $\lambda = \lambda(\epsilon)$, and $\psi_n \in \mbox{CSA}(\Psi)$ is a cutoff function for $B_{n-1}$ in $B_n$, where $B_n=B(x,R+r_n)$ and the sequence $r_n \uparrow r' < r$ is defined by \eqref{eq:r_n}.
By Lemma \ref{lem:CSA epsilon}, we have $\psi \in \mbox{CSA($\Psi,\epsilon,C_0$)}$ for a suitable choice of $\lambda(\epsilon)$. We will prove the weighted Poincar\'e inequality  \eqref{eq:weighted PI} for the weight $\psi$ given by \eqref{eq:construction psi}.
By the triangle inequality,
\begin{align*}
\int | f - f_{\psi}|^2 \psi^2 d\mu 
& \leq 
\int | f - f_{B_0}|^2 \psi^2 d\mu +  \int | f_{B_0} - f_{\psi}|^2 \psi^2 d\mu. 
\end{align*}
The second integral on the right hand side can be estimated by
\begin{align*}
 \int | f_{B_0} - f_{\psi}|^2 \psi^2 d\mu
 = \int \left| \frac{\int (f - f_{B_0}) \psi^2 d\mu}{ \int \psi^2 d\mu } \right|^2 \psi^2 d\mu 
 \leq  \int | f - f_{B_0}|^2 \psi^2 d\mu,
\end{align*}
where we used the definition of $f_{\psi}$ and the Cauchy-Schwarz inequality.
Thus, it suffices to show that there exists a constant $C \in (0,\infty)$ such that 
 \[ \forall f \in \F_{\mbox{\tiny{loc}}}(Y), \qquad \int | f - f_{B_0}|^2 \psi^2 d\mu \leq C \Psi(R+r) \int \psi^2 d\Gamma(f,f). \]
By \eqref{eq:construction psi} and the fact that $\psi_n$ vanishes outside $B_n$ and $0 \leq \psi_n \leq 1$, we have
\begin{align*}
 \int | f - f_{B_0}|^2 \psi^2 d\mu
& = 
\sum_n \sum_m (b_{n-1} - b_n)(b_{m-1} - b_m) \int | f - f_{B_0}|^2 \psi_n \psi_m d\mu  \\
& \leq 
\sum_n \sum_m (b_{n-1} - b_n)(b_{m-1} - b_m) \int_{B_n \cap B_m} | f - f_{B_0}|^2 d\mu  \\
& \leq  
I_1 + I_2,
\end{align*}
where we applied the triangle inequality with
\[ I_1 := 2\sum_n \sum_m (b_{n-1} - b_n)(b_{m-1} - b_m)   \int_{B_n \cap B_m} |f-f_{B_n \cap B_m}|^2 d\mu \]
and
\[ I_2 := 2\sum_n \sum_m (b_{n-1} - b_n)(b_{m-1} - b_m)   \int_{B_n\cap B_m} |f_{B_n \cap B_m} - f_{B_0}|^2 d\mu. \]
Observe that
\[ b_{n-1} - b_n = e^{\lambda} ( b_n - b_{n+1} ). \]
Applying the strong Poincar\'e inequality on the ball $B_n \cap B_m = B_{n \wedge m}$, and using the fact that $\psi_{n+1} = 1$ on $B_{n}$, we obtain
\begin{align*}
I_1
& \leq
C_{\mbox{\tiny{PI}}} \sum_n \sum_m (b_{n-1} - b_n)(b_{m-1} - b_m)  \Psi(R+(r_{n} \wedge r_{m})) \int_{B_{n} \cap B_{m}} d\Gamma(f,f)   \\
& \leq 
 C_{\mbox{\tiny{PI}}} \Psi(R+r) \sum_n \sum_m (b_{n-1} - b_n)(b_{m-1} - b_m)   \int \psi_{n+1} \psi_{m+1} d\Gamma(f,f)   \\
& \leq 
 C_{\mbox{\tiny{PI}}}  \Psi(R+r) \sum_n \sum_m  e^{2\lambda} (b_{n} - b_{n+1})(b_{m} - b_{m+1})  \int \psi_{n+1} \psi_{m+1} d\Gamma(f,f) \\
& \leq 
 C_{\mbox{\tiny{PI}}} \, e^{2\lambda} \Psi(R+r) \int \psi^2 d\Gamma(f,f).
\end{align*}
Now we estimate $I_2$. Note that $|f_{B_n \cap B_m} - f_{B_0}|$ is constant and $\mu(B_n \cap B_m) \le V(x,R+r) \leq C_{\mbox{\tiny{VD}}} \mu(B_0)$ by the volume doubling property. We apply the triangle inequality and then the Poincar\'e inequality on the balls $B_n \cap B_m$ and $B_0$. This yields
\begin{align*}
I_2 
& \le 2C_{\mbox{\tiny{VD}}} \sum_n \sum_m (b_{n-1} - b_n)(b_{m-1} - b_m)   \int_{B_0} |f_{B_n \cap B_m} - f_{B_0}|^2 d\mu \\
& \le 
4C_{\mbox{\tiny{VD}}} \sum_n \sum_m (b_{n-1} - b_n)(b_{m-1} - b_m)  \left( \int_{B_n \cap B_m} |f_{B_n \cap B_m} - f|^2  d\mu  +  \int_{B_0} |f - f_{B_0}|^2  d\mu \right) \\
& \le
8 C_{\mbox{\tiny{VD}}} C_{\mbox{\tiny{PI}}} \sum_n \sum_m (b_{n-1} - b_n)(b_{m-1} - b_m) \Psi(R+(r_{n} \wedge r_{m})) \int_{B_{n} \cap B_{m}} d\Gamma(f,f) \\
& \le
8 C_{\mbox{\tiny{VD}}} C_{\mbox{\tiny{PI}}} e^{2\lambda} \Psi(R+r) \int \psi^2 d\Gamma(f,f).
\end{align*}
\end{proof}

\begin{definition}
$(\e^*,\F)$ satisfies the \emph{weak Poincar\'e inequality} weak-PI($\Psi$) on $Y$ up to scale $R_0$, if there exist constants $\kappa \in (0,1)$ and $C(\kappa) \in (0,\infty)$ 
such that for any $0 < r < \kappa R < R \leq R_0$ and any ball $B(x,2R) \subset Y$,
\begin{align*}
\forall f \in \F_{\mbox{\tiny{loc}}}(Y), \  \int_B |f - f_B|^2 d\mu  \leq  C(\kappa) \, \Psi(2R) \int_{B(x,2R)} d\Gamma(f,f), 
\end{align*}
where $B=B(x,R+r)$. 
\end{definition}

\begin{remark} \label{rem:strong and weak PI}
If (A2-Y) and VD hold on $Y$ up to scale $R_0$ and if the metric $d$ is geodesic, then the weak Poincar\'e inequality PI($\Psi$) on $Y$ up to scale $R_0$ implies the strong Poincar\'e inequality on $Y$ up to scale $R_0$. 
This is immediate from a weighted Poincar\'e inequality with weight function $\psi = 1_{B(x,R)}$, see \cite[Corollary 5.3.5]{SC02}. 
The weighted Poincar\'e inequality with weight $\psi = 1_{B(x,R)}$ can be proved using a Whitney covering and chaining arguments that
are applicable when the metric is geodesic. See \cite[Section 5.3.2 - 5.3.5]{SC02}.
\end{remark}

\begin{lemma} \label{lem:pseudo PI}
Assume that $(\e^*,\F)$ satisfies {\em A2-$Y$} and {\em VD, PI($\Psi$)} on $Y$ up to scale $R_0$. Then the pseudo-Poincar\'e inequality holds: There is a constant 
$C=C(\beta_1, \beta_2, C_{\Psi}, C_{\mbox{\em \tiny{VD}}}, C_{\mbox{\em \tiny{PI}}}) \in (0,\infty)$ such that for any ball $B(x,2R) \subset Y$ with 
$0 < R \le R_0$, and any $f \in \F_{\mbox{\em \tiny{c}}}(B(x,R))$,
\[  \int |f - f_s|^2 d\mu \leq C \Psi(s) \int d\Gamma(f,f), \quad \forall s \in (0,R), \] 
where $f_s(y) := \frac{1}{V(y,s)} \int_{B(y,s)} f d\mu$.
If, in addition, $f \in \F_{\mbox{\em \tiny{c}}}(B(x,R/4))$ and $B(x,R) \neq Y$, then
\[  \int f^2 d\mu \leq C \Psi(R) \int d\Gamma(f,f), \] 
\end{lemma}

\begin{proof}
The proof is as in the classical case $\Psi(r) = r^2$, with the obvious changes regarding the use of $\Psi(r)$. The idea is to cover $B(x,R)$ with balls $2B_i$ where each $B_i$ has radius $s/10$, and to apply the Poincar\'e inequality to each of the balls $4B_i$. For details, see \cite[Lemma 5.3.2 and Lemma 5.2.5]{SC02}.
\end{proof}

\subsection{Localized Sobolev inequality}

\begin{definition}
$(\e^*,\F)$ satisfies the localized \emph{Sobolev inequality} SI($\Psi$) on $Y$ up to scale $R_0$, if there exist constants $\kappa > 1$ and $C_{\mbox{\tiny{SI}}} \in (0,\infty)$ 
such that for any ball $B(x,4R) \subsetneq B(x,8R) \subset Y$ with $0 < R \leq R_0/4$, and all $f \in \F_{\mbox{\tiny{c}}}(B(x,R))$, we have
\begin{align}  \label{eq:sobolev}
 \left( \int_{B(x,R)} |f|^{2\kappa} d\mu \right)^{\frac{1}{\kappa}}   
    \leq \frac{ C_{\mbox{\tiny{SI}}} }{ V(x,R)^{1 - \frac{1}{\kappa}} } \Psi(R)  \int_{B(x,R)} d\Gamma(f,f). 
\end{align}
\end{definition}

\begin{theorem} \label{thm:sobolev ineq}
If {\em A2-$Y$} and {\em VD, PI($\Psi$)} are satisfied on $Y$ up to scale $R_0$, then 
$(\e^*,\F)$ satisfies {\em SI($\Psi$)} on $Y$ up to scale $R_0$.
The Sobolev constant  $C_{\mbox{\em \tiny{SI}}}$ depends only on $\beta_1, \beta_2, C_{\Psi}$, $C_{\mbox{\em \tiny{VD}}}$ and $C_{\mbox{\em \tiny{PI}}}$. The constant $\kappa$ satisfies $1-\frac{1}{\kappa} = \beta_1/\log_2(C_{\mbox{\em\tiny{VD}}})$.
\end{theorem}

\begin{proof}
We follow \cite[Theorem 5.2.3]{SC02}. It suffices to proof the assertion for non-negative $f$.
For any $y \in B = B(x,R)$, $0 < s < R$, we have by Lemma \ref{lem:nu} that
\[   |f_s(y)|  
\leq \frac{1}{\mu(B(y,s))}    \int_{B(y,s)} |f| d\mu 
\leq \frac{C_{\mbox{\tiny{VD}}}^2}{\mu(B)} \left(\frac{R}{s}\right)^{\nu} ||f||_1, \]
where $\nu = \log_2(C_{\mbox{\tiny{VD}}})$. 
For  $0 \leq f \in \F_{\mbox{\tiny{c}}}(B)$ and $\lambda \geq 0$, write
\[ \mu(\{f \geq \lambda \}) \leq \mu(\{ |f - f_s| \geq \lambda/2 \} \cap B ) + \mu(\{ f_s \geq \lambda/2 \} \cap B ) \]
and consider two cases. 

\textbf{Case 1:}
If $\lambda$ is such that 
\[ \frac{\lambda}{4} > \frac{C_{\mbox{\tiny{VD}}}^2}{\mu(B)} ||f||_1, \]
then pick $s \in (0,R)$ depending on $\lambda$ in such a way that
\[ \frac{\lambda}{4} = \frac{C_{\mbox{\tiny{VD}}}^2}{\mu(B)} \left(\frac{R}{s}\right)^{\nu} ||f||_1. \] 
For this choice of $s$, 
\[ \mu(\{ f_s \geq \lambda/2 \} \cap B) = 0. \]
By \eqref{eq:beta}, we then have for $\kappa$ satisfying $1-\frac{1}{\kappa} = \frac{\beta_1}{\nu}$ that
\begin{align} \label{eq:lambda and kappa}
\lambda^{1-\frac{1}{k}}
\leq 
C \frac{\Psi(R)}{\Psi(s)} \left(\frac{||f||_1}{\mu(B)} \right)^{1-\frac{1}{k}},
\end{align}
where $C$ denotes a positive constant that may change from line to line and depends only on $\beta_1, \beta_2, C_{\Psi}, C_{\mbox{\tiny{VD}}}, C_{\mbox{\tiny{PI}}}$.
Applying the pseudo-Poincar\'e inequality of Lemma \ref{lem:pseudo PI} and \eqref{eq:lambda and kappa}, we obtain
\begin{align*}
\mu(\{ f \geq \lambda \}) 
& \leq  \mu(\{ |f - f_s| \geq \lambda/2 \} \cap B) \\
& \leq  \frac{4}{\lambda^2}    \int |f-f_s|^2 d\mu  \\
& \leq C \frac{4}{\lambda^2} \Psi(s) \int d\Gamma(f,f)  \\
& \leq
  \frac{C}{\lambda^{3 - \frac{1}{\kappa}}} \left( \frac{||f||_1}{\mu(B)} \right)^{1 - \frac{1}{\kappa}} \Psi(R)  \int d\Gamma(f,f).
\end{align*}

\textbf{Case 2:}
If $\lambda$ is such that
\[ \frac{\lambda}{4} \leq \frac{C_{\mbox{\tiny{VD}}}^2}{\mu(B)} ||f||_1, \]
then it follows from the second part of Lemma \ref{lem:pseudo PI} that
\begin{align*}
  \int f^2 d\mu  \leq C \Psi(R) \int d\Gamma(f,f).
\end{align*}
Hence,
\[  \mu( \{ f \geq \lambda \}) \leq \frac{1}{\lambda^2}   \int f^2 d\mu  \leq \frac{C}{\lambda^2} \Psi(R) \int d\Gamma(f,f). \]
We obtain that
\begin{align} \label{eq:S*}
 \lambda^{3-\frac{1}{\kappa}} \mu(\{f \geq \lambda\}) \leq C  \left( \frac{||f||_1}{\mu(B)} \right)^{1 - \frac{1}{\kappa}} \Psi(R) \int d\Gamma(f,f)
\end{align}
holds in both cases.
Now the proof can be completed easily by following the reasoning in \cite[Theorem 3.2.2 and Lemma 3.2.3]{SC02}.
\end{proof}

\section{The Moser iteration technique} \label{sec:Moser}
\subsection{Time-dependent forms}
\label{ssec:time-dependence}
For the rest of the paper, we fix a reference form $(\e^*,\F)$ as in Section \ref{ssec:reference form} and an open set $Y \subset X$. We assume $(\e^*,\F)$ satisfies A2-Y, VD, PI($\Psi$), CSA($\Psi$) on $Y$ up to scale $R_0>0$.
Let $(\e_t,\F)$, $t \in \R$, be a family of bilinear forms that satisfy Assumption \ref{as:0} and Assumption \ref{as:skew1}.

\subsection{Local very weak solutions} \label{ssec:very weak solutions}
We recall the notion of very weak solutions introduced in \cite{LierlSC2}. For an open time interval $I$ and a separable Hilbert space $H$, 
let $L^2(I \to H)$ be the Hilbert space of 
those functions $v: I \to H$ such that 
 \[  \Vert v \Vert_{L^2(I \to H)} := \left( \int_I \Vert v(t) \Vert_{H}^2 \, dt \right)^{1/2} < \infty. \]
It is well-known that $L^2(I \to L^2(X,\mu))$ can be identified with  $L^2(I \times X,dt \times d\mu)$. Indeed, continuous functions with compact support in $I \times X$ 
are dense in both spaces and the two norms coincide on these functions.

Let $L^2_{\mbox{\tiny{loc}}}(I \to \F;U)$ be the space of all functions $u:I \times U \to \R$ such that for any open interval $J$ relatively compact in $I$, and any open subset $A$ relatively compact in $U$, there exists a function $u^{\sharp} \in L^2(I \to \F)$ such that $u^{\sharp} = u$ a.e.~in $J \times A$.

\begin{definition} 
Define
\[ D(L_t) = \{ f \in \F : g \mapsto \e_t(f,g) \mbox{ is continuous w.r.t. } \Vert \cdot \Vert_{2} \mbox{ on } \F_{\mbox{\tiny{c}}} \}. \]
For $f \in D(L_t)$, let $L_t f$ be the unique element in $L^2(X)$ such that
 \[ - \int L_t f g d\mu  = \e_t(f,g) \quad \mbox{ for all } g \in \F_{\mbox{\tiny{c}}}. \]
Then we say that $(L_t,D(L_t))$ is the infinitesimal generator 
of $(\e_t,\F)$ on $X$. See, e.g., \cite{MR92}.
\end{definition}

\begin{definition} \label{def:local very weak solution}
Let $I$ be an open interval and $U \subset X$ open. Set $Q = I \times U$. A function $u: Q \to \R$ is a \emph{local very weak solution} of the heat equation $\frac{\partial}{\partial t} u = L_t u$ in $Q$, if
\begin{enumerate}
\item
$u \in L^2_{\mbox{\tiny{loc}}}(I \to \F;U)$,
\item 
For almost every $a,b \in I$,
\begin{align} \label{eq:loc weak sol}
 \forall \phi \in \F_{\mbox{\tiny{c}}}(U), \quad  \int u(b,\cdot)  \phi \, d\mu - \int u(a,\cdot)  \phi \, d\mu + \int_a^b \e_t(u(t,\cdot),\phi) dt = 0.
\end{align} 
\end{enumerate} 
\end{definition}

\begin{definition}
Let $I$ be an open interval and $U\subset X$ open. Set $Q = I \times U$. A function $u: Q \to \R$ is a \emph{local very weak subsolution} of 
$\frac{\partial}{\partial t} u = L_t u$ in $Q$, if
\begin{enumerate}
\item
$u \in L^2_{\mbox{\tiny{loc}}}(I \to \F;U)$,
\item 
For almost every $a,b \in I$ with $a < b$, and any non-negative $\phi \in \F_{\mbox{\tiny{c}}}(U)$,
\begin{align} \label{eq:local very weak subsolution}
 \int u(b,\cdot)  \phi \, d\mu - \int u(a,\cdot)  \phi \, d\mu + \int_a^b \e_t(u(t,\cdot),\phi) dt
\leq 0.
\end{align}
\end{enumerate} 
A function $u$ is called a \emph{local very weak supersolution} if $-u$ is a local very weak subsolution.
\end{definition}

Note that a local very weak solution is not required to have a weak time-derivative. A function $u:Q \to \R$ is a {\em local weak solution} in the classical sense if and only if $u$ is a local very weak solution and $u \in \mathcal{C}_{\mbox{\tiny{loc}}}(I \to L^2(U))$, where $\mathcal{C}_{\mbox{\tiny{loc}}}(I \to L^2(U))$ is the space of measurable functions $u: I \times U \to \R$ such that for any open interval $J$ relatively compact in $I$ and any open subset $A$ relatively compact in $U$, there exists a continuous function $u^{\sharp}:I \to L^2(U)$ such that $u=u^{\sharp}$ on $J \times A$.
See \cite[Proposition 7.8]{LierlSC2}.

\subsection{Estimates for sub- and supersolutions} 
\label{ssec:estim for sub and supsol}
Let $B=B(x,r) \subset Y$ and $a \in \R$. For $\sigma,\delta \in (0,1]$, set
\begin{align*}
& \delta B  = B(x,\delta r), \\
& I^- = (a - \Psi(r), a), \quad
I^+ = (a, a + \Psi(r)), \quad
 I^-_{\sigma} = (a - \sigma \Psi(r), a), \quad
 I^+_{\sigma}  = (a, a + \sigma \Psi(r)), \\
& Q^-(x,a,r)  = I^- \times B(x,r), \quad
 Q^+(x,a,r)  = I^+ \times B(x,r), \\
& Q^-_{\sigma,\delta} = I^-_{\sigma} \times \delta B, \quad
 Q^+_{\sigma,\delta} = I^+_{\sigma} \times \delta B.
\end{align*}
Let $0 < \sigma' < \sigma \leq 1$ and $\hat\sigma := \sigma - \sigma'$.
Let $\chi$ be a smooth function of the time variable $t$ such that $0 \leq \chi \leq 1$,  $\chi = 0 \mbox{ in } (-\infty, a - \sigma \Psi(r))$, 
$\chi = 1$ in $(a - \sigma' \Psi(r), \infty)$ and 
\[ 0 \leq \chi' \leq \frac{2}{\hat\sigma \Psi(r)}. \]
Let $0 < \delta' < \delta < 1$ and $\hat\delta := \delta - \delta'$.
 Let $d\bar \mu = d\mu \times dt$.

\begin{lemma} \label{lem:estimate subsol p>2}
 Let $p \geq 2$. Then there exists a cutoff function $\psi \in \mbox{\em CSA}(\Psi,C_0)$ for $B(x,\delta' r)$ in $B(x,\delta' r + \hat \delta r)$ and constants $a_1 \in (0,1)$, $A_1, A_2 \in [0,\infty)$ depending on $C_0$, $C_{10}$, $C_{11}$ such that
\begin{align} \label{eq:subsol}
& \sup_{t \in I^-_{\sigma'}} \int u^p \psi^2 d\mu + a_1 \int_{I^-_{\sigma'}} \int \psi^2 d\Gamma(u^{p/2},u^{p/2}) dt \nonumber \\
\leq & \, \left(\left(A_1(1 + C_2)\frac{1}{ \Psi( \hat\delta r )} + A_2 C_3 \right) p^{\beta_2} + \frac{2}{\hat\sigma \Psi(r)} \right) \int_{Q_{\sigma,\delta}^-} u^p d\bar\mu
\end{align} 
holds for any non-negative local very weak subsolution $u$ of the heat equation for $L_t$ in $Q = Q^-(x,a,r)$ which satisfies 
$\int_{I^-_{\sigma}} \int_{\delta B} u^p d\mu \, dt < \infty$. 
\end{lemma}

\begin{proof}
We follow the line of reasoning in \cite[Proof of Theorem 3.11]{LierlSC2}. 
We pick $k =2(p-1)$ and  $\epsilon = \frac{c^*}{p^2}$ for some sufficiently small $c >0$ that will be chosen later. 
By Lemma \ref{lem:SUP s}, we have for any $s \in I^-$, any cutoff function $\psi \in \mbox{CSA}(\Psi,\epsilon,C_0)$ 
for $B(x,\delta'r)$ in $B(x,\delta' r + \hat \delta r)$, and any non-negative function $f \in \F \cap \mathcal C_{\mbox{\tiny{c}}}(X)$, $f_n := f \wedge n$, that
\begin{align*} 
-  \e^{\mbox{\tiny{s}}}_s(f,f f_n^{p-2} \psi^2)
& \le \left( 16 C_{10} \epsilon  - \frac{1}{2C_{10}} \right) 
\int \psi^2  f_n^{p-2} d\Gamma(f,f) \nonumber \\       
& \quad + \left( 4 C_{10} \epsilon (p-2)^2 - \frac{1}{C_{10}} (p-2) \right) 
\int \psi^2  f_n^{p-2} d\Gamma(f_n,f_n) \nonumber \\
 & \quad + 8 C_{10}
 \frac{C_0(\epsilon)}{\Psi( \hat\delta r )} \int \psi f^2 f_n^{p-2} d\mu.
\end{align*}
By Lemma \ref{lem:SUP skew}, we have
\begin{align*}
& |\e^{\mbox{\tiny{skew}}}_s(f,f f_n^{p-2} \psi^2)|
+ |\e^{\mbox{\tiny{sym}}}_s(f^2 f_n^{p-2} \psi^2,1)| \\
& \leq 4 C_{11} \epsilon^{1/2} \int f_n^{p-2} \psi^2  d\Gamma(f,f)  \\
& \quad   +  C_{11} \epsilon^{1/2} \left( (p-2)^2 + \frac{p (p-2)}{4} \right) 
\int f_n^{p-2} \psi^2  d\Gamma(f_n, f_n) \\
  & \quad + 4 \big( C_2 + C_3  \Psi( \hat\delta r ) \big) \frac{C_1(\epsilon)}{\Psi( \hat\delta r )} \int_{\delta B} f^2 f_n^{p-2}  d\mu.
\end{align*} 
Combining the two estimates, we get
\begin{align} \label{eq:subsol p>2 estimate for e(f,f f^p)}
& \quad -\e_s(f,f f_n^{p-2} \psi^2) \\
& \leq 
 \left( 16C_{10} \epsilon - \frac{1}{2C_{10}}  + 4 C_{11} \epsilon^{1/2} \right)
\int \psi^2  f_n^{p-2} d\Gamma(f,f) \nonumber \\       
& \quad 
+ \left( 4C_{10} \epsilon (p-2)^2 - \frac{1}{C_{10}} (p-2)  + C_{11} \epsilon^{1/2} \left( (p-2)^2 + \frac{p (p-2)}{4} \right)  \right)
\int \psi^2  f_n^{p-2} d\Gamma(f_n,f_n) \nonumber \\
 & \quad + 
 \left[ 8C_{10}
+ 4 \big( C_2 + C_3  \Psi( \hat\delta r ) \big) \right] \frac{C_1(\epsilon)}{\Psi( \hat\delta r )} \int_{\delta B} f^2 f_n^{p-2} d\mu, \nonumber
\end{align}
for any non-negative $f \in \F \cap \mathcal C_{\mbox{\tiny{c}}}(X)$. By the regularity of the reference form, Assumption \ref{as:0} and 
\cite[Lemma 2.12]{LierlSC2}, we can, for any $t \in I^-$, approximate the very weak subsolution $u(t,\cdot)$ by functions in $\F \cap \mathcal C_{\mbox{\tiny{c}}}(X)$, 
so that \eqref{eq:subsol p>2 estimate for e(f,f f^p)} holds with $u(t,\cdot)$ in place of $f$. 
On each side of the inequality, we take the Steklov average at $t$. Notice that, in fact, the right hand side does not depend on $s$. Writing $u$ for $u(t,\cdot)$ and $u_n$ for $u_n(t,\cdot)$, 
we obtain
\begin{align} \label{eq:subsol p>2 estimate for e(u,u u^p)}
& \quad - \frac{1}{h} \int_t^{t+h} \e_s(u,u u_n^{p-2} \psi^2) ds \\
& \le 
 \left( 16C_{10} \epsilon  - \frac{1}{2C_{10}} + 4 C_{11} \epsilon^{1/2} \right)
\int \psi^2  u_n^{p-2} d\Gamma(u,u)  \nonumber \\       
& \quad 
+ \left( 4C_{10} \epsilon (p-2)^2 - \frac{1}{C_{10}}(p-2) + C_{11} \epsilon^{1/2} \left( (p-2)^2 + \frac{p (p-2)}{4} \right)  \right)
 \int \psi^2  u_n^{p-2} d\Gamma(u_n,u_n)  \nonumber \\
 & \quad + 
 \left[ 8C_{10} 
+ 4 \big( C_2 + C_3  \Psi( \hat\delta r ) \big) \right] \frac{C_1(\epsilon)}{\Psi( \hat\delta r )} \int_{\delta B} u^2 u_n^{p-2} d\mu   \nonumber
\end{align}
This is the analog of Step 1 in \cite[Proof of Theorem 3.11]{LierlSC2}. 

For a positive integer $n$, let $u_n := u \wedge n$, and define a function $\mathcal H_n: \R  \to \R$ by
\begin{align*}
\mathcal{H}_n(v) := \begin{cases} & \frac{1}{p} v^2 (v \wedge n)^{p-2}, \qquad\qquad\qquad\qquad \, \textrm{if } v \leq n, \\
                               & \frac{1}{2} v^2 (v \wedge n)^{p-2} + n^p \left( \frac{1}{p} - \frac{1}{2} \right), \qquad \textrm{if } v > n.
                 \end{cases}
\end{align*} 
Then $\mathcal{H}_n'(v) = v (v \wedge n)^{p-2}$.
For a small real number $h > 0$, let
\[  u_h(t) := \frac{1}{h} \int_t^{t+h} u(s) ds, \quad t \in (a- \Psi(r), a - h),\]
be the Steklov average of $u$.
In this proof, the subscript of the Steklov average will always be denoted as $h$, and $u_h$ should not be confused with the bounded approximation $u_n$.

We will write $u_h(t,\cdot)$ for $u_h(t)$. Note that $u_h \in L^1((a- \Psi(r), a -h) \to \F)$, 
and $\mathcal{H}_n(u(t,\cdot)), \mathcal{H}_n(u_h(t,\cdot)) \in \F_{\mbox{\tiny{loc}}}$ at almost every $t$.
The Steklov average $u_h$ has a strong time-derivative
\[ \frac{\partial}{\partial t} u_h(t,x) = \frac{1}{h} \big( u(t+h,x) - u(t,x) \big). \] 

Let $s_0 = a -\frac{1+\sigma}{2}\Psi(r)$. Following \cite[Proof of Theorem 3.11]{LierlSC2} line  by line, we obtain that for a.e. $t_0 \in I^-_{\sigma'}$, 
for $h$ sufficiently small so that $t_0 + h < a$, and for $J := (s_0, t_0)$, 
\begin{align} \label{eq:steklov subsol estimate lhs}
& \int_X \mathcal H_n( u_h(t_0,\cdot) )  \psi^2 d\mu  \\
& \le 
- \int_J \int_X \frac{\partial u_h (t,\cdot)}{\partial t}  \mathcal H'_n(u_h(t,\cdot)) \psi^2  \chi(t)  d\mu \, dt 
+ \int_J \int_X  \mathcal H_n(u_h) \psi^2 \chi' \, d\mu \, dt \nonumber \\
& \le
- \int_J \frac{1}{h} \int_t^{t+h}\e_s( u(s,\cdot), \mathcal H'_n(u_h(t,\cdot)) \psi^2 ) \chi(t) ds \, dt
+ \int_J \int_X  \mathcal H_n(u_h) \psi^2 \chi' \, d\mu \, dt \nonumber \\
& \le \label{eq:want to let h to 0 part1}
- \int_J \frac{1}{h} \int_t^{t+h}\e_s( u(s,\cdot), [\mathcal H'_n(u_h(t,\cdot)) - \mathcal H'_n(u(t,\cdot))] \psi^2 ) ds \, \chi(t) dt \\
& \quad \label{eq:want to let h to 0 part2}
- \int_J \frac{1}{h} \int_t^{t+h} \e_s( u(s,\cdot) - u(t,\cdot), \mathcal H'_n(u(t,\cdot)) \psi^2 ) ds \, \chi(t) dt \\
& \quad \label{eq:steklov subsol estimate iii}
- \int_J \frac{1}{h} \int_t^{t+h}\e_s( u(t,\cdot), \mathcal H'_n(u(t,\cdot)) \psi^2 ) ds \, \chi(t) dt \\
& \quad \label{eq:steklov subsol estimate iv}
+ \int_J \int_X  \mathcal H_n(u_h) \psi^2 \chi' \, d\mu \, dt.
\end{align}
We will take the limit as $h \to 0$ on both sides of the inequality. As in Step 2 of \cite[Proof of Theorem 3.11]{LierlSC2}, it can be seen that 
\eqref{eq:want to let h to 0 part1} and \eqref{eq:want to let h to 0 part2} go to $0$ as $h \to 0$. As in Step 3 of \cite[Proof of Theorem 3.11]{LierlSC2}, it can be seen that
\[ \lim_{h \to 0} \int_X \mathcal H_n( u_h(t_0,\cdot) )  \psi^2 d\mu = \int_X \mathcal H_n( u(t_0,\cdot) )  \psi^2 d\mu,\]
and
\[ \lim_{h \to 0} \int_J \int_X  \mathcal H_n(u_h) \psi^2 \chi' \, d\mu \, dt = \int_J \int_X  \mathcal H_n(u) \psi^2 \chi' \, d\mu \, dt. \]
We have already estimated the Steklov average in \eqref{eq:steklov subsol estimate iii} in inequality \eqref{eq:subsol p>2 estimate for e(u,u u^p)}. Thus, taking the limit
as $h \to 0$ in \eqref{eq:steklov subsol estimate lhs} - \eqref{eq:steklov subsol estimate iv}, we get
\begin{align}
& \int_X \mathcal H_n( u(t_0,\cdot) )  \psi^2 d\mu  \nonumber \\
& \quad 
- \left( 16 C_{10} \epsilon  - \frac{1}{2C_{10}}   + 4 C_{11} \epsilon^{1/2} \right)
\int_J \int \psi^2  u_n^{p-2} d\Gamma(u,u) \chi(t) dt \nonumber \\       
& \quad 
- \left( 4 C_{10} \epsilon (p-2)^2 - \frac{1}{C_{10}}(p-2)  + C_{11} \epsilon^{1/2} \left( (p-2)^2 + \frac{p (p-2)}{4} \right)  \right)
 \int_J \int \psi^2  u_n^{p-2} d\Gamma(u_n,u_n) \chi(t)  dt \nonumber \\
& \leq
 \left[ 8 C_{10}
+ 4 \big( C_2 + C_3  \Psi( \hat\delta r ) \big) \right] \frac{C_1(\epsilon)}{\Psi( \hat\delta r )} \int_J \int_{\delta B} u^2 u_n^{p-2} d\mu \, \chi(t)  dt   \nonumber \\
& \quad 
+ \int_J \int_X  \mathcal H_n(u) \psi^2 \chi' \, d\mu \, dt.
\end{align}
Finally, we take the supremum over all $t_0 \in I^-_{\sigma'}$ on both sides of the above inequality, and then we let $n$ tend to infinity. This is where we use the assumption that
$\int_{I^-_{\sigma}} \int_{\delta B} u^p d\mu \, dt < \infty$. Multiplying both sides by $p$ and setting  $\epsilon = \frac{c}{p^2}$ for some sufficiently small $c >0$ completes the proof.
\end{proof}

\begin{lemma} \label{lem:estimate subsol p<2}
Let $p \in (1+\eta,2]$ for some small $\eta >0$. Then there exists a cutoff function $\psi \in \mbox{\em CSA}(\Psi,C_0)$ for $B(x,\delta' r)$ in 
$B(x,\delta' r + \hat \delta r)$ and constants $a_1 \in (0,1)$, $A_1, A_2 \in [0,\infty)$ depending on $\eta$, $C_0$, $C_{10}$, $C_{11}$ such that
\begin{align} \label{eq:subsol p<2}
& \sup_{t \in I^-_{\sigma'}} \int u^p \psi^2 d\mu + a_1 \int_{I^-_{\sigma'}} \int \psi^2 d\Gamma(u^{p/2},u^{p/2}) dt \nonumber \\
\leq & \left(\left( A_1(1 + C_2) \frac{1}{ \Psi( \hat\delta r )} + A_2 C_3 \right) p^{\beta_2}  + \frac{2}{\hat\sigma \Psi(r)} \right)  \int_{Q_{\sigma,\delta}^-} u^p d\bar\mu.
\end{align}
holds for any locally bounded, non-negative local very weak subsolution $u$ of the heat equation for $L_t$ in $Q = Q^-(x,a,r)$.
\end{lemma}

We omit the proof of Lemma \ref{lem:estimate subsol p<2} because it is analogous to the proofs of Lemma \ref{lem:estimate subsol p>2} and Lemma \ref{lem:estimate supsol}.
See also \cite[Proof of Lemma 3.12]{LierlSC2}.

Let $\varepsilon \in(0,1)$ and $u_{\varepsilon} := u + \varepsilon$.

\begin{lemma} \label{lem:estimate supsol}
Let $0 \neq p \in (-\infty, 1-\eta)$ for some $\eta \in (0,1/2)$. 
Then there exists a cutoff function $\psi \in \mbox{\em CSA}(\Psi,C_0)$ for $B(x,\delta' r)$ in $B(x,\delta' r + \hat \delta r)$ such that the following holds for any locally bounded, 
non-negative local very weak supersolution $u$ of the heat equation for $L_t$ in $Q$.
\begin{enumerate}
\item
Let $Q=Q^-(x,a,r)$. If $p<0$, then there are $a_1 \in (0,1)$ and $A_1, A_2 \in [0,\infty)$ depending on $C_0$, $C_{10}$, $C_{11}$ such that
\begin{align} \label{eq:supsol p<0}
& \sup_{t \in I^-_{\sigma'}} \int u_{\varepsilon}^p \psi^2 d\mu + a_1 \int_{I^-_{\sigma'}} \int \psi^2 d\Gamma(u_{\varepsilon}^{p/2},u_{\varepsilon}^{p/2}) dt \nonumber \\
\leq & \left(\left( A_1(1 + C_2) \frac{1}{ \Psi( \hat\delta r )} + A_2 C_3 \right) |p| ( 1+|p|^{\beta_2-1})  + \frac{2}{\hat\sigma \Psi(r)} \right) \int_{Q_{\sigma,\delta}^-} u_{\varepsilon}^p d\bar\mu.
\end{align}
\item
Let $Q=Q^+(x,a,r)$. If $p\in (0,1-\eta)$, then there are $a_1 \in (0,1)$ and $A_1, A_2 \in [0,\infty)$ depending on $\eta$, $C_0$, $C_{10}$, $C_{11}$ such that
\begin{align} \label{eq:supsol 0<p<1}
& \sup_{t \in I^+_{\sigma'}} \int u_{\varepsilon}^p \psi^2 d\mu + a_1 \int_{I^+_{\sigma'}} \int \psi^2 d\Gamma(u_{\varepsilon}^{p/2},u_{\varepsilon}^{p/2}) dt \nonumber \\
\leq & \left(\left( A_1(1 + C_2) \frac{1}{ \Psi( \hat\delta r )} + A_2 C_3 \right) p (1+p^{\beta_2-1}) + \frac{2}{\hat\sigma \Psi(r)} \right)  \int_{Q_{\sigma,\delta}^+} u_{\varepsilon}^p  d\bar\mu.
\end{align}
\end{enumerate}
\end{lemma}

\begin{proof}
First, consider the case $p \in (-\infty,0)$.
Let $\epsilon \in (0,1)$ be small (to be chosen later). Let $\psi \in \mbox{CSA}(\Psi,\epsilon,C_0)$ a cutoff function for $B(x,\delta' r)$ in $B(x,\delta' r + \hat \delta r)$.
By Lemma \ref{lem:SUP s for p<1}, we have for small $\varepsilon >0$ and for large $k \sim (1-p)$, that
\begin{align*}
& \e^{\mbox{\tiny{s}}}_t (u_{\varepsilon}, u_{\varepsilon}^{p-1} \psi^2) \\
& \le   \left(  \frac{2C_{10} \epsilon}{\eta} p^2 + \frac{1}{C_{10}} \left( p - (1-\eta/2) \right) \right) \int \psi^2  u_{\varepsilon}^{p-2} d\Gamma(u_{\varepsilon},u_{\varepsilon})     
 + \frac{8C_{10}}{\eta}  \frac{C_0(\epsilon)}{\Psi( r )} \int \psi u_{\varepsilon}^p d\mu.
\end{align*}
By \eqref{eq:skew identity with p} and Assumption \ref{as:skew1}, we have for $C = 1 + \frac{|2-p|}{|p|}$, 
\begin{align*}
& \quad |\e^{\mbox{\tiny{sym}}}_t(u_{\varepsilon}^p \psi^2,1)| + |\e^{\mbox{\tiny{skew}}}_t(u_{\varepsilon},u_{\varepsilon}^{p-1} \psi^2)| \\
& \leq 
C C_{11} \epsilon^{1/2} \frac{p^2}{4} \int \psi^2 u_{\varepsilon}^{p-2} d\Gamma(u_{\varepsilon},u_{\varepsilon}) 
 +  C \big( C_2 + C_3  \Psi( \hat\delta r )\big) \frac{C_1(\epsilon)}{\Psi( \hat\delta r )} \int_{\delta B} u_{\varepsilon}^p  d\mu.
\end{align*}
By Lemma \ref{lem:e(1,)}, we have 
\begin{align*}
 |\e_t(\varepsilon,u_{\varepsilon}^{p-1} \psi^2)|
& \le
 C_{11} \epsilon^{1/2} \frac{(p-1)^2}{4} 
\int \psi^2 u_{\varepsilon}^{p-2} d\Gamma(u_{\varepsilon},u_{\varepsilon}) 
 +  \big( C_2 + C_3  \Psi( \hat\delta r )\big) \frac{C_1(\epsilon)}{\Psi( \hat\delta r )} \int_{\delta B} u_{\varepsilon}^p  d\mu.
\end{align*}
If $p< -(1-\eta)$, then we choose $\epsilon = \frac{c \eta}{p^2}$ for a sufficiently small constant $c>0$. Otherwise, we let $\epsilon =c\eta^2$. 
Then the proof for the case $p \in (-\infty,0)$ can be completed similarly to the proof of Lemma \ref{lem:estimate subsol p>2}, see also \cite[Lemma 3.13]{LierlSC2}.

For the case $p \in (0,1-\eta)$, let $\chi$ be such that $0 \leq \chi \leq 1$, $\chi = 0$ in $(a + \sigma \Psi(r), \infty)$, $\chi = 1$ in $(-\infty, a + \sigma' \Psi(r))$, and 
\[ 0 \geq \chi' \geq - \frac{2}{\hat\sigma \Psi(r)}. \]
The proof of \eqref{eq:supsol 0<p<1} can be now completed similarly to the case $p \in (-\infty,0)$, we skip the details.
\end{proof}

It is clear from the proofs that in the above lemmas the cutoff functions $\psi$ can be chosen to be in CSA($\Psi,c(p^{-2} \wedge 1),C_0$) for a small enough constant $c = c(\eta)>0$. 

\subsection{Mean value estimates} \label{ssec:Mean value estimates}

In addition to the assumptions made in Section \ref{ssec:time-dependence}, we assume here that 
the reference form $(\e^*,\F)$ satisfies the localized Sobolev inequality SI($\Psi$) on $Y$ up to scale $R_0$.
Let $a_1$ be small enough and $A_1$, $A_2$ large enough so that the estimates of Section \ref{ssec:estim for sub and supsol} hold with these constants.
Set $A_1':=A_1(1 + C_2 )/a_1$ and $A_2' := A_2 C_3 /a_1$.
Define $\delta B$, $I^-$, $I^+$, $I^-_{\sigma}$, $I^+_{\sigma}$, $Q^-_{\sigma,\delta}$, $Q^+_{\sigma,\delta}$ as in Section \ref{ssec:estim for sub and supsol}. In addition, assume that $2 B \subset Y$.

\begin{theorem} \label{thm:MVE subsol p>2}
Suppose {\em Assumptions \ref{as:0} and \ref{as:skew1}, A2-$Y$, VD, CSA($\Psi$)} and {\em SI($\Psi$)} are satisfied on $Y$ up to scale $R_0$.
Let $p> 1+\eta$ for some $\eta >0$. Fix a ball $B = B(x,r)$, $0 < r \leq \frac{R_0}{4}$, with $B(x,4r) \subsetneq B(x,8r) \subset Y$.
Then there exists a constant $A$, depending only on $\eta, \beta_1, \beta_2, C_{\Psi}, \kappa, C_{\mbox{\em \tiny{SI}}}, C_{\mbox{\em \tiny{VD}}}, C_0, C_{10}, C_{11}$, such that, for any $a \in \R$, 
any $0 < \sigma' < \sigma \le 1$, $0 < \delta' < \delta 
\le 1$, and any non-negative local very weak subsolution $u$ of the heat equation for $L_t$ in $Q = Q^-(x,a,r)$, we have
\begin{align} \label{eq:MVE subsol p>2}
 \sup_{Q^-_{\sigma',\delta'}} \{ u^p \}  
& \leq  \left[  \big(A_1' + A_2' \Psi((\delta - \delta') r)\big)  (\delta - \delta')^{-\beta_2} p^{\beta_2} + (\sigma - \sigma')^{-1} \right]^{\frac{2\kappa - 1}{\kappa -1}} \nonumber \\
& \quad \frac{A}{\Psi(r) \mu(B)}  
 \int_{Q^-_{\sigma,\delta}} u^p d\bar \mu.
\end{align}
\end{theorem}

\begin{proof}
First, consider the case $p\geq 2$. For a ball $B_R = B(x,R)$, let $E(B_R) = C_{\mbox{\tiny{SI}}} \Psi(R) V(x,R)^{-1+\frac{1}{\kappa}}$ be the prefactor in the Sobolev inequality \eqref{eq:sobolev}. Consider $0 \leq v \in \F_{\mbox{\tiny{loc}}}(B)$ and let $v_n = v \wedge n$. By Lemma \ref{lem:SUP s}, we have $v_n^q \in \F_{\mbox{\tiny{loc}}}(B)$ for all $q \geq 1$. 

Let $0 < \delta_1 < \delta_0 \le 1$ and $\hat\delta_0 := \delta_0 - \delta_1$.
Let $\psi \in \mbox{CSA}(\Psi,\epsilon,C_0)$ be the cutoff function for $B(x,\delta_1 r)$ in $B(x,\delta_1 r + \hat \delta_0 r)$ provided by Lemma \ref{lem:estimate subsol p>2}. 
We now apply the H\"older inequality, the Sobolev inequality on $B_{\delta_0 r}$ with $f=\psi v_n$, \eqref{eq:Gamma(fg)} and CSA($\Psi,\epsilon,C_0)$. We get
\begin{align*}
& \int_{B(x,\delta_1 r)} v_n^{2(2-\frac{1}{\kappa})} d\mu  \nonumber \\
& \leq  \left( \int_{B(x,\delta_1 r)} v_n^{2\kappa} d\mu \right)^{1/\kappa}  \left( \int_{B(x,\delta_1 r)} v_n^2 d\mu \right)^{1 - \frac{1}{\kappa}} \nonumber \\
& \leq  E(B(x,\delta_0 r)) \left(\int_{B(x,\delta_0 r)} d\Gamma(\psi v_n, \psi v_n)\right) \left( \int_{B(x,\delta_1 r)} v_n^2 d\mu \right)^{1 - \frac{1}{\kappa}} \nonumber \\
& \leq  E(B(x,\delta_0 r)) \left( 2 \int_{B(x,\delta_0 r)} \psi^2 d\Gamma(v_n,v_n) + 2 \int_{B(x,\delta_0 r)} v_n^2 d\Gamma(\psi,\psi) \right) \\
& \quad \left( \int_{B(x,\delta_1 r)} v_n^2 d\mu \right)^{1 - \frac{1}{\kappa}} \nonumber \\
& \leq 2 E(B(x,\delta_0 r)) \left( \left( 1 + \epsilon \right) \int_{B(x,\delta_0 r)} \psi^2 d\Gamma(v_n,v_n) + \frac{C_0(\epsilon)}{\Psi(\hat \delta_0 r)} \int_{B(x,\delta_0 r)} v_n^2 d\mu \right) \\
& \quad  \left( \int_{B(x,\delta_1 r)} v_n^2 d\mu \right)^{1 - \frac{1}{\kappa}}.
\end{align*}
Letting $n \to \infty$, we obtain
\begin{align} \label{eq:HoelderSobolev}
& \int_{B(x,\delta_1 r)} v^{2(2-\frac{1}{\kappa})} d\mu  \nonumber \\
& \leq 2 E(B(x,\delta_0 r)) \left(  (1 + \epsilon)\int_{B(x,\delta_0 r)} \psi^2 d\Gamma(v,v) + \frac{C_0(\epsilon)}{\Psi(\hat \delta_0 r)} \int_{B(x,\delta_0 r)} v^2 d\mu \right) \nonumber \\
& \quad \left( \int_{B(x,\delta_1 r)} v^2 d\mu \right)^{1 - \frac{1}{\kappa}}.
\end{align}

Now let $u \in L^2_{\mbox{\tiny{loc}}}(I \to \F;B)$ be a non-negative local very weak subsolution of the heat equation in $Q$.
Then for almost every $t \in I$, $v := u(t,\cdot)$ is in $\F_{\mbox{\tiny{loc}}}(B)$ and satisfies \eqref{eq:HoelderSobolev}. Let $0 < \sigma_1 < \sigma_0 \leq 1$ and integrate \eqref{eq:HoelderSobolev} over $I^-_{\sigma_1}$. Applying then the 
H\"older inequality to the time integral yields
\begin{align} \label{eq:HoelderSobolev integrated}
& \int_{I_{\sigma_1}^-} \int_{\delta_1 B} u^{2\theta} d\mu \,dt \nonumber \\
& \leq 
 2 E(\delta_0 B) \left( (1 + \epsilon) \int_{I^-_{\sigma_1}} \int_{\delta_0 B} \psi^2 d\Gamma(u,u)dt + \frac{C_0(\epsilon)}{\Psi(\hat \delta_0 r)} \int_{I^-_{\sigma_1}} \int_{\delta_0 B} u^2 d\mu \, dt\right) \nonumber \\
& \quad
  \sup_{t \in I^-_{\sigma_1}} \left( \int_{\delta_0 B} \psi^2 u^2 d\mu \right)^{1 - \frac{1}{\kappa}},
\end{align}
where $\theta=2-\frac{1}{\kappa}$.
Note that the right hand side of \eqref{eq:HoelderSobolev integrated} is finite by Lemma \ref{lem:estimate subsol p>2} (applied with $p=2$). 
Hence the left hand side is finite and this means that $u^{\theta}$ is in $L^2(I^-_{\sigma_1} \times \delta_1 B)$, which is the prerequisite to apply 
Lemma \ref{lem:estimate subsol p>2} with $p=2\theta$ in the next step.

Let $0 < \sigma_2 < \sigma_1$ and $0 < \delta_2 < \delta_1$.
Applying Lemma \ref{lem:estimate subsol p>2} with $p = 2\theta$, we obtain that there exists a cutoff function in CSA($\Psi,C_0$) for $B(x,\delta_2 r + (\delta_1 - \delta_2) r)$ in $\delta_1 B$ with which we can repeat the argument above to obtain that $u^{\theta \cdot \theta} \in L^2(I^-_{\sigma_2} \times \delta_2 B)$.
Iteratively, we obtain that, for any strictly decreasing sequences $(\sigma_i)$, $0 < \sigma_{i+1} < \sigma_i \leq 1$, and $(\delta_i)$, $0 < \delta_{i+1} < \delta_i \leq 1$, we have $\int_{Q_{\sigma_{i+1}, \delta_{i+1}}^-} u^{2\theta^{i+1}} d\mu\, dt < \infty$. Therefore, for $p$ as in Theorem \ref{thm:MVE subsol p>2},
\begin{align} \label{eq:u^pq in L^2}
\int_{Q_{\sigma,\delta}^-} u^{pq} d\mu\, dt < \infty,
\quad \mbox{ for arbitrary } \sigma,\delta \in (0,1), \mbox{ and } q \geq 1.
\end{align}

Now we pick specific sequences $(\sigma_i)$, $(\delta_i)$ and $(q_i)$ with the aim of applying  \eqref{eq:HoelderSobolev integrated}, Lemma \ref{lem:estimate subsol p>2} and CSA($\Psi$) iteratively. 
Let $\sigma',\sigma,\delta',\delta$ be as in the Theorem.
Set $\hat\sigma_i = (\sigma - \sigma') 2^{-i-1}$ so that $\sum_{i=0}^{\infty} \hat\sigma_i = \sigma - \sigma'$. Set also $\sigma_0 = \sigma$, $\sigma_{i+1} = \sigma_i - \hat\sigma_i = \sigma - \sum_{j=0}^i \hat\sigma_j$.
Set $\hat\delta_i = (\delta - \delta') 2^{-i-1}$ so that $\sum_{i=0}^{\infty} \hat\delta_i = \delta - \delta'$. Set also $\delta_0 = \delta$, $\delta_{i+1} = \delta_i - \hat\delta_i = \delta - \sum_{j=0}^i \hat\delta_j$.

By Lemma \ref{lem:nu} and \eqref{eq:beta},
\begin{align} \label{eq:Psi(delta_i)}
 \left( \frac{\mu(B)}{\mu(\delta_i B)} \right)^{1-\frac{1}{\kappa}} \leq C \frac{\Psi(r)}{\Psi(\delta_i r)}.
\end{align}

Let $\psi_i \in \mbox{CSA}(\Psi,\epsilon,C_0)$ be the cutoff function for $B(x,\delta_{i+1}r)$ in $B(x,\delta_{i+1}r + \hat \delta_i r)$ that is given by Lemma \ref{lem:estimate subsol p>2}. Here, $\epsilon = c (p \theta^i)^{-2}$ for some small fixed constant $c>0$ that depends at most on $C_{10}$ and $C_{11}$.

Similar to how we obtained \eqref{eq:HoelderSobolev integrated} but with $u^{p \theta^{i}/2}$ in place of $u$, we get
\begin{align*}
&   \int \int_{Q^-_{\sigma_{i+1},\delta_{i+1}}} u^{p \theta^{i+1}} d\bar\mu \\
& \le 
2 E(\delta_i B) \left( (1+\epsilon) \int_{I^-_{\sigma_{i+1}}} \int_{\delta_i B} \psi_i^2 d\Gamma( u^{p\theta^i/2},u^{p\theta^i/2}) dt 
 +   \frac{C_0(\epsilon)}{\Psi(\hat\delta_i r)}  \int_{I^-_{\sigma_{i+1}}}  \int_{\delta_i B} u^{p\theta^i} d\bar\mu \right) \\
 & \quad \left( \sup_{t \in I^-_{\sigma_{i+1}}} \int_{\delta_i B} \psi_i^2 u^{p\theta^i} d\mu \right)^{1 - \frac{1}{\kappa}}.
 \end{align*}
 By Lemma \ref{lem:estimate subsol p>2} together with \eqref{eq:u^pq in L^2}, and by \eqref{eq:Psi(delta_i)}, the right hand side is no more than 
\begin{align*}
& \quad 
\frac{C \Psi(\delta_i r)}{ \mu(\delta_i B)^{1-\frac{1}{\kappa}}}
\Bigg( \left[  \left( \left( A_1' \frac{1}{\Psi(\hat\delta_i r)} + A_2' \right) (p\theta^i)^{\beta_2}  + \frac{2}{\hat\sigma_i \Psi(r)}  \right)
+ \frac{C_0(\epsilon)}{\Psi(\hat\delta_i r)} \right] \\
&\quad 
\int \int_{Q^-_{\sigma_i,\delta_i}} u^{p\theta^i} d\bar\mu \Bigg)^{\theta} \\
& \leq  
\frac{C}{[ \Psi(r) \mu(B)]^{1 - \frac{1}{\kappa}}}
\Bigg( \left[ \left( \left( (A_1' + A_2' \Psi(\hat\delta_i r))  \frac{\Psi(r)}{\Psi(\hat\delta_i r)} \right) (p\theta^i)^{\beta_2}+ \frac{2}{\hat\sigma_i}  \right) 
+ \frac{C_0(\epsilon) \Psi(r)}{\Psi(\hat\delta_i r)} \right] \\
& \quad 
 \int \int_{Q^-_{\sigma_i,\delta_i}} u^{p\theta^i} d\bar\mu \Bigg)^{\theta} \\
& \leq 
 \frac{1}{[ \Psi(r) \mu(B)]^{1 - \frac{1}{\kappa}}}
\left( C^{i+1} \left(  \big(A_1' + A_2' \Psi(\hat\delta r)\big)   \hat\delta^{-\beta_2}   p^{\beta_2} + \hat\sigma^{-1}  \right)
\int \int_{Q^-_{\sigma_i,\delta_i}} u^{p\theta^i} d\bar\mu \right)^{\theta},
\end{align*}
where the constant $C \in (0,\infty)$ (which may change from line to line) depends at most on $\theta$, $\beta_1$, $\beta_2$, $\kappa$, $C_{\Psi}$, $C_{\mbox{\tiny{SI}}}$, $C_{\mbox{\tiny{VD}}}$, $C_0$, $C_{10}$, $C_{11}$.
Hence,
\begin{align*}
& \left( \int \int_{Q^-_{\sigma_{i+1},\delta_{i+1}}} u^{p \theta^{i+1}} d\bar\mu \right)^{ \theta^{-i-1} } \\
& \leq 
\left( \frac{1}{[ \Psi(r) \mu(B)]^{1 - \frac{1}{\kappa}}} \right)^{\sum \theta^{-1-j}}
 C^{ \sum (j+1) \theta^{-j} }   \\
& \quad
\left[  \left(\big(A_1' + A_2' \Psi(\hat\delta r)\big) \hat\delta^{-\beta_2} p^{\beta_2} + \hat\sigma^{-1}  \right)  \right]^{ \sum \theta^{-j} } 
 \int_{Q^-_{\sigma,\delta}} u^p d\bar\mu,
\end{align*}
where all the summations are taken from $j=0$ to $j=i$. Letting $i$ tend to infinity, we obtain
\begin{align*}
 \sup_{Q^-_{\sigma',\delta'}} \{ u^p \} 
& \leq  \left[ \left( \big(A_1' + A_2' \Psi(\hat\delta r)\big)   \hat\delta^{-\beta_2} p^{\beta_2} + \hat\sigma^{-1} \right)  \right]^{\frac{2\kappa-1}{\kappa -1}} 
 \frac{C}{\Psi(r) \mu(B)}  \int_{Q^-_{\sigma,\delta}} u^p d\bar\mu.
\end{align*}
This yields \eqref{eq:MVE subsol p>2}.

At this stage of the proof, Corollary \ref{cor:u loc bounded} already follows. Thus, in the case $1+\eta < p < 2$ the assertion can be proved similarly, by using Lemma \ref{lem:estimate subsol p<2} and Corollary \ref{cor:u loc bounded}.
\end{proof}

\begin{corollary} \label{cor:u loc bounded}
Under the same hypotheses as Theorem \ref{thm:MVE subsol p>2}. Then any non-negative local very weak subsolution $u$ for $L_t$ in $Q$ is 
locally bounded.
In particular, any local very weak solution of $u$ for $L_t$ in $Q$ is locally bounded.
\end{corollary}

\begin{proof}
The first statement follows from the proof of Theorem \ref{thm:MVE subsol p>2}. By \cite[Proposition 3.4]{LierlSC2}, for any local very weak solution $u$ of the heat equation, 
$|u|$ is a non-negative local very weak subsolution.
\end{proof}

\begin{theorem} \label{thm:MVE subsol 0<p<2}
Suppose {\em Assumptions \ref{as:0} and \ref{as:skew1}, A2-$Y$, VD, CSA($\Psi$)} and {\em SI($\Psi$)} are satisfied on $Y$ up to scale $R_0$.
Let $0 < p < 2$. Fix a ball $B = B(x,r)$, $0 < r \leq \frac{R_0}{4}$, with $B(x,4r) \subsetneq B(x,8r) \subset Y$.
Then there exists a constant $A$, depending only on $C_{\Psi}$, $\beta_1$, $\beta_2$, $\kappa$, $C_{\mbox{\em \tiny{SI}}}$, $C_{\mbox{\em \tiny{VD}}}$, $C_0$, $C_{10}$, $C_{11}$, such that, for any $a \in \R$, 
any $0 < \sigma' < \sigma \le 1$, $0 < \delta' < \delta 
\le 1$, and any non-negative local very weak subsolution $u$ of the heat equation for $L_t$ in $Q = Q^-(x,a,r)$, we have
\begin{align*}
 \sup_{Q^-_{\sigma',\delta'}} \{ u^p \}  
& \le \left(\frac{4}{3}\right)^{\beta_2 \frac{2\kappa - 1}{2(\kappa -1)} 4/p}  \left[  \big( A_1' + A_2' \Psi((\delta - \delta') r)\big)  (\delta - \delta')^{-\beta_2} 2^{\beta_2} + (\sigma - \sigma')^{-1}   \right]^{\frac{2\kappa - 1}{\kappa -1}} \nonumber \\
& \quad \frac{A}{\Psi(r) \mu(B)}  
 \int_{Q^-_{\sigma,\delta}} u^p d\bar \mu.
\end{align*}
\end{theorem}

\begin{proof}
We follow \cite[Theorem 2.2.3, Theorem 5.2.9]{SC02}. 
Let $D_1 := 2^{\beta_2} \big( A_1' + A_2' \Psi((\delta - \delta') r) \big)$.
By \eqref{eq:MVE subsol p>2} with $p=2$, we have for any $0 < \sigma' < \sigma \le 1$, $0 < \delta' < \delta \le 1$,
\begin{align*}
 \sup_{Q^-_{\sigma',\delta'}} u
& \le 
 \left[  D_1 (\delta - \delta')^{-\beta_2}  + (\sigma - \sigma')^{-1}  \right]^{\frac{2\kappa - 1}{2(\kappa -1)}} \left(\frac{A}{\Psi(r) \mu(B)}\right)^{1/2}  
 \left(\int_{Q^-_{\sigma,\delta}} u^2 d\bar \mu \right)^{1/2} \\
& \le 
 \left[  D_1 (\delta - \delta')^{-\beta_2}  + (\sigma - \sigma')^{-1}  \right]^{\frac{2\kappa - 1}{2(\kappa -1)}} J  \sup_{Q^-_{\sigma,\delta}} u^{\frac{2-p}{2}},
\end{align*}
where $J := \left(\frac{A}{\Psi(r) \mu(B)}\right)^{1/2}  
 \left(\int_{Q^-_{\sigma,\delta}} u^p d\bar \mu \right)^{1/2}$. 
 
Set $\delta_0 := \delta'$, $\delta_{i+1} := \delta_i + (\delta - \delta_i)/4$. Then $(\delta - \delta_i) = \left(\frac{3}{4}\right)^i (\delta-\delta')$. Similarly, we set
$\sigma_0 := \sigma'$, $\sigma_{i+1} := \sigma_i + (\sigma - \sigma_i)/4$. 
Applying the above inequality for each $i$ yields
\begin{align*}
 \sup_{Q^-_{\sigma_i,\delta_i}} u
& \le 
 \left(\frac{4}{3}\right)^{i \beta_2 \frac{2\kappa - 1}{2(\kappa -1)}} \left[  D_1  (\delta - \delta')^{-\beta_2}  + (\sigma - \sigma')^{-1}  \right]^{\frac{2\kappa - 1}{2(\kappa -1)}} J  \sup_{Q^-_{\sigma_{i+1},\delta_{i+1}}} u^{\frac{2-p}{2}},
\end{align*}
Iterating this inequality, we get for $i=1,2,\ldots$,
\begin{align*}
 \sup_{Q^-_{\sigma',\delta'}} u
& \le 
 \left(\frac{4}{3}\right)^{\beta_2 \frac{2\kappa - 1}{2(\kappa -1)} \sum_{j=0}^{i-1} j (1-p/2)^j} \left[ \left[  D_1  (\delta - \delta')^{-\beta_2}  + (\sigma - \sigma')^{-1}  \right]^{\frac{2\kappa - 1}{2(\kappa -1)}}  J \right]^{\sum_{j=0}^{i-1} (1-p/2)^j}  \sup_{Q^-_{\sigma_i,\delta_i}} u^{(1-p/2)^i}.
\end{align*}
Letting $i \to \infty$, and noting that  $\lim_{i \to \infty} \sup_{Q^-_{\sigma_i,\delta_i}} u^{(1-p/2)^i} = 1$, we get
\begin{align*}
 \sup_{Q^-_{\sigma',\delta'}} u
& \le 
 \left(\frac{4}{3}\right)^{\beta_2 \frac{2\kappa - 1}{2(\kappa -1)} 4/p^2} \left[ \left[  D_1  (\delta - \delta')^{-\beta_2}  + (\sigma - \sigma')^{-1}  \right]^{\frac{2\kappa - 1}{2(\kappa -1)}}  J \right]^{2/p}.
\end{align*}
Rasing each side to power $p$ we get the desired inequality.
\end{proof}

The next theorem can be proved analogously to the proof of Theorem \ref{thm:MVE subsol p>2}, by applying Lemma \ref{lem:estimate supsol} instead of Lemma \ref{lem:estimate subsol p>2}.

\begin{theorem} \label{thm:MVE supsol}
Suppose {\em Assumptions \ref{as:0} and \ref{as:skew1}, A2-$Y$, VD, CSA($\Psi$)} and {\em SI($\Psi$)} are satisfied on $Y$ up to scale $R_0$.
Let $0 \neq p \in (-\infty, 1-\eta)$ for some small $\eta \in (0,1)$.
Fix a ball $B = B(x,r)$, $0 < r \leq \frac{R_0}{4}$, with $B(x,4r) \subsetneq B(x,8r) \subset Y$. Let $a \in \R$.
Let $u \in \F_{\mbox{\emph{\tiny{loc}}}}(Q)$ be any non-negative local very weak supersolution of the heat equation for $L_t$ in $Q$. Suppose that $u$ is locally bounded. Let $\varepsilon \in(0,1)$ and $u_{\varepsilon} := u + \varepsilon$.
Let $0 < \sigma' < \sigma \leq 1$, $0 < \delta' < \delta \leq 1$.
\begin{enumerate}
\item
Let $Q=Q^-(x,a,r)$.
If $p \in (-\infty,0)$, then there exists a constant $A$, depending only on $\beta_1, \beta_2, C_{\Psi}, \kappa, C_{\mbox{\em \tiny{SI}}}, C_{\mbox{\em \tiny{VD}}}, C_0, C_{10}, C_{11}$, 
such that
\begin{align*}
 \sup_{Q^-_{\sigma',\delta'}} \{ u_{\varepsilon}^p \}  
& \leq  
\left[  \big(A_1' + A_2' \Psi((\delta - \delta') r)\big)  (\delta - \delta')^{-\beta_2}  (1+|p|^{\beta_2}) + (\sigma - \sigma')^{-1} \right]^{\frac{2\kappa-1}{\kappa-1}} \\
& \quad 
\frac{A}{\Psi(r) \mu(B)} 
 \int_{Q^-_{\sigma,\delta}} u_{\varepsilon}^p d\bar \mu.
\end{align*}
\item
Let $Q=Q^+(x,a,r)$.
If $p \in (0,1-\eta)$, then there exists a constant 
$A$, depending only on $\eta, \beta_1, \beta_2, C_{\Psi}, \kappa, C_{\mbox{\em \tiny{SI}}}, C_{\mbox{\em \tiny{VD}}}, C_0, C_{10}, C_{11}$, such that
\begin{align*}
 \sup_{Q^+_{\sigma',\delta'}} \{ u_{\varepsilon}^p \}  
& \leq  
\left[  \big(A_1' + A_2' \Psi((\delta - \delta') r)\big)  (\delta - \delta')^{-\beta_2} (1+p^{\beta_2}) + (\sigma - \sigma')^{-1} \right]^{\frac{2\kappa-1}{\kappa-1}} \\
& \quad 
\frac{A}{\Psi(r) \mu(B)} 
 \int_{Q^+_{\sigma,\delta}} u_{\varepsilon}^p d\bar \mu.
\end{align*}
\end{enumerate}
\end{theorem}

\section{Parabolic Harnack inequality} \label{sec:PHI}

\subsection{The log lemma and an abstract lemma}
\label{ssec:bombieri}
Let $(\e_t,\F)$, $t \in \R$, be as in Section \ref{ssec:time-dependence}. In this section, we suppose that Assumptions \ref{as:0} - \ref{as:VD+PI} are satisfied.

Let $a_1$ be small and $A_1$, $A_2$ large enough so that the estimates of Section \ref{ssec:estim for sub and supsol} hold with these constants.
Recall that for $\varepsilon \in(0,1)$, $u_{\varepsilon}:= u + \varepsilon$.

\begin{lemma}\label{lem:log u estimate} 
Suppose that {\em Assumptions \ref{as:0} - \ref{as:VD+PI}} are satisfied.
Let $0 < \sigma < 1$, $0 < \delta < 1$ and $\hat\delta := 1-\delta$. There exists a constant $C \in(0,\infty)$ such that, for any $a \in\R$, $0 < r \leq R_0$, 
$B=B(x,r) \subset Y$, and any non-negative, locally bounded function $u \in \mathcal C_{\mbox{\em \tiny{loc}}}(I \to L^2(B))$ which is a local very weak supersolution of the heat equation for
$L_t$ in $Q$, there exists a constant $c \in (0,\infty)$ depending on $u(a,\cdot)$, such that
\begin{enumerate}
\item
 \[  \bar\mu( \{ (t,z) \in K_+ : \log u_{\varepsilon} < - \lambda - c \} ) \leq C \Psi(r) \mu(B) \lambda^{-1}, \qquad \forall \lambda>0, \]
where $Q = Q^+(x,a,r)$, $K_+ = (a, a + \sigma \Psi(r)) \times \delta B$, and
\item
 \[  \bar\mu( \{ (t,z) \in K_- : \log u_{\varepsilon} > \lambda - c \} ) \leq C \Psi(r) \mu(B) \lambda^{-1},  \qquad \forall \lambda>0, \]
where $Q = Q^-(x,a,r)$, $K_- = (a - \sigma \Psi(r),a) \times \delta B$. 
\end{enumerate}
The constant $C$ depends on $C_{\mbox{\em \tiny{VD}}}$, $C_{\mbox{\em \tiny{PI}}}$, $C_{\Psi}$, $\beta_1$, $\beta_2$, $C_0$, $C_{10}$, $C_{11}$, and upper bounds on $(1+C_2 + C_4)+(C_3+C_5)\Psi(\hat\delta r)$, $\frac{1}{\delta}$ and $\frac{1}{\hat\delta}$.
\end{lemma}

\begin{proof}
For $h > 0$, let
\[ u_{\varepsilon,h}(t) := \frac{1}{h} \int_t^{t+h} u_{\varepsilon}(\tau) d\tau \]
be the Steklov average of $u_{\varepsilon}$. Let $\epsilon \in (0,1)$ (to be chosen later), and let $\psi \in \mbox{CSA}(\Psi,\epsilon,C_0)$ be the cutoff function for
 $B(x,\delta r)$ in $B(x,r')$ given by 
Theorem \ref{thm:weighted PI}, for some $r' \in (\delta r, r)$. 

Using the fact that the Steklov average has a strong time-derivative and the assumption that $u$ is local very weak supersolution, we obtain
\begin{align*}
 & - \frac{d}{dt} \int \log u_{\varepsilon,h}(t) \psi^2 d\mu \\
& = 
- \frac{1}{h} \int [ u(t+h) - u(t) ] \frac{1}{u_{\varepsilon,h}(t)} \psi^2 d\mu \nonumber \\
& \leq 
 \frac{1}{h} \int_t^{t+h} \e_s \left( u(s),\frac{1}{u_{\varepsilon,h}(t)} \psi^2 \right) ds \nonumber \\
& =  
 \frac{1}{h} \int_t^{t+h} \e_s \left( u(s),\frac{1}{u_{\varepsilon,h}(t)} \psi^2  - \frac{1}{u_{\varepsilon}(t)} \psi^2 \right) ds \\
& \quad 
+  \frac{1}{h} \int_t^{t+h} \e_s \left( u(s) - u(t),\frac{1}{u_{\varepsilon}(t)} \psi^2 \right) ds \\
& \quad 
+  \frac{1}{h} \int_t^{t+h} \e_s \left( u_{\varepsilon}(t),\frac{1}{u_{\varepsilon}(t)} \psi^2 \right) 
                                 -  \e_s \left( \varepsilon,\frac{1}{u_{\varepsilon}(t)} \psi^2 \right)ds  \\
& = 
f_h(t) + \hat f_h(t) + g_h(t).
\end{align*}
It can be shown that $f_h(t)$ and $\hat f_h(t)$ tend to $0$ in $L^1((a,a+\sigma \Psi(r)) \to \R)$ as $h \to 0$. Next, we will estimate $g_h(t)$.
We write $u_{\varepsilon} = u_{\varepsilon}(t)$.
Applying \eqref{eq:chain rule for Gamma}, \eqref{eq:CS}, \eqref{eq:C_10} and CSA($\Psi,\epsilon,C_0$), we have for any $k_0>0$ that
\begin{align*}
\e_s^{\mbox{\tiny{s}}}(u_{\varepsilon},u_{\varepsilon}^{-1} \psi^2)
& =  \int 2\psi d\Gamma_s(\log(u_{\varepsilon}),\psi) 
   - \int \psi^2 d\Gamma_s(\log(u_{\varepsilon}),\log(u_{\varepsilon})) \\
&\le
4k_0 \int  d\Gamma(\psi,\psi) 
 - \left( \frac{1}{C_{10}} - \frac{C_{10}}{k_0} \right) \int \psi^2 d\Gamma(\log(u_{\varepsilon}),\log(u_{\varepsilon})) \\
& \le
- \left( \frac{1}{C_{10}} - \frac{C_{10}}{k_0}  \right) \int \psi^2 d\Gamma(\log(u_{\varepsilon}),\log(u_{\varepsilon})) + \frac{4k_0 C_{10} C_0(\epsilon)}{\Psi(\hat \delta r)} \int_B d\mu.
\end{align*}
By Assumption \ref{as:skew1} and Assumption \ref{as:skew2}, we have
\begin{align*}
& \quad \e_s^{\mbox{\tiny{skew}}}(u_{\varepsilon},u_{\varepsilon}^{-1} \psi^2) + \e_s^{\mbox{\tiny{sym}}}(\psi^2,1)  \\
& \le 
C_{11} \epsilon^{1/2} \int \psi^2 d\Gamma(\log(u_{\varepsilon}),\log(u_{\varepsilon})) 
+ \left( C_2 + C_4 + (C_3 + C_5) \Psi(\hat\delta r) \right) \frac{C_1(\epsilon)}{\Psi(\hat \delta r)} \int_B d\mu.
\end{align*}
By Lemma \ref{lem:e(1,)},
\begin{align*}
& \quad - \e_s( \varepsilon,u_{\varepsilon}^{-1} \psi^2)  \\
& \le 
C_{11} \epsilon^{1/2} \frac{1}{4} \int \psi^2 d\Gamma(\log(u_{\varepsilon}),\log(u_{\varepsilon})) 
+ \left( C_2 + C_3 \Psi(\hat\delta r) \right) \frac{C_1(\epsilon)}{\Psi(\hat \delta r)} \int_B d\mu.
\end{align*}

Hence, making a suitable choice of $k_0$ (large) and $\epsilon$ (small), we find that for sufficiently large $k >1$ depending on $C_0$, $C_{10}$, $C_{11}$ and an upper bound for $(1+ C_2 + C_4 + (C_3 + C_5)\Psi(\hat \delta r))$, we have
\begin{align} \label{eq:lem5.4.1a}
 & - \frac{d}{dt} \int \log u_{\varepsilon,h}(t) \psi^2 d\mu  +  \frac{1}{k} \int \psi^2 d\Gamma(\log(u_{\varepsilon}),\log(u_{\varepsilon})) \nonumber \\
& \leq  
f_h(t) + \hat f_h(t) + \left( 1+C_2 + C_4 + (C_3 + C_5) \Psi(\hat\delta r) \right) \frac{k}{\Psi(\hat \delta r)} \mu(B).
\end{align}
Let 
\[ W(t) :=  -\frac{\int \log u_{\varepsilon}(t) \psi^2 d\mu }{ \int \psi^2 d\mu} \quad \mbox{ and} \quad W_h(t) :=  -\frac{\int \log u_{\varepsilon,h}(t) \psi^2 d\mu }{ \int \psi^2 d\mu}. \]
By the weighted Poincar\'e inequality of Theorem \ref{thm:weighted PI}, there is a constant $C_{\mbox{\tiny{wPI}}} \in (0,1)$ such that, for a.e. $t \in I$, 
\begin{align*}
\int |-\log u_{\varepsilon}(t) - W(t)|^2 \psi^2 d\mu  \leq  C_{\mbox{\tiny{wPI}}} \, \Psi(r) \int \psi^2 d\Gamma(\log u_{\varepsilon}(t), \log u_{\varepsilon}(t)).
\end{align*}
The constant $C_{\mbox{\tiny{wPI}}}$ depends only on $C_{\mbox{\tiny{PI}}}$ and an upper bound on $\frac{\mu(B)}{\mu(\delta B)}$.

This and \eqref{eq:lem5.4.1a} yield
\begin{align*} 
& \quad  \frac{d}{dt} W_h(t) + \frac{1}{C \Psi(r) \mu(B)}  \int_{\delta B} |-\log u_{\varepsilon}(t) - W(t)|^2 \psi^2 d\mu \\
& \le
C' \frac{f_h(t) + \hat f_h(t)}{\int \psi^2 d\mu} + C' \left(1+ C_2 + C_4 + (C_3 + C_5) \Psi(\hat\delta r) \right)  \frac{k}{\Psi(\hat \delta r)},
\end{align*}
for some constants $C,C' \in (0,\infty)$ that depend only on $k$, $C_{\mbox{\tiny{VD}}}$ $C_{\mbox{\tiny{PI}}}$, $C_0$ and an upper bound on $\frac{1}{\delta}$.
Notice that by \eqref{eq:beta}, $\frac{1}{\Psi(\hat \delta r)} \le \frac{C''}{\Psi(r)}$ for some constant $C''$ depending only on $C_{\Psi}$, $\beta_2$, and on an upper bound on $\frac{1}{\hat\delta}$. 

Now the proof can be completed easily by following \cite[Lemma 4.12]{LierlSC2} line by line, except for replacing $r^2$ by $\Psi(r)$ and applying \eqref{eq:beta} where needed. 
\end{proof}

Let $U_{\delta}$ be a collection of measurable subsets of $X$ such that $U_{\delta'} \subset U_{\delta}$ for any $0 < \delta' < \delta \leq 1$. 
Let $J_{\sigma}$ be a collection of intervals in $\R$ such that $J_{\sigma'} \subset J_{\sigma}$ for any $0 < \sigma' < \sigma \leq 1$. 

\begin{lemma} \label{lem:Bombieri}
Fix  $\sigma^*,\delta^* \in (0,1)$.
Let $f$ be a positive measurable function on $J_1 \times U_1$ which satisfies
\[ \sup_{J_{\sigma'} \times U_{\delta'}} f \leq  \left(\left( \frac{C}{ (\delta - \delta')^{\gamma_1}} + \frac{C}{(\sigma - \sigma')^{\gamma_2}} \right) 
\frac{1}{|J_{1}| \mu(U_{1})} \int_{J_{\sigma}} \int_{U_{\delta}} f^p d\mu \, dt \right)^{\frac{1}{p}} \]
for all $\sigma^* \leq \sigma' < \sigma < 1$, $\delta^* \leq \delta' < \delta < 1$, $p \in (0,1-\eta)$, for some $\gamma_1, \gamma_2 \geq0$, $\eta \in (0,1)$, $C \in (0,\infty)$. 
Suppose further that 
\begin{align} \label{eq:log f > lambda}
 \bar\mu(\{ \log f > \lambda \}) \leq C \frac{|J_1| \mu(U_1) }{\lambda}, \qquad \forall \lambda > 0.
\end{align}
Then there is a constant $A_3 \in [1,\infty)$, depending only on $\sigma^*, \delta^*, \gamma_1, \gamma_2, C$ and a positive lower bound on $\eta$, such that
\[ \sup_{J_{\sigma^*} \times U_{\delta^*}} f \leq A_3. \]
\end{lemma}

\begin{proof}
We follow \cite[Proof of Lemma 3]{Moser71} (see also \cite[Proof of Theorem 4]{BomGiu72}, \cite[Lemma 2.2.6]{SC02}).
Without loss of generality, assume for the proof that $|J_1| \mu(U_1) = 1$.
Define
\[ \phi = \phi(\sigma,\delta) :=  \sup_{J_{\sigma} \times U_{\delta}} f. \]
Decomposing $J_{\sigma} \times U_{\delta}$ into the sets where $\log f > \frac{1}{2}\log(\phi)$ and where $\log f \leq \frac{1}{2}\log(\phi)$, we get from \eqref{eq:log f > lambda} that
\begin{align*}
\int_{J_{\sigma}} \int_{U_{\delta}} f^p d\mu \, dt 
& \leq 
\left(\sup_{J_{\sigma} \times U_{\delta}} f^p \right)  \mu\left( \log f > \frac{1}{2}\log\phi \right) + \phi^{\frac{p}{2}} |J_{\sigma}| \mu(U_{\delta})  \\
& \leq
\phi^p \frac{2C}{\log \phi} + \phi^{\frac{p}{2}}.
\end{align*}
The two terms on the right hand side are equal if 
\[ p = \frac{2}{\log \phi}  \log \left( \frac{\log \phi}{2C} \right). \]
We have $p < 1 - \eta$ if $\phi$ is sufficiently large, that is, if 
\begin{align} \label{eq:A_1 2.2.10}
\phi \geq A_1
\end{align}
for some $A_1$ depending only on $\eta$ (note we can always take $C \geq 1$). 
Hence, for $\phi \geq A_1$, the first hypothesis of the lemma yields
\begin{align*}
\log \phi(\sigma',\delta')
& \leq \frac{1}{p} \log  \left( \frac{C}{ (\delta - \delta')^{\gamma_1}} + \frac{C}{(\sigma - \sigma')^{\gamma_2}} \right)  + \frac{\log \phi}{2} + \frac{\log 2}{p} \\
& \leq \frac{1}{p} \log  \left( \frac{2C}{ (\delta - \delta')^{\gamma_1}} + \frac{2C}{(\sigma - \sigma')^{\gamma_2}} \right)  + \frac{\log \phi}{2} \\
& = \frac{\log \phi}{2} \left[ \frac{\log \left( \frac{2C}{ (\delta - \delta')^{\gamma_1}} + \frac{2C}{(\sigma - \sigma')^{\gamma_2}} \right)}{\log \left( \frac{ \log \phi}{2C} \right)} + 1  \right].
\end{align*}
If 
\begin{align} \label{eq:2.2.12}
 \frac{\log \phi}{2C} \geq \left( \frac{2C}{ (\delta - \delta')^{\gamma_1}} + \frac{2C}{(\sigma - \sigma')^{\gamma_2}} \right)^2,
 \end{align}
then 
\[ \log \phi(\sigma',\delta') \leq \frac{3}{4} \log \phi. \]
On the other hand, if \eqref{eq:2.2.12} or \eqref{eq:A_1 2.2.10} is not satisfied, then
\[ \log \phi(\sigma',\delta') \leq \log \phi \leq \log A_1 + 2C \left( \frac{2C}{ (\delta - \delta')^{\gamma_1}} + \frac{2C}{(\sigma - \sigma')^{\gamma_2}} \right)^2. \]
In all cases, we obtain
\begin{align} \label{eq:bombieri 3/4}
 \log \phi(\sigma',\delta')  \leq \frac{3}{4} \log \phi(\sigma,\delta) + A_2 \left( \frac{C^2}{ (\delta - \delta')^{2\gamma_1}} + \frac{C^2}{(\sigma - \sigma')^{2\gamma_2}}\right)
\end{align}
for some constant $A_2 \in (0,\infty)$ depending only on $\sigma^*, \delta^*, \gamma_1, \gamma_2, C$ and a positive lower bound on $\eta$.
Let $\sigma_j = 1 - \frac{1-\sigma*}{1+j}$ and  
$\delta_j = 1 - \frac{1-\delta*}{1+j}$. Iterating \eqref{eq:bombieri 3/4}, we get
\[ \log \phi(\sigma^*,\delta^*) 
\leq A_2 \sum_{j=0}^{\infty} \left( \frac{3}{4} \right)^j \left( \frac{C^2}{ (\delta_{j+1} - \delta_j)^{2\gamma_1}} + \frac{C^2}{(\sigma_{j+1} - \sigma_j)^{2\gamma_2}} \right) =: A_3 < \infty. \]
\end{proof}

\subsection{Parabolic Harnack inequalities} \label{ssec:PHI main results}
Let $(\e_t,\F)$, $t \in \R$, be as in Section \ref{ssec:time-dependence}. In this section, we suppose that Assumptions \ref{as:0} - \ref{as:VD+PI} are satisfied for an open subset $Y  \subset X$.

Let $B=B(x,r) \subset X$, $a \in \R$. Fix $\delta \in (0,1)$ and let $0 < \tau_1 < \tau_2 < \tau_3 < \tau_4 \leq 1$. 
Set
\begin{align*}
\delta B & = B(x,\delta r), \\
Q &= Q(x,a,r) = (a,a + \Psi(r)) \times B, \\
 Q^- &= (a + \tau_1 \Psi(r), a + \tau_2 \Psi(r)) \times \delta B, \\ 
 Q^+ &= (a + \tau_3 \Psi(r), a + \tau_4 \Psi(r)) \times \delta B.
\end{align*}

\begin{theorem} \label{thm:PHI}
Suppose {\em Assumption \ref{as:0} - \ref{as:VD+PI}} are satisfied. Then the family $(\e_t,\F)$, $t \in \R$, satisfies the parabolic Harnack inequality 
$\mbox{\em PHI($\Psi$)}$ on $Y$ up to scale $R_0$. That is, there is a constant $C_{\mbox{\em \tiny{PHI}}} \in (0,\infty)$  such that for any $a \in \R$, any ball
 $B(x,4r) \subsetneq B(x,8r) \subset Y$, $0 < r < \frac{R_0}{4}$, and any non-negative local weak solution $u$ of the heat equation for $L_t$ in $Q=Q(x,a,r)$, we have
\[ \sup_{Q^-} u \leq C_{\mbox{\em \tiny{PHI}}} \inf_{Q^+} u. \]
The constant $C_{\mbox{\em \tiny{PHI}}}$ depends only on $\delta$, $\tau_1, \tau_2, \tau_3, \tau_4$, $C_{\Psi}$, $\beta_1$, $\beta_2$, $C_{\mbox{\em \tiny{VD}}}$, $C_{\mbox{\em \tiny{PI}}}$, $C_0$, $C_{10}$, $C_{11}$, and an upper bound on $[ (1 + C_2 + C_4) +(C_3+C_5) \Psi((1-\delta)r)]$.
\end{theorem}

\begin{proof}
Let $\varepsilon \in (0,1)$ and $u_{\varepsilon} := u + \varepsilon$.
By Corollary \ref{cor:u loc bounded}, Theorem \ref{thm:MVE supsol}, and Theorem \ref{lem:log u estimate}, we can apply Lemma \ref{lem:Bombieri} to $u_{\varepsilon}$ on 
$(a,a+\tau_2 \Psi(r)) \times \delta B$. We obtain that there is some $c$ such that 
\[ \sup_{Q^-} u_{\varepsilon} e^c \leq C. \]
Similarly, apply Lemma \ref{lem:Bombieri} to $u_{\varepsilon}^{-1}$ on $(a+\tau_2 \Psi(r), a+\tau_4\Psi(r)) \times \delta B$. We obtain that, for the same $c$ as above,
\[ \sup_{Q^+} (u_{\varepsilon} e^c)^{-1} \leq C'. \]
Hence,
\[ \sup_{Q^-} u_{\varepsilon} \leq e^{-c} C  \leq C \frac{C'}{\sup_{Q^+} u_{\varepsilon}^{-1}} \leq C C' \inf_{Q^+} u_{\varepsilon}. \]
Letting $\varepsilon \to 0$ on both sides finishes the proof.
\end{proof}

\begin{corollary} \label{cor:global PHI}
Suppose {\em Assumptions \ref{as:0} - \ref{as:VD+PI}} are satisfied globally on $Y=X$. If $C_3=C_5=0$, then the family $(\e_t,\F)$ satisfies the parabolic Harnack 
inequality {\em PHI($\Psi$)} on $X$. That is, there is a constant $C_{\mbox{\em \tiny{PHI}}}$ such that for any $a \in \R$ and any ball $B(x,4r) \subsetneq X$, any non-negative 
local weak solution $u$ of the heat equation for $L_t$ in $Q=Q(x,a,r)$, we have
\[ \sup_{Q^-} u \leq C_{\mbox{\em \tiny{PHI}}} \inf_{Q^+} u. \]
The constant $C_{\mbox{\em \tiny{PHI}}}$ depends only on $\delta$, $\tau_1, \tau_2, \tau_3, \tau_4$, $C_{\Psi}$, $\beta_1$, $\beta_2$, $C_{\mbox{\em \tiny{VD}}}$, $C_{\mbox{\em \tiny{PI}}}$, $C_0$, $C_{10}$, $C_{11}$, $C_2$, $C_4$.
\end{corollary}

\begin{corollary} \label{cor:Hoelder} 
Suppose {\em Assumptions \ref{as:0} - \ref{as:VD+PI}} are satisfied and each $\e_t$ is left-strongly local. Let $\delta \in (0,1)$.
Then there exist $\alpha \in (0,1)$ and $C \in (0,\infty)$ such that for any $a \in \R$, any ball $B(x,4r) \subsetneq B(x,8r) \subset Y$ with $0 < r < \frac{R_0}{4}$, 
any local weak solution $u$ of the heat equation for $L_t$ in $Q=Q(x,a,r)$ has a continuous version 
which satisfies
 \[ \sup_{(t,y),(t',y')\in Q'} \left\{ \frac{ |u(t,y) - u(t',y')| }{ [ \Psi^{-1}(|t-t'|) + d(y,y')]^{\alpha} } \right \}
\leq \frac{C}{r^{\alpha} } \sup_{Q} |u| \]
where $Q'= (a+\Psi((1-\delta)r),a+\Psi(r)) \times \delta B$.
The constant $C$ depends only on $\delta$, $C_{\Psi}$, $\beta_1$, $\beta_2$, $C_{\mbox{\em \tiny{VD}}}$, $C_{\mbox{\em \tiny{PI}}}$, $C_0$, $C_{10}$, $C_{11}$, and an upper bound on $[ (1 + C_2 + C_4) +(C_3+C_5) \Psi((1-\delta)r)]$.
\end{corollary}

\begin{proof}
The proof is standard. For instance, the reasoning in \cite[Proof of Theorem 5.4.7]{SC02} applies with only minor changes such as replacing $r^2$ by $\Psi(r)$. The left-strong locality is assumed because then constant functions are local weak solutions to the heat equation, a fact that is used in this proof.
\end{proof}

\subsection{Characterization of the parabolic Harnack inequality in the symmetric strongly local case} \label{ssec:PHI symmetric case}

It is known from the works of Grigor'yan \cite{Gri91} and Saloff-Coste \cite{SC92} that on complete Riemannian manifolds, the parabolic Harnack inequality is characterized by the volume doubling condition
together with the Poincar\'e inequality, as well as by two-sided Gaussian heat kernel bounds. For related results on fractal-type metric measure spaces with a symmetric strongly local regular Dirichlet form see, e.g., \cite{BBK06, GT12, BGK12} and references therein.

The parabolic Harnack inequality PHI($\Psi$) stated above is slightly different from the Harnack inequalities w-PHI($\Psi$) or s-PHI($\Psi$) introduced in \cite{BGK12} because, in 
defining $Q, Q^-, Q^+$, we used $\tau_i \Psi(r)$ rather than $\Psi(\tau_i r)$. Our choice is in accordance with the parabolic Harnack inequality stated in \cite{HSC01}.
In order to clarify the relation between PHI($\Psi$) and w-PHI($\Psi$), let us define time-space cylinders $\hat Q$ as follows. For $0 < \sigma_1 < \sigma_2 < \sigma_3 < \sigma_4 < 1$, set 
\begin{align*}
\hat Q &= \hat Q(x,a,r) = (a,a + \Psi(r)) \times B, \\
 \hat Q^- &= (a + \Psi(\sigma_1 r), a + \Psi(\sigma_2 r)) \times \delta B, \\ 
 \hat Q^+ &= (a + \Psi(\sigma_3 r), a + \Psi(\sigma_4 r)) \times \delta B.
\end{align*}
Let $\F'$ be the dual space of $\F$.
\begin{definition}
$(\e^*,\F)$ satisfies the {\em weak parabolic Harnack inequality} w-PHI($\Psi$) on $X$ (for local weak solutions) if there is a constant
 $C \in (0,\infty)$ such that for any $a \in \R$, any ball $B(x,r) \subset X$, and any {\em bounded} local weak solution $u$
of the heat equation for $L_t$ in $\hat Q = \hat Q(x,a,r)$, it holds
\[ \sup_{\hat Q^-} u \leq C \inf_{\hat Q^+} u. \]
\end{definition}

\begin{remark}
 In fact, \cite{BGK12} introduced the condition w-PHI($\Psi$) for a space of so-called caloric functions. We show in Proposition \ref{prop:very weak is caloric} below that
local weak solutions have all the properties that define a space of caloric functions.
\end{remark}

\begin{proposition} \label{prop:VD,PI,CSA equiv}
Let $(X,d,\mu,\e^*,\F)$ be a symmetric strongly local regular Dirichlet space. Assume that all metric balls in $(X,d)$ are precompact and {\em VD} is satisfied. Let $\Psi$ be as in \eqref{eq:beta} and 
consider
\begin{enumerate}
\item
$(\e^*,\F)$ satisfies {\em PI($\Psi$)}, and {\em CSA($\Psi$)} on $X$,
\item
$(\e^*,\F)$ satisfies {\em PHI($\Psi$)} on $X$,
\item
$(\e^*,\F)$ satisfies {\em w-PHI($\Psi$)} on $X$ (for local weak solutions),
\item
$(\e^*,\F)$ satisfies {\em weak-PI($\Psi$)}, and {\em CSA($\Psi$)} on $X$.
\end{enumerate}
The following implications hold:
\[ (i) \Rightarrow (ii) \Rightarrow (iii) \Rightarrow (iv). \]
If, in addition, $d$ is geodesic, then $ (iv) \Rightarrow (i)$.
\end{proposition}

\begin{proof}
The implication (i) to (ii) is the content of Corollary \ref{cor:global PHI}. 
To verify the implication (ii) to (iii), it suffices to find parameters $\tau_i$ and $\sigma_i$ such that
 $\hat Q^- \subset Q^-$ and $\hat Q^+ \subset Q^+$. By \eqref{eq:beta}, we have
\begin{align*}
\frac{\tau_4 \Psi(r)}{\Psi(\sigma_4 r)} \geq C_{\Psi}^{-1} \tau_4 \sigma_4^{-\beta_1},
\end{align*} 
for any $\tau_4,\sigma_4 \in (0,1)$. 
We pick $\tau_4$ and $\sigma_4$ such that the right hand side is greater than $1$. 
Applying \eqref{eq:beta} once again, we get
\begin{align*}
\frac{\tau_3 \Psi(r)}{\Psi(\sigma_3 r)} \leq C_{\Psi} \tau_3 \sigma_3^{-\beta_2},
\end{align*} 
for any $\tau_3,\sigma_3 \in (0,1)$. 
We pick $\tau_3 < \tau_4$ and $\sigma_3< \sigma_4$ such that the right hand side is less than $1$. Then $\hat Q^+ \subset Q^+$. 
Similarly, we find $0 < \tau_1 < \tau_2 < \tau_3$ and $0 < \sigma_1 < \sigma_2 < \sigma_3$ such that $\hat Q^- \subset  Q^-$. 

Under VD, condition w-PHI($\Psi$) is equivalent to weak heat kernel estimates (w-HKE($\Psi$) and w-LLE($\Psi$)) by \cite[Theorem 3.1]{BGK12}. 
Under VD, these heat kernel estimates imply the weak Poincar\'e inequality weak-PI($\Psi$) and CSA($\Psi$) by \cite[Theorem 2.12]{LierlBHPf} except for the continuity of the cutoff functions which follows from the H\"older continuity of the Dirichlet heat kernel, which is a consequence of the parabolic Harnack inequality; see also \cite{AB15, BBK06}.
This proves that (iii) implies (iv).
For the implication (iv) $\Rightarrow$ (i) we refer to Remark \ref{rem:strong and weak PI}.
\end{proof}

\begin{definition}
The reverse volume doubling property (RVD) holds if there are constants $C_{\mbox{\tiny{RVD}}}$ and $\nu_0 \in [1,\infty)$ such that
\begin{align} \label{eq:rvd inequality}
\frac{\mu(B(x,R))}{\mu(B(y,s))} \geq C_{\mbox{\tiny{RVD}}} \left(\frac{R}{s} \right)^{\nu_0}
\end{align}
for any $0 < s \le R$, $x \in X$, $y \in B(x,R)$ with $X \setminus B(x,R) \neq \emptyset$.
\end{definition}

\begin{remark}
\begin{enumerate}
\item
Suppose in addition to the  hypotheses of Proposition \ref{prop:VD,PI,CSA equiv} that   RVD holds.
Then condition CSA($\Psi$) in (iv) can equivalently be replaced by the generalized capacity condition introduced in \cite{GHL15}. 
Moreover, under RVD, (iv) is equivalent to a weak upper bound and a 
weak near-diagonal lower bound for the heat kernel, see \cite[Theorem 1.2]{GHL15}. The weak heat kernel bounds imply (iii) by \cite[Theorem 3.1]{BGK12}.
\item
If the metric space $(X,d)$ is not geodesic then (iii) may fail to imply (ii). See \cite{BGK12} for a counterexample on a non-geodesic space.
\item
For the implication (iv) $\Rightarrow$ (i), the hypothesis that $(X,d)$ is geodesic could be replaced by a chaining condition. Then the strong Poincar\'e inequality can be derived from
the weak Poincar\'e inequality by a Whitney covering argument; see, e.g. \cite{SC02}.
\end{enumerate}
\end{remark}

{\em Conjecture:}
 The strong parabolic Harnack inequality PHI($\Psi$) implies the strong Poincar\'e inequality PI($\Psi$), that is, (ii) $\Leftrightarrow$ (i) in Proposition \ref{prop:VD,PI,CSA equiv}.

\section{Estimates for the heat propagator}   \label{sec:heat propagator}

Let $(\e_t,\F)$ be a family of bilinear forms that satisfies Assumptions \ref{as:0}, \ref{as:skew1}, \ref{as:skew2} globally on $Y=X$ with respect to the reference form $(\e^*,\F)$. Observe that the bilinear forms $\hat \e_t(f,g) := \e_t(g,f)$ satisfy the same assumptions. In addition, we suppose that 
Assumption \ref{as:VD+PI} is satisfied locally on $X$, that is, every point $x \in X$ has a neighborhood $Y_x = B(x,8r_x)$ where Assumption \ref{as:VD+PI} is satisfied with $Y=Y_x$ up to scale $R_0=4r_x$ and $B(x,4r_x) \subsetneq Y_x$. 
Recall that $\alpha$ and $c$ are positive constants introduced in Assumption \ref{as:0}(vi).

\begin{proposition} \label{prop:IVP}
Let $s < T \le +\infty$. 
For every $f \in L^2(X)$ there exists a unique weak solution $u$ to the heat equation for $L_t$ on $(s,T) \times X$ satisfying the initial condition $u(s,\cdot) = f$. 
More precisely, there exists a unique $u \in L^2( (s,T) \to \F) $ of the initial value problem
\begin{align} \begin{split} \label{eq:initial value problem}
& \int_s^{T} \left\langle \frac{\partial}{\partial t} u, \phi \right\rangle_{\F',\F} dt + \int_s^{T} \e_t(u,\phi) dt = 0,   \qquad \mbox{ for all } \phi \in L^2( (s,T) \to \F), \\
& \lim_{t \downarrow s} u(t,\cdot) = f \quad \mbox{ in } L^2(X). 
\end{split}
\end{align}
In particular, $u$ has a weak time-derivative $\frac{\partial}{\partial t} u \in L^2((s,T) \to \F')$ and $u \in C^0( [s,T] \to L^2(X))$.
\end{proposition}

\begin{proof}
The proof for the case when $\e_t$ is non-negative definite is given in \cite[Chap.~3, Theorem 4.1 and Remark 4.3]{LM68}. 
For the general case, it suffices to notice that $\e_t + \alpha \langle \cdot, \cdot \rangle$ is positive definite by Assumption \ref{as:0}(vi), and
$u$ is a solution to the initial value problem for $L_t$ if and only if $e^{-\alpha(t-s)} u$ is a solution to the initial value problem for $L_t - \alpha$.
\end{proof}

For $t>s$ we consider the \emph{transition operator} associated with $L_t - \frac{\partial}{\partial t}$,
\[ T^s_t:L^2(X) \to \F. \]
The transition operator assigns to every $f \in L^2(X)$ the function $u(t) = T^s_t f \in \F$,  where $u:t \mapsto T^s_t f$ is the unique solution of the initial value 
problem \eqref{eq:initial value problem} with $T = +\infty$ given by Proposition \ref{prop:T^s_t}.
We set $T^s_s f := \lim_{t \downarrow s} T^s_tf = f$. From Corollary \ref{cor:Hoelder}, we obtain that $(t,y) \mapsto T^s_t f(y)$ has a jointly continuous version which we will denote by $P^s_t f(y)$.  
The transition operators satisfy
\begin{align} \label{eq:semigroup property T^s_t}
T^r_t f = T^s_t \circ T^r_s f, \quad \forall r \leq s \leq t, \ f \in L^2(X).
\end{align}
This follows from the fact that both $t \mapsto T^r_t f$ and $t \mapsto T^s_t \circ T^r_s f$ are weak solutions of the heat equation on $(s,\infty) \times X$ and satisfy the initial 
condition $T^r_t f \big|_{t=s} = T^r_s f = T^s_t \circ T^r_s f \big|_{t=s}$, and because the weak solution to this initial value problem is unique by Proposition \ref{prop:IVP}. 
Moreover, applying Assumption \ref{as:0}(vi) it follows that 
\begin{align} \label{eq:alpha-contraction} 
\| T^s_t f \|_{L^2} \leq e^{(\alpha-c)(t-s)} \| f \|_{L^2}, \quad \forall f \in L^2(X).
\end{align}

\begin{proposition} \label{prop:T^s_t pos preserving}
The transition operators $T^s_t f$, $s \le t$, are positivity preserving. That is, if 
$f \in L^2(X)$, $f \ge 0$, then $T^s_t f \ge 0$.
\end{proposition}

\begin{proof} 
Since $e^{-\alpha(t-s)}T^s_t f$ is the transition operator for $L_t - \alpha$, and $e^{-\alpha(t-s)}T^s_t f \ge 0 $ if and only if $e^{-\alpha(t-s)}T^s_t f \ge 0$, and by Assumption \ref{as:0}(vi), it suffices to give the proof for the case when $\e$ is non-negative definite. 

Take $u = T^s_t f$ and $\phi = u-u^+$ in \eqref{eq:initial value problem}.
Let $u^+ := \max\{u,0\}$. By locality, $\e(u,u-u^+) \ge 0$. We also have $\left\langle  \frac{\partial}{\partial t} (u - u^+ ),u \right\rangle_{\F',\F} \le 0$. 
Therefore,
\begin{align*}
0 & \ge \int_s^{T} \left\langle \frac{\partial}{\partial t} u, u - u^+ \right\rangle_{\F',\F} dt  \\
& \ge \int_s^{T} \frac{\partial}{\partial t} \left\langle u, u - u^+ \right\rangle_{\F',\F} dt = \left\langle u(T), u(T) - u^+(T) \right\rangle_{\F',\F} - \left\langle u(s), u(s) - u^+(s) \right\rangle_{\F',\F}.
\end{align*}
Since $u(s) = f \ge 0$, we have $u(s)-u^+(s) = 0$ and therefore $\left\langle u(T), u(T) - u^+(T) \right\rangle_{\F',\F} \le 0$. Thus, $u(T) = u^+(T) \ge 0$. 
\end{proof}

Similarly, there exist transition operators $S^t_s$, $s \le t$, corresponding to the heat equation for the adjoints $\hat L_s$ of the time-reversed generators $L_s$. It is immediate from \eqref{eq:initial value problem} that $S^t_s$ is the adjoint of $T^s_t$.
 Let $Q^t_s f$ be the continuous version of $S^t_s f$ which exists by Corollary \ref{cor:Hoelder}.

\begin{proposition} \label{prop:T^s_t}
There exists a unique
integral kernel $p(t,y,s,x)$ with the following properties:
\begin{enumerate}
\item
$p(t,y,s,x)$ is non-negative and jointly continuous in $(t,y,x) \in (s,\infty) \times X \times X$.
 \item
For every fixed $s < t$ and $y \in X$, the maps $x \mapsto p(t,y,s,x)$ and $y \mapsto p(t,y,s,x)$ are in $L^2(X)$.
\item
For every $s< t$, all $x,y \in X$ and every $f \in L^2(X)$,
 \[ P^s_t f(y) = \int_X p(t,y,s,x) f(x) \mu(dx). \]
and
 \[ Q^t_s f(x) = \int_X p(t,y,s,x) f(y) \mu(dy). \]
 \item
There exists a constant $C \in (0,\infty)$ such that, for every $s < t$ and $x\in X$,
\[ p(t,x,s,x) 
\le e^{(\alpha-c)(t-s)} \frac{C}{V(x,\tau_x)}, \]
where $\tau_x = r_x \wedge \Psi^{-1}(2(t-s))$, and
and $C$ depends at most on $\beta_1$, $\beta_2$, $C_{\Psi}$, $C_0$, $C_{10}$, $C_{11}$, $C_{\mbox{\tiny \em VD}}$, $C_{\mbox{\tiny \em PI}}$, and on an upper bound on $(1+C_2+C_3\Psi(\tau_x))$. 
\item
For every $s<r<t$ and all $x,y \in X$,
 \[ p(t,y,s,x) = \int_X p(t,y,r,z) p(r,z,s,x) d\mu(z). \]
\item
For every $s < r$ and every fixed $x \in X$, the map $(t,y) \mapsto p(t,y,s,x)$ is a weak solution of the heat equation for $L_t$ in $(r,\infty) \times X$. 
\end{enumerate}
\end{proposition}

\begin{proof}
In the special case when $(\e_t,\F)$ is a time-independent symmetric strongly local regular Dirichlet form, the proof is given in \cite[Section 4.3.3]{BGK12}.

Let $f \in L^2(X)$, $f \geq 0$, and let $s < t$. Then $(t-\frac{1}{2}\Psi(\tau_y), t+ \frac{1}{2}\Psi(\tau_y)) \subset (s,s+\Psi(r_y))$. 
By the mean value estimate of Theorem \ref{thm:MVE subsol p>2}, the joint continuity of $P^s_t f(y)$ in $(t,y)$, and by \eqref{eq:alpha-contraction}, we have
\begin{align} \begin{split} \label{eq:P^s_t estimate}
 \big[ P^s_t f(y) \big]^2 
& \le \frac{C}{\Psi(\tau_y) V(y,\tau_y)} \int_{t - \frac{1}{2}\Psi(\tau_y)}^{t + \frac{1}{2}\Psi(\tau_y)} \int_{B(y,\tau_y)} 
\big[ P^s_u f(z) \big]^2 d\mu(z) du \\
& \le e^{(\alpha-c)(t - s)} \frac{C}{V(y,\tau_y)} \| f \|_2^2,
\end{split}
\end{align}
for some constant $C \in (0,\infty)$ that depends on $y$ only through an upper bound on $C_3(\Psi(\tau_y))$. 
Considering $f^+$ and $f^-$, the displayed inequality extends to all $f \in L^2$. This shows that $f \mapsto P^s_tf(y)$ is a bounded linear functional.
By the Riesz representation theorem, there exists a unique function $p^s_{t,y} \in L^2(X)$ such that, for every $y \in X$,
\begin{align} \label{eq:P^s_t kernel}
P^s_t f(y) = \int p^s_{t,y} (x) f(x) d\mu(x), \quad \mbox{ for all } f \in L^2(X),
\end{align}
and
\begin{align} \label{eq:norm estimate of p^s_t,y}
 \| p^s_{t,y} \|_2^2 
\leq \frac{C e^{(\alpha-c)(t-s)} }{V(y,\tau_y)}.
\end{align}
By similar arguments, we obtain that there exists a function $q^t_{s,x} \in L^2(X)$ such that, 
\begin{align} \label{eq:Q^t_s kernel}
Q^t_s f(x) = \int q^t_{s,x} f(y) d\mu(y), \quad \mbox{ for all } f \in L^2(X),
\end{align}
and
\begin{align} \label{eq:norm estimate of q^t_s,x}
 \| q^t_{s,x} \|_2^2 
\leq \frac{C e^{(\alpha-c)(t-s)} }{V(x,\tau_x)}.
\end{align}
Since $Q^t_s$ is the adjoint of $P^s_t$, we have $p^s_{t,y}(x) = q^t_{s,x}(y)$ for almost every $x,y \in X$. 
We define
\begin{align} \label{eq:construction of propagator}
 p(t,y,s,x) := \int p^{r}_{t,y}(z) q^r_{s,x}(z) d\mu(z).
\end{align}
for some $r \in (s,t)$.
Then
\begin{align*}  
 p(t,y,s,x)= \int p^{r}_{t,y}(z) p^s_{r,z}(x) d\mu(z) \mbox{ for a.e. } x \in X.
\end{align*}
Proposition \ref{prop:T^s_t pos preserving} together with \eqref{eq:P^s_t kernel} and \eqref{eq:Q^t_s kernel} implies that $p^r_{t,y}$ and $q^r_{t,x}$ are non-negative almost everywhere, hence $p(t,y,s,x)$ is non-negative for all $x,y \in X$. 
Applying \eqref{eq:semigroup property T^s_t}, we get for any $f \in L^2(X)$,
\begin{align*}
P^s_t f(y) 
& = P^r_t \circ P^s_r f(y) = \int p^r_{t,y}(x) P^s_r f(x) d\mu(x) \\
& = \int Q^r_s p^r_{t,y}(x) f(x) d\mu(x) \\
& = \int \int q^r_{s,x}(z) p^r_{t,y}(z) d\mu(z) f(x) d\mu(x) \\
& = \int p(t,y,s,x) f(x) d\mu(x).
\end{align*}
Similary, we obtain $Q^t_s f(x) = \int p(t,y,s,x) f(y) d\mu(y)$.
Combining with \eqref{eq:P^s_t kernel} and \eqref{eq:Q^t_s kernel}, we see that 
$p(t,y,s,\cdot) = p^s_{t,y} \in L^2(X)$ and $p(t,\cdot,s,x) = q^t_{s,x} \in L^2(X,\mu)$.

From a computation similar to the one above, we see that $p(t,y,s,x)$ is in fact independent of the choice of $r$, and the semigroup property (v) holds.

The upper bound (iv) follows from \eqref{eq:construction of propagator}, the Cauchy-Schwarz inequality, as well as \eqref{eq:norm estimate of p^s_t,y} and \eqref{eq:norm estimate of q^t_s,x}. 

Since $p(r,\cdot,s,x)$ is in $L^2(X)$ when $s<r$, the semigroup property implies that
$p(t,y,s,x) = P^r_t p(r, y, s,x)$ for almost every $x \in X$. Since $(t,y) \mapsto P^r_t p(r, y, s,x)$ a weak solution on $(r,\infty) \times X$, we have proved (vi). 

It remains to show the joint continuity. It suffices to show that $p(t,y,s,x)$ is continuous in $x$ locally uniformly in $(t,y)$. Let $f \in L^2(X)$. We apply Corollary \ref{cor:Hoelder} to the weak solution $P^s_t f$ for $L_t$ in 
$Q = Q(x,t,\tau_x) = (t-\frac{1}{2}\Psi(\tau_x),t+\frac{1}{2}\Psi(\tau_x)) \times B(x,\tau_x)$,
Then,
\begin{align*}
| P^s_{t} f(x') - P^s_{t} f(x) |
& \le C \left( \frac{d(x,x')}{\tau_x} \right)^{\alpha} \sup_{(a,z) \in Q} |P^s_a f(z)|.
\end{align*}
By \eqref{eq:P^s_t estimate},
\begin{align*}
\sup_{(a,z) \in Q} |P^s_a f(z)|
\le e^{(\alpha - c) (t-s)}  \frac{C'}{V(x,\tau_x)^{1/2}} \| f \|_2. 
\end{align*}
Here, $C$ is a positive constant that may change from line to line.
Now we set $f = p^s_{r,y}$ where $r = \frac{s+t}{2}$. Then $P^r_t f = p(t,y,s,\cdot)$ 
and $\| f \|_2^2 \le  \frac{C e^{(\alpha-c)(r-s)}}{V(y,\tau_y)}$ by \eqref{eq:norm estimate of p^s_t,y}. 
Hence,
\begin{align*}
| P^s_{t} f(x') - P^s_{t} f(x) |
& \le C \left( \frac{d(x,x')}{\tau_x} \right)^{\alpha} e^{(\alpha - c) 2(t-s)} \frac{1}{V(x,\tau_x)^{1/2}} \frac{1}{V(y,\tau_y)^{1/2}}.
\end{align*}
This shows that $p(t,y,s,x)$ is continuous in $x$ locally uniformly in $(t,y)$, and completes the proof of the joint continuity. 
\end{proof}

The next lemma is immediate from CSA and \eqref{eq:chain rule for Gamma}, \eqref{eq:Gamma(fg)}.
\begin{lemma} \label{lem:e^M psi CSA}
Let $\psi \in \mbox{\em CSA($\Psi,\epsilon,C_0$)}$ be a cutoff function for $B(x,R)$ in $B(x,R+r)$. Let $\phi = e^{M \psi}$ for some constant $M \in \R$. Let $A = B(x,R+r) \setminus B(x,R)$. Then
\begin{align*}
\int f^2 d\Gamma( \phi,\phi)
& \le \frac{2\epsilon}{1-2\epsilon} M^2 \int_A \phi^2 d\Gamma(f,f) + \frac{C_0 \epsilon^{1-\beta_2/2}}{(1-2\epsilon)\Psi(r)} M^2 \int_A \phi^2 f^2 d\mu.
\end{align*}
\end{lemma}

\begin{assumption} \label{as:skew davies} 
There are constants $C_6, C_7, C_{11} \in [0,\infty)$ such that for all $t \in \R$, for any $\epsilon \in (0,1)$, any $0 < r < R \leq R_0$, any ball 
$B(x,2R) \subset Y$, any $M \ge 1$, any cutoff function $\psi \in \mbox{CSA}(\Psi,\epsilon,C_0)$ for $B(x,R)$ in $B(x,R+r)$, and any 
$0 \leq f \in \F_{\mbox{\tiny{loc}}}(Y) \cap L^{\infty}_{\mbox{\tiny{loc}}}(Y,\mu)$,
\begin{align*}
& \quad |\e_t^{\mbox{\tiny{sym}}}(f^2 \phi^2,1)| + \left| \e_t^{\mbox{\tiny{skew}}}(f,f \phi^2) \right| \\
& \leq C_{11} \epsilon^{1/2} M \int \phi^2 d\Gamma(f,f) + (C_6 + C_7  \Psi( r )) \frac{C_1(\epsilon)}{ \Psi( r )} M \int f^2 \phi^2 d\mu,
\end{align*}
where $B=B(x,R+r)$, $\phi = e^{-M\psi}$. 
\end{assumption}

We set $p(t,y,s,x):= \delta_x(y)$ whenever $t \le s$.
Let 
\begin{align*}
 \Phi_{\beta_2}(R,t) := \sup_{r>0} \left\{ \frac{R}{r} - \frac{t}{r^{\beta_2}} \frac{R^{\beta_2}}{\Psi(R)}\right\}.
 \end{align*}

\begin{lemma} \label{lem:Davies-Gaffney}
Let $x, y \in X$. Suppose Assumption \ref{as:skew davies} is satisfied and {\em CSA($\Psi,C_0$) holds locally on $B(x,d(x,y))$ up to scale $\frac{1}{2}d(x,y)$}. Let $f_1 \in L^2(X)$ with support in $B(x,d(x,y)/4)$, and let $f_2 \in L^2(X)$ with support in 
$B(y,d(x,y)/4)$. Then there is a constant $C' \in (0,\infty)$ such that, for any $s < t$,
\[ \int T^s_t f_1(x) f_2(x) d\mu(x) \leq \| f_1 \|_{L^2} \| f_2 \|_{L^2} \exp \left( -  \Phi_{\beta_2}(d(x,y),C'(t-s)) + (\alpha-c)(t-s) \right). \]
The constant $C'$ depends at most on $C_{\Psi}$, $\beta_1$, $\beta_2$, $C_0$, $C_{10}$, $C_{11}$, and on an upper bound on $(C_6 + C_7\Psi(d(x,y)))$.
\end{lemma}

\begin{proof}
Set $R = d(y,x)$.  Let $\psi \in \mbox{CSA}(\Psi,\epsilon,C_0)$ be a cutoff function for $B(x, \frac{1}{4}R)$ in $B(x,\frac{3}{4}R)$. Let $\phi = e^{-M \psi}$ for some $M \ge 1$ that we will choose later. Let $u = T^s_t f_x$.
Following \cite[Theorem 2]{Dav92}, we get
\begin{align*}
\frac{1}{2} \frac{\partial}{\partial t} \| \phi u \|^2_2
& = - \e_t(u,u\phi^2)  \\
& \le 
- \int \phi^2 d\Gamma_t(u,u) + \frac{1}{2} \int \phi^2 d\Gamma_t(u,u) + 2 \int u^2 d\Gamma_t(\phi,\phi) \\
& \quad - \e_t^{\mbox{\tiny{sym}}}(u^2 \phi^2,1) -  \e_t^{\mbox{\tiny{skew}}}(u,u \phi^2) \\
& \le  
\left( -1 + \frac{1}{2} + 2 \frac{2 \epsilon}{1-2\epsilon} M^2 + C_{11} \epsilon^{1/2} M \right) \int \phi^2 d\Gamma_t(u,u) \\
& \quad + \left( \frac{M^2 }{1-2\epsilon} + (C_6 + C_7\Psi(R)) \epsilon^{-1/2} M \right) \frac{C \epsilon^{1-\beta_2/2}}{\Psi(R)} \| \phi u \|^2_2,
\end{align*}
by Lemma \ref{lem:e^M psi CSA} and Assumption \ref{as:skew davies}. Here, $C$ is a positive constants that depends at most on $C_{\Psi}$, $\beta_1$, $\beta_2$, $C_0$, $C_{10}$, $C_{11}$, and on an upper bound on $(C_6 + C_7\Psi(R))$.
Choosing $\epsilon = \hat c/M^2$ for some small enough $\hat c =\hat c(C_{11})$, we get
\begin{align*}
\| \phi T^s_t f_1 \|_2 
& \le \exp \left( \frac{C' M^{\beta_2}}{\Psi(R)}  (t-s) \right) \| \phi f_1 \|_2.
\end{align*}

If $(t-s) \ge \Psi(R)$, then $\Phi_{\beta_2}(R,C'(t-s))$ is bounded from above. In this case 
the desired estimate follows by the 
Cauchy-Schwarz inequality and \eqref{eq:alpha-contraction}. Indeed,
\[ \int T^s_t f_1(x) f_2(x) d\mu(x) \leq \| T^s_t f_1 \|_{L^2} \| f_2 \|_{L^2} \leq e^{(\alpha-c)(t-s)} \| f_1 \|_{L^2} \| f_2 \|_{L^2}. \]
Similarly, if the supremum (in the definition of) $\Phi_{\beta_2}(R,C'(t-s))$ is attained at some $r > R$, then
$\Phi_{\beta_2}(R,C'(t-s)) \le \frac{R}{r} < 1$, and the assertion follows.

It remains to consider the case when $(t-s) < \Psi(R)$ and the supremum $\Phi_{\beta_2}(R,C'(t-s))$ is attained at some $r \le R$. Then we choose $M := \frac{R}{r} \ge 1$. We get 
\begin{align*}
M - \frac{C' M^{\beta_2}}{\Psi(R)}(t-s) 
= \frac{R}{r} - \frac{C' (t-s) R^{\beta_2}}{r^{\beta_2} \Psi(R)}
 = \Phi_{\beta_2}(R,C'(t-s)).
\end{align*}
Hence,
\begin{align*}
& \quad \int T^s_t f_1(x) f_2(x) d\mu(x)  \\
& \le  \| \phi T^s_t f_1 \|_{L^2} \| \phi^{-1} f_2 \|_{L^2} \\
& \le  \exp \left( \frac{C' M^{\beta_2}}{\Psi(R)}  (t-s)   \right)  \| \phi f_1 \|_{L^2} \| \phi^{-1} f_2 \|_{L^2} \\
& \le  \exp \left( \frac{C' M^{\beta_2}}{\Psi(R)}  (t-s)   \right)  \left( \sup_{B(x,R/4)} \phi \right) \left( \sup_{B(y,R/4)} \phi^{-1} \right)  \| f_1 \|_{L^2} \| f_2 \|_{L^2} \\
& \le  \exp \left( \frac{C' M^{\beta_2}}{\Psi(R)}  (t-s) -M  \right)   \| f_1 \|_{L^2} \| f_2 \|_{L^2} \\
& \le  \exp \left( - \Phi_{\beta_2}(R,C'(t-s)) \right) \| f_1 \|_{L^2} \| f_2 \|_{L^2}.
\end{align*}
\end{proof}

By Theorem \ref{thm:MVE subsol 0<p<2}, there exists a constant $C \in (0,\infty)$ such that the following $L^1$-mean value estimate holds for any $0 < r \leq r_y$ and any  non-negative local very weak subsolution $u$ of the heat equation for $L_t$ in $(t-\frac{1}{2}\Psi(r), t+ \frac{1}{2}\Psi(r)) \times B(y,r)$, 
\begin{align} \label{eq:L1 MVE}
 u(t,y) 
& \leq  \frac{C}{\Psi(r) \mu(B(y,r))}  
 \int_{t-\frac{1}{2}\Psi(r)}^{t+\frac{1}{2}\Psi(r)} \int_{B(y,r)} u \, d\mu \, dt,
\end{align}
where $C$ depends on $C_{\Psi}$, $\beta_1$, $\beta_2$, $C_0$, $C_{10}$, $C_{11}$, $C_{\mbox{\tiny{PI}}}$, $C_{\mbox{\tiny{VD}}}$,
and on an upper bound on $(1+C_2+C_3\Psi(r))$.
Here, on the left hand side, we used the jointly continuous version of $u$ that exists by Corollary \ref{cor:Hoelder}.

\begin{theorem} \label{thm:upper HKE}
Suppose Assumptions \ref{as:0}, \ref{as:skew1}, \ref{as:skew2}, \ref{as:skew davies} are satisfied globally on $X$, and Assumption \ref{as:VD+PI} is satisfied locally on $X$.
Let $x,y \in X$.  Suppose {\em CSA($\Psi,C_0$) holds locally on $B(x,d(x,y))$ and on $B(y,d(x,y))$ up to scale $\frac{1}{2}d(x,y)$}.
Then there exist constants $C, C' \in (0,\infty)$ such that, for all $s<t$,
\begin{align*}
p(t,y,s,x)
\leq  C \frac{ \exp\left(- \Phi_{\beta_2}(d(x,y),C'(t-s)) + (\alpha-c) (t-s) \right)}
{V(x,\tau_x)^{\frac{1}{2}} V(y,\tau_y)^{\frac{1}{2}}},
\end{align*}
where $\tau_x = \Psi^{-1}(\frac{t-s}{2}) \wedge r_x$, $\tau_y = \Psi^{-1}(\frac{t-s}{2}) \wedge r_y$. 
The constants $C,C'$ depend only on $C_{\Psi}$, $\beta_1$, $\beta_2$, $C_0$, $C_{10}$, $C_{11}$,  $C_{\mbox{\em \tiny{VD}}}(Y)$, 
$C_{\mbox{\em \tiny{PI}}}(Y)$ for $Y=Y_x$ and for $Y=Y_y$, and on an upper bound on $(1+C_2 +C_6 + C_3 (\Psi(\tau_x) + \Psi(\tau_y)) + C_7 \Psi(d(x,y)))$.
\end{theorem}

\begin{proof}
Applying the $L^1$-mean value estimate \eqref{eq:L1 MVE} to $(t,y) \mapsto p(t,y,s,x)$ and to $(s,x) \mapsto p(t',y',s,x)$, we get
\begin{align*}
& \quad p(t,y,s,x) \\
& \le  \frac{C}{\Psi(\tau_y) V(y,\tau_y)} \int_{t - \frac{1}{2} \Psi(\tau_y)}^{t + \frac{1}{2} \Psi(\tau_y)}  \int_{B(y,\tau_y)} p(t',y',s,x) d\mu(y') dt' \\
& \le D \int_{t - \frac{1}{2} \Psi(\tau_y)}^{t + \frac{1}{2} \Psi(\tau_y)}  \int_{B(y,\tau_y)} 
\int_{s - \frac{1}{2} \Psi(\tau_x)}^{s + \frac{1}{2} \Psi(\tau_x)}  \int_{B(x,\tau_x)} p(t',y',s',x')  d\mu(x') ds' d\mu(y') dt',
\end{align*}
where $D= \frac{C^2}{\Psi(\tau_x) \Psi(\tau_y) V(x,\tau_x) V(y,\tau_y)}$.

In the case $\tau_x \vee \tau_y \leq d(x,y)/4$, Lemma \ref{lem:Davies-Gaffney} yields
\begin{align*}
& \quad \int_{B(y,\tau_y)} \int_{B(x,\tau_x)} p(t',y',s',x') d\mu(x') d\mu(y') \\
 & \leq V(x,\tau_x)^{1/2} V(y,\tau_y)^{1/2} \exp \left(- \Phi_{\beta_2}(d(x,y),C'(t-s) ) + (\alpha-c)(t-s) \right).
\end{align*}
In the case $\tau_x \vee \tau_y \geq d(x,y)/4$, $\Phi_{\beta_2}(d(x,y),C'(t-s))$ is bounded from above. By Cauchy-Schwarz inequality and \eqref{eq:alpha-contraction},
\begin{align*}
 \int_{B(y,\tau_y)} \int_{B(x,\tau_x)} p(t',y',s',x')  d\mu(x') d\mu(y') 
& = \int P^{s'}_{t'} 1_{B(x,\tau_x)}(y') 1_{B(y,\tau_y)}(y') d\mu(y') \\
& \leq \| T^{s'}_{t'} 1_{B(x,\tau_x)} \|_{2} \,  \| 1_{B(y,\tau_y)} \|_{2} \\
& \leq 
e^{(\alpha-c)(t'-s')} V(x,\tau_x)^{1/2} V(y,\tau_y)^{1/2}.
\end{align*}
In both cases, we obtain the desired estimate.
\end{proof}

\begin{definition}
For an open set $U \subset X$, the time-dependent Dirichlet-type forms on $U$ are defined by
\[ \e^D_{U,t}(f,g) := \e_t(f,g),  \quad f,g \in D(\e^D_U), \]
where, for each $t \in \R$, the domain $D(\e^D_{U,t}) := \F^0(U)$ is defined as the closure of $\F \cap \mathcal C_{\mbox{\tiny{c}}}(U)$ in $\F$ for the norm $\| \cdot \|_{\F}$.
Let $T^D_U(t,s)$, $t \geq s$, be the associated transition operators with integral kernel $p^D_U(t,y,s,x)$.
\end{definition}

\begin{proposition} \label{prop:set monotonicity of propagator}
Let $V \subset U \subset X$ be open subsets.
For any $t>s$, $x,y \in V$, 
\[ p^D_V(t,y,s,x) \le p^D_U(t,y,s,x). \]
\end{proposition}

\begin{proof}
We may assume that each $\e_t$ is non-negative definite (if not, multiply the kernels by $e^{-\alpha(t-s)}$ and notice that the associated bilinear forms $\e_t+\alpha$ are non-negative definite by Assumption \ref{as:0}(vi)).
Let $r \in (s,t)$.

Let $f(z) = p^D_U(r,z,s,x)$. Then $p^D_U(t,\cdot,s,x) = P^D_U(t,r) f$ is a non-negative local weak solution of the heat equation in $(r,\infty) \times V$. As $t \downarrow r$, $P^D_U(t,r) f  \to f$ in $L^2(U)$, and by non-negativity also in $L^2(V)$. 
Hence, by Corollary \ref{cor:super mvi}, 
\begin{align*}
p^D_U(t,y,s,x) \ge P^D_V(t,r) p^D_U(r,\cdot,s,x) = \int_V p^D_V(t,y,r,z) p^D_U(r,z,s,x) d\mu(z).
\end{align*}
Similarly, we have for $p^D_V(t,y,s,x) = Q^D_V(s,r) p^D_V(t,y,r,\cdot)(x)$ that 
\begin{align*}
p^D_V(t,y,s,x)  
\le Q^D_U(s,r) p^D_V(t,y,r,\cdot)(x) = \int_U p^D_U(r,z,s,x) p^D_V(t,y,r,z) d\mu(z).
\end{align*}
Combining both inequalities finishes the proof.
\end{proof}

\begin{theorem} \label{thm:basic p^D_B estimate} 
Suppose Assumptions \ref{as:0}, \ref{as:skew1}, \ref{as:skew2}, \ref{as:skew davies} are satisfied globally on $X$, and Assumption \ref{as:VD+PI} is satisfied locally on $X$.
Let $a \in X$ and $B=B(a,r_a)$.
\begin{enumerate}
\item
For any fixed $\epsilon \in (0,1)$ there are constants $c', C' \in (0,\infty)$, such that for any $x \in B(a,(1-\epsilon)r_a)$ and $0 < \epsilon (t-s) \leq \Psi(r_a)$, 
the Dirichlet heat propagator $p^D_B$ satisfies the near-diagonal lower bound
 \[  p^D_B(t,y,s,x) \geq \frac{c'}{V(x,\Psi^{-1}(t-s) \wedge R_x)},
 \]
for any $y \in B(a,(1-\epsilon)r_a)$ with $d(y,x) \leq \epsilon \Psi^{-1}(t-s)$, where $R_x = d(x,\partial B)$.
The constants $c',C'$ depend at most on $C_{\Psi}$, $\beta_1$, $\beta_2$, $C_0$, $C_{10}$, $C_{11}$, on $C_{\mbox{\em \tiny{VD}}}(Y_a)$ and $C_{\mbox{\em \tiny{PI}}}(Y_a)$
for $Y_a = B(a,8r_a)$, and on an upper bound on $(1+C_2+C_4+(C_3+C_5) \Psi(\tau_a))$. 
\item 
There exist constants $C, C' \in (0,\infty)$ such that for any $x,y \in B$, $t > s$, the Dirichlet heat propagator $p^D_B$ satisfies the upper bound
\begin{align*}
 p^D_B(t,y,s,x) 
\leq 
C \frac{ \exp\left(- \Phi_{\beta_2}(d(x,y),C'(t-s)) + (\alpha-c) (t-s) \right)}
{V(x,\tau_a)^{\frac{1}{2}} V(y,\tau_a)^{\frac{1}{2}}},
\end{align*}
where $\tau_a = \Psi^{-1}\left(\frac{t-s}{2} \right) \wedge r_a$.
\end{enumerate}
The constants $c',C,C'$ depend at most on $C_{\Psi}$, $\beta_1$, $\beta_2$, $C_0$, $C_{10}$, $C_{11}$, on $C_{\mbox{\em \tiny{VD}}}(Y_a)$ and $C_{\mbox{\em \tiny{PI}}}(Y_a)$
for $Y_a = B(a,8r_a)$, and on an upper bound on $(1+C_2+C_6+C_3\Psi(\tau_a) + C_7 \Psi(d(x,y)))$. 
\end{theorem}

\begin{proof} 
The on-diagonal estimate in (i) can be proved in the same way as in \cite[Theorem 5.6]{LierlSC2}. See also \cite[Theorem 5.4.10]{SC02}. 
For the near-diagonal estimate, apply the parabolic Harnack inequality of Theorem \ref{thm:PHI}. 

(ii) is immediate from Theorem \ref{thm:upper HKE} and the set monotonicity of the heat propagator proved in Proposition \ref{prop:set monotonicity of propagator}.
\end{proof}

If $(X,d)$ satisfies a chain condition as in \cite{GT12}, then we can apply the parabolic Harnack inequality repeatedly along chains to obtain an off-diagonal lower bound. In particular, if $d$ is geodesic, then the lower bound in Proposition \ref{thm:basic p^D_B estimate}(i) can be improved to the following corollary. By Proposition \ref{prop:set monotonicity of propagator}, we obtain the same lower bound for the global heat propagator $p(t,y,s,x)$. 

Let 
\begin{align*}
 \Phi(R,t) := \sup_{r>0} \left\{ \frac{R}{r} - \frac{t}{\Psi(r)} \right\}.
 \end{align*}

\begin{corollary}
Suppose $d$ is geodesic.
Then there are constants $C'', c', c'' \in (0,\infty)$ such that for any $a \in X$, all $x,y \in B(a,r_a/2)$, and $t>s$, the Dirichlet heat kernel on $B = B(a,r_a)$ satisfies the lower bound
\[ 
 p^D_B(t,y,s,x) \geq \frac{c'}{V(x,\Psi^{-1}\left(\frac{t-s}{2} \right) \wedge r_a)}
\exp\left(- C'' \Phi(d(x,y),c''(t-s))  \right), \]
The constants $c', c'', C''$ depend on $C_{\Psi}$, $\beta_1$, $\beta_2$, $C_0$, $C_{10}$, $C_{11}$, $C_2$, $C_3$, $C_4$, $C_5$, on 
$C_{\mbox{\em \tiny{VD}}}(Y)$ and $C_{\mbox{\em \tiny{PI}}}(Y)$ for $Y = B(a,8r_a)$, and on an upper bound on $(1+C_2+C_4+(C_3+C_5) \Psi(r_a))$.
\end{corollary}

\begin{proof}
From Theorem \ref{thm:basic p^D_B estimate}(i) we obtain an on-diagonal bound for $0 < \epsilon (t-s) < \Psi(r_a)$.
The off-diagonal estimate (for any $t>s$) follows from the parabolic Harnack inequality.
\end{proof}

\begin{corollary}
Suppose {\em Assumptions \ref{as:0}, \ref{as:skew1}, \ref{as:skew2}} and {\em A2-Y, VD, PI($\Psi$), CSA($\Psi$)} are satisfied globally on $Y=X$. Suppose $d$ is geodesic. If  $C_3=C_5=0$, then there are constants $C, C', c', c'', C'' \in (0,\infty)$ such that for any $x,y \in X$ and $t>s$, we have
\begin{align*}
c'  \frac{\exp\left( - C'' \Phi(d(x,y),c''(t-s)) \right)}{V(x,\Psi^{-1}(t-s))}
\leq p(t,y,s,x)
\leq  C \frac{\exp\left( - \Phi_{\beta_2}(d(x,y),C'(t-s)) + (\alpha-c)(t-s) \right)}{V(x,\Psi^{-1}(t-s))}.
\end{align*}
The constants $C, C', c', c'', C''$ depend only on $C_{\Psi}$, $\beta_1$, $\beta_2$, $C_0$, $C_{10}$, $C_{11}$, $C_2$, $C_4$, $C_{\mbox{\em \tiny{VD}}}(X)$, $C_{\mbox{\em \tiny{PI}}}(X)$.
\end{corollary}

\section{Parabolic maximum principle and caloric functions} \label{sec:parabolic max princ}

\begin{proposition}[Parabolic maximum principle] \label{prop:para max princ}
Suppose $(\e_t,\F)$, $t \in \R$, is a family of bilinear forms satisfying Assumption \ref{as:0}. Assume that $\e_t^{\mbox{\em\tiny{sym}}}(f,f) \ge 0$ for all $t \in \R$ and $f \in \F$.
Let $I = (s,T)$ for some $- \infty < s <  T \leq \infty$. Let $U \subset X$ be an open subset. Let $u \in \mathcal{C}_{\mbox{\em\tiny{loc}}}(I \to L^2(U))$ be a local very weak subsolution of the heat equation 
for $L_t$ in $I \times U$.
Assume that $u^+(t,\cdot) \in \F^0(U)$ for every $t \in I$, and $u^+(t,\cdot) \to 0$ in $L^2(U)$ as $t \to s$. 
Then $u\leq 0$ almost everywhere on $I \times U$.
\end{proposition}

For weak subsolutions of the heat equation for symmetric regular Dirichlet forms, the parabolic maximum principle is proved in \cite[Proposition 5.2]{GHL09} 
(see also \cite[Proposition 4.11]{GH08}). Their proof makes explicit use of the Markov property of the Dirichlet form. 
Below we give a proof of Proposition \ref{prop:para max princ} that relies on Steklov averages.

\begin{proof}[Proof of Proposition \ref{prop:para max princ}]
Let $u$ be as in the proposition. Then \eqref{eq:local very weak subsolution} extends to all $\phi \in \F^0(U)$ by an approximation argument together with the Cauchy-Schwarz inequality and 
Assumption \ref{as:0}. Thus, for any fixed $t$, we can take $\phi = (u^+)_h(t) \in \F^0(U)$ as test function in \eqref{eq:local very weak subsolution}. 
Let $s < a < b < T$ and $h>0$ be so small that $b+h < T$. 
Since $u_h$ has the strong time-derivative $\frac{\partial}{\partial t} (u^+)_h(t) = \frac{1}{h} [ u^+(t+h) - u^+(t) ]$, we have
\begin{align}
& \int_U (u^+)_h^2(b)  d\mu - \int_U (u^+)_h^2(a) d\mu \\
& = \int_a^b \frac{d}{dt} \int_U(u^+)_h^2(t)  d\mu \, dt \nonumber \\
& = 2 \int_a^b \frac{1}{h} \int_U  [ u^+(t+h) - u^+(t) ] (u^+)_h(t) d\mu \, dt  \nonumber \\
& = 2 \int_a^b \frac{1}{h} \int_U  [ u(t+h) - u(t) ] (u^+)_h(t) d\mu \, dt  \nonumber \\
& \quad - 2 \int_a^b \frac{1}{h} \int_U  [ u^-(t+h) - u^-(t) ] (u^+)_h(t) d\mu \, dt  \nonumber \\
& \le - 2 \int_a^b \frac{1}{h} \int_t^{t+h} \e_s \left( u(s),(u^+)_h(t) \right) ds \, dt  \nonumber \\
& \quad - 2 \int_a^b \frac{1}{h} \int_U  [ u^-(t+h) - u^-(t) ] (u^+)_h(t) d\mu \, dt  \nonumber \\
& \leq - 2 \int_a^b \frac{1}{h} \int_t^{t+h} \e_s \left( u(s),(u^+)_h(t) - u^+(t) \right) ds \, dt   \label{eq:goes to 0 with h part1}\\
& \quad - 2 \int_a^b \frac{1}{h} \int_t^{t+h} \e_s \left( u(s) - u(t), u^+(t) \right) ds \, dt   \label{eq:goes to 0 with h part2} \\
& \quad -2 \int_a^b \frac{1}{h} \int_t^{t+h} \e_s \left( u(t), u^+(t) \right) ds\, dt 
\label{eq:want to estimate integrand} \\
& \quad + \frac{2}{h} \int_a^b \int_U    u^-(t)  (u^+)_h(t) d\mu \, dt. 
\label{eq:goes to <0 with h part3}
\end{align}
Letting $h$ go to $0$, we see that \eqref{eq:goes to 0 with h part1} and \eqref{eq:goes to 0 with h part2} tend to $0$ by Assumption \ref{as:0} and \cite[Lemma 3.8 and Corollary 3.10]{LierlSC2}. 
In \eqref{eq:want to estimate integrand}, observe that $- \e_s \left( u(t), u^+(t) \right) = - \e_s \left( u^+(t), u^+(t) \right) \le 0$ because $\e_s$ is local and its symmetric part is non-negative definite. The integrand in \eqref{eq:goes to <0 with h part3} converges to $0$ pointwise almost everywhere. Hence \eqref{eq:goes to <0 with h part3} goes to $0$ by the dominated convergence theorem. 
Thus, we obtain
\begin{align*}
 \int_U (u^+)^2(b)  d\mu - \int_U (u^+)^2(a) d\mu 
\leq 0.
\end{align*}
for almost every $s < a < b < T$. 
The assumption that $u^+(t,\cdot) \to 0$ in $L^2(U)$ as $t \to s$ implies that
we can make $\int_U (u^+)^2(a) d\mu $ arbitrarily small by choosing $a$ sufficiently close to $s$. Hence,
\begin{align*}
 \int_U (u^+)^2(b)  d\mu
\leq 0,
\end{align*}
so $u^+(b) = 0$ $\mu$-almost everywhere on $U$, for almost every $b \in I$.
This proves that $u \leq 0$ almost everywhere on $I \times U$.
\end{proof}

\begin{corollary}[Super-mean value inequality] \label{cor:super mvi}
Suppose $(\e_t,\F)$, $t \in \R$, is a family of bilinear forms satisfying Assumption \ref{as:0}. Assume that $\e_t^{\mbox{\em\tiny{sym}}}(f,f) \ge 0$ for all $t \in \R$ and $f \in \F$.
Let $I = (s,T)$ for some $- \infty < s <  T \leq \infty$. Let $f \in L^2(U)$, $f \ge 0$. Let $u \in \mathcal{C}_{\mbox{\em\tiny{loc}}}(I \to L^2(U))$ be a non-negative local very weak supersolution of the heat equation for $L_t$ in $(s,T) \times U$ such that $u(t,\cdot) \to f$ in $L^2(U)$ as $t \downarrow s$. Then, for every $t \in (s,T)$,
\[ u(t,x) \ge P^D_U(t,s)f(x) \quad \mbox{ for a.e. } x \in U. \]
\end{corollary}

\begin{proof}
Following \cite[Corollary 2.3]{BGK12}, we apply the parabolic maximum principle to the local very weak subsolution 
$v(t,\cdot) = P^D_U(t,s)f - u(t,\cdot)$. 
Indeed, we have $v^+(t,\cdot) \in \F^0(U)$ for every $t \in I$ by Proposition \ref{prop:T^s_t pos preserving} and \cite[Lemma 4.4]{GH08}. Now Proposition \ref{prop:para max princ} yields that $v \le 0$ almost everywhere in $I \times U$. Continuity in $t$ completes the proof of the super-mean value inequality.
\end{proof}

The properties listed in the next Proposition are the defining properties of a space of caloric functions as defined in \cite{BGK12}.

\begin{proposition} \label{prop:very weak is caloric}
Suppose $(\e_t,\F)$, $t \in \R$, is a family of left-strongly local bilinear forms satisfying Assumption \ref{as:0}. Assume that $\e_t^{\mbox{\em\tiny{sym}}}(f,f) \ge 0$ for all $t \in \R$ and $f \in \F$.
Let $I =(s,T)$ for some $T \le \infty$. Let $U \subset X$ be open.
Let $\mathcal W(I\times U)$ be the space of local weak solutions of the heat equation for $L_t$ on $I \times U$. Then
\begin{enumerate}
\item 
$\mathcal W(I\times U)$ is a linear space over $\R$.
\item
If $I' \subset I$ and $U' \subset U$, then $\mathcal W(I\times U) \subset \mathcal W('I\times U')$.
\item
For any $f \in L^2(U)$, the function $(t,x) \mapsto P^D_{U,t}f(x)$ is in $\mathcal W(I\times U)$.
\item
Any constant function in $U$ is the restriction to $U$ of a time-independent function in $\mathcal W(I\times U)$.
\item For any non-negative $u \in \mathcal W(I\times U)$ and every $s < r < t < T$,
\[ u(t,x) \geq P^D_U(t,r) u(r,x) \quad \mbox{ for a.e. } x \in U. \]
\end{enumerate}
\end{proposition}

\begin{proof}
Properties (i) and (ii) are immediate from the definition of local weak solutions. Property (iii) is immediate from the definition of $P^D_U$. Property (iv) follows from the left-strong locality and the definition of local weak solutions. Property (v) follows from Corollary \ref{cor:super mvi}.
\end{proof}

\section{Construction of non-symmetric local forms}
\label{sec:examples}

In this section we show how to construct non-symmetric forms on a given symmetric strongly local regular Dirichlet space $(X,d,\mu,\e^*,\F)$.
Let $Y \subset X$ be an open subset and $R_0 >0$. Suppose Assumption \ref{as:VD+PI} is satisfied. 

\renewcommand{\H}{\mathcal H}
\begin{definition} Let $\H$ be the space of all functions $h \in \F \cap L^{\infty}(Y,
\mu)$, $h\geq 0$, for which there exists a constant $C_h' \in (0,\infty)$ such that
\begin{align} \label{eq:H2}
\forall \, 0 \le f \in \F_{\mbox{\tiny{c}}}(Y) \cap L^{\infty}(Y,\mu), \qquad  \left| \int d\Gamma(f,h) \right|
\leq C_h' \int f \, d\mu.
\end{align}  
\end{definition}

For instance, $\H$ contains any non-negative bounded function that is $\e^*$-harmonic on the subset $Y$. Also the ground state on a bounded domain containing $Y$ is in $\H$.

\begin{lemma}
Let $h \in \H$. Then there exists a constant $C_h \in (0,\infty)$ such that
\begin{align} \label{eq:H1}
\forall f \in \F_{\mbox{\em \tiny{c}}}(Y), \qquad  
\int f^2 d\Gamma(h,h)  \leq C_h \| f \|_{\F}^2.
\end{align}  
\end{lemma}

\begin{proof}
It suffices to show \eqref{eq:H1} when $f \in \F_{\mbox{\tiny{c}}}(Y)$ is bounded (otherwise approximate $f$ by bounded functions $f_n := (f \wedge n) \vee (-n)$). Then $f^2 h \in \F_{\mbox{\tiny{c}}}(Y) \cap L^{\infty}(Y,\mu)$. By \eqref{eq:chain rule for Gamma}, \eqref{eq:H2}, and \eqref{eq:CS},
\begin{align*}
\int f^2 d\Gamma(h,h)  
& = \int d\Gamma(f^2 h,h) - 2 \int fh \, d\Gamma(f,h)  \\
& \leq \frac{1}{2} \int f^2  d\Gamma(h,h)   + 8 \int h^2  d\Gamma(f,f) + C_h' \int f^2 h \, d\mu.
\end{align*}
Since, by assumption, $h$ is  bounded, \eqref{eq:H1} follows by rearranging the terms in the above inequality.
\end{proof}

\begin{proposition} \label{prop:nonsym form}
Let $h \in \H$.
For $f,g \in \F_{\mbox{\em \tiny{b}}}$, set  
\begin{align} \label{eq:nonsym form}
 \e(f,g) & := \e^*(f,g)  + \int g \, d\Gamma(f,h) - \int f \, d\Gamma(g,h).
 \end{align}
Then the results of Section \ref{sec:Moser}, Sections \ref{ssec:bombieri} - \ref{ssec:PHI main results}, and Section \ref{sec:heat propagator} apply to $(\e,\F)$, provided that $(\e^*,\F)$ satisfies Assumption \ref{as:VD+PI} as required in these results (locally or globally).
\end{proposition}

\begin{proof}
We point out that the bilinear form $\e$ is defined on $\F_{\mbox{\tiny{b}}} \times \F_{\mbox{ \tiny{b}}}$ and its skew-symmetric part may not satisfy the inequality in Assumption 0(i). Nevertheless, by locality the definition \eqref{eq:nonsym form} makes sense for any pair $(f,g)$ where $f \in \F_{\mbox{\tiny{loc}}}(Y)$ and $g = f \psi^2$ for some $\psi \in \F_{\mbox{\tiny{c}}} \cap L^{\infty}(Y,\mu)$. Moreover, if $(f_k) \subset \F \cap \mathcal C_{\mbox{\tiny{c}}}(X)$ converges to some $f \in \F$ in $(\F,\| \cdot \|_{\F})$ and quasi-everywhere, then one can easily verify that, for any positive integer $n$ and any $p \ge 2$,
\[ \lim_{k \to \infty} \e(f_k,f_k (f_k \wedge n)^{p-2} \psi^2) = \e(f,f (f \wedge n)^{p-2} \psi^2). \]
This is in fact sufficient to apply the argument of \eqref{eq:subsol p>2 estimate for e(f,f f^p)} and the paragraph thereafter, which is the only place where we have used Assumption 0(i) within Section \ref{sec:Moser} and Section \ref{ssec:bombieri} - \ref{ssec:PHI main results}. 

Below we will verify that $(\e,\F)$ satisfies Assumption 0(ii)--(vi), Assumption 1 and Assumption 2.

It is immediate from \eqref{eq:nonsym form} that $(\e,\F)$ is a local bilinear form.
The symmetric part of $\e$ is $\e^{\mbox{\tiny{sym}}} = \e^{\mbox{\tiny{s}}} = \e^*$. This follows easily from the definition of $\e^{\mbox{\tiny{s}}}$ and the strong locality of $(\e^*,\F)$. Thus, part (ii), (iii) and (vi) of Assumption \ref{as:0} are trivially satisfied.
Observe that $\l(f,g) =  \int g \, d\Gamma(f,h)$. Since $\Gamma$ obeys the  product rule and the chain rule, part (iv) and part (v) of Assumption \ref{as:0} are verified.

Next, we show that $(\e,\F)$ satisfies Assumption \ref{as:skew1}. The estimate on $\e^{\mbox{\tiny{sym}}}$ is trivially satisfied. 
 Let $\epsilon \in (0,1)$. Let  $0 < r < R \leq R_0$ and $B(x,2R)\subset Y$. Let $g \in \mbox{CSA($\Psi, \epsilon, C_0$)}$ be a cutoff function for $B(x,R)$ in $B=B(x,R+r)$. 
Let $0 \leq f \in \F_{\mbox{\tiny{loc}}}(Y) \cap L^{\infty}_{\mbox{\tiny{loc}}}(Y,\mu)$. 
By \eqref{eq:H2},
\begin{align*}
 \left|\e^{\mbox{\tiny{skew}}}(f^2g^2,1)\right|
 = 
 \left| \int  d\Gamma(f^2g^2,h) \right|
\le C_h' \int f^2 g^2 d\mu \le C_h' \int_B f^2 d\mu.
\end{align*}
Furthermore, we have
\begin{align*}
\e^{\mbox{\tiny{skew}}}(f,f g^2)
 = - 2 \int f^2 g \, d\Gamma(g,h).
\end{align*}
By \eqref{eq:CS}, \eqref{eq:H1}, \eqref{eq:Gamma(fg)}, and the cutoff Sobolev inequality \eqref{eq:CSA},
\begin{align*}
& \quad 2 \left| \int f^2 g \, d\Gamma(g,h) \right| \\
& \leq 
2 \left( \int f^2 d\Gamma(g,g) \right)^{1/2} 
       \left( \int (fg)^2 d\Gamma(h,h) \right)^{1/2} \\
& \leq 
2  C_h^{1/2}  \left( \int f^2 d\Gamma(g,g) \right)^{1/2} 
       \left( 2 \int g^2 d\Gamma(f,f) + 2 \int f^2 d\Gamma(g,g) + \int f^2 g^2 d\mu \right)^{1/2} \\
& \leq 
2  C_h^{1/2}  \left( \epsilon \int g^2 d\Gamma(f,f) + \frac{C_0(\epsilon)}{\Psi(r)} \int g f^2 d\mu \right)^{1/2}  \left( 2 \int g^2 d\Gamma(f,f) \right)^{1/2} \\
& \quad + 2  C_h^{1/2}  \left(2 \int f^2 d\Gamma(g,g) + \int f^2 g^2 d\mu \right) \\
& \leq 
2  C_h^{1/2}  \left[ \frac{1}{\epsilon^{1/2}} \left( \epsilon \int g^2 d\Gamma(f,f) + \frac{C_0(\epsilon)}{\Psi(r)} \int g f^2 d\mu \right)
+  2 \epsilon^{1/2} \int g^2 d\Gamma(f,f) \right] \\
& \quad + 2  C_h^{1/2}  \left( 2 \epsilon \int g^2 d\Gamma(f,f) + \left(2 \frac{C_0(\epsilon)}{\Psi(r)} + 1 \right) \int f^2 g d\mu \right) \\
& \leq 
C_{11} \epsilon^{1/2} \int g^2 d\Gamma(f,f) + (C_2+C_3\Psi(r)) \frac{ \epsilon^{-1/2} C_0(\epsilon)}{\Psi(r)} \int_B f^2 d\mu,
\end{align*}
for some constants $C_{11}$, $C_2$, $C_3$ depending only on $C_h$ and $C_0$.
This proves that $(\e,\F)$ satisfies Assumption \ref{as:skew1}.
Similarly, one can verify that Assumption \ref{as:skew davies} is satisfied.

Next, we show that $(\e,\F)$ satisfies Assumption \ref{as:skew2}. Let $g$ be as above and $0 \leq f \in \F_{\mbox{\tiny{loc}}}(Y)$ with $f + f^{-1} \in L^{\infty}_{\mbox{\tiny{loc}}}(Y)$. By \eqref{eq:chain rule for Gamma}, \eqref{eq:CS}, \eqref{eq:H1}, and by the cutoff Sobolev inequality \eqref{eq:CSA},
\begin{align*}
& \quad | \e^{\mbox{\tiny{skew}}}(f,f^{-1}g^2) |
 = \left| \int f^{-1} g^2 d\Gamma(f,h) - \int f d\Gamma(f^{-1} g^2,h) \right| \\
& = \left|- 2\int g d\Gamma(g,h) + 2 \int g^2 d\Gamma(\log f,h) \right| \\
& \leq  
2   \left( \int d\Gamma(g,g) \right)^{1/2}    \left( \int g^2 d\Gamma(h,h) \right)^{1/2} \\
& \quad + 2  \left( \int g^2 d\Gamma(\log f, \log f) \right)^{1/2} 
       \left( \int g^2 d\Gamma(h,h) \right)^{1/2} \\
 & \leq  
2 C_h^{1/2} \left[ \left(  \int d\Gamma(g,g) + \int g^2 d\mu  
\right) + \left( \int g^2 d\Gamma(\log f, \log f) \right)^{1/2}
       \left(  \int d\Gamma(g,g) +  \int g^2 d\mu \right)^{1/2} \right]\\
& \leq  
2  C_h^{1/2} \left[ \epsilon^{1/2} \int g^2 d\Gamma(\log f, \log f) 
+\left( \frac{C_0}{\Psi(r)} + 1 \right) (1 +\epsilon^{-1/2} ) \int g d\mu \right] \\
& \leq  
C_{11}  \epsilon^{1/2} \int g^2 d\Gamma(\log f, \log f) 
+\left( C_4 + C_5 \Psi(r) \right)  \frac{\epsilon^{-1/2} C_0}{\Psi(r)} \int_B d\mu,
\end{align*}
for some constants $C_{11}$, $C_4$, $C_5$ depending only on $C_h$ and $C_0$.
\end{proof}

It might be possible to weaken Assumption \ref{as:0} in such a way that it covers the example constructed above.
However, this issue concerns the (local) domains of the bilinear forms. We chose to keep Assumption \ref{as:0} as it is for the sake of the readability of the paper.

\bibliographystyle{siam}

\def\cprime{$'$} \def\cprime{$'$}

\noindent
Janna Lierl, 
Department of Mathematics, University of Connecticut, 341 Mansfield Road, Storrs, CT 06269, USA.  janna.lierl@uconn.edu.

\end{document}